\documentclass[hidelinks,onefignum,onetabnum]{siamart220329}

\usepackage{lipsum}
\usepackage{amsfonts}
\usepackage{graphicx}
\usepackage{epstopdf}
\usepackage{algorithmic}
\usepackage{tikz}
\usetikzlibrary{arrows,calc}
\ifpdf
  \DeclareGraphicsExtensions{.eps,.pdf,.png,.jpg}
\else
  \DeclareGraphicsExtensions{.eps}
\fi

 \usepackage{amsmath,amssymb}
 \usepackage{mathrsfs}
 \usepackage{dsfont} 
 \usepackage{mathtools} 
 \usepackage{esint} 
 \usepackage{xfrac} 
 \usepackage{defs} 
 \usepackage{fonts} 

 \usepackage[shortlabels]{enumitem} 

\usepackage[noadjust]{cite} 
\usepackage{cleveref}

\crefname{enumi}{Case}{Cases}
\Crefname{enumi}{Case}{Cases}

\newsiamremark{remark}{Remark}
\newsiamremark{hypothesis}{Hypothesis}
\crefname{hypothesis}{Hypothesis}{Hypotheses}
\newsiamthm{claim}{Claim}

\headers{Nonlocal transmission problems and their function spaces}
{Q. Du, Z. Han, T. Mengesha, J. M. Scott and X. Tian}

\title{On parameterized nonlocal-fractional transmission problems and associated function spaces 
\thanks{\textbf{Funding:} Z.~Han and X.~Tian were supported in part by NSF DMS-2240180 and the Alfred P. Sloan Fellowship. Q.~Du and J.~Scott  were supported in part by NSF DMS-2309245 and DMS-1937254. T. Mengesha is supported in part by NSF DMS-2206252.}
}

\author{Qiang Du\thanks{Applied Physics and Applied Mathematics, Columbia University, 500 W. 120th St, New York, NY 10027 (\email{qd2125@columbia.edu}).}
\and Zhaolong Han\thanks{Department of Mathematics, University of California San Diego, 9500 Gilman Drive, La Jolla, CA 92093 (\email{zhhan@ucsd.edu}, \email{xctian@ucsd.edu}).} 
\and Tadele Mengesha\thanks{Department of Mathematics, The University of Tennessee, 227 Ayres Hall, 1403 Circle Drive, Knoxville TN 37996 (\email{mengesha@utk.edu}).}
\and James M. Scott\thanks{Department of Mathematics and Statistics, Auburn University, 221 Parker Hall, Auburn, AL 36849 (\email{james.m.scott@auburn.edu}).}
\and Xiaochuan Tian\footnotemark[3]}

\usepackage{amsopn}
\DeclareMathOperator*{\argmin}{arg\,min}

\begin{document}

\maketitle

\begin{abstract}
In this paper, we consider a family of seamlessly coupled nonlocal models associated with transmission conditions across an interface. The models are derived from the variation of a parameterized family of energies consisting of a fractional type Dirichlet energy on one subdomain and a nonlocal Dirichlet energy involving a finite range of interactions on another subdomain. We present the rigorous mathematical formulation and its well-posedness. We also investigate the behavior of the model in various limiting regimes.
\end{abstract}



\section{Introduction}

 In this paper we present an analytical investigation of a transmission problem that seamlessly couples two 
 distinct models, with one  model based on the regional fractional Laplacian and another model employing a nonlocal operator characterized by a position-dependent interaction kernel. Both operators are nonlocal and act on functions that are defined within their respective spatial domains, with the coupling facilitated through a transmission condition across a hypersurface interface in between the domains. This presents a very unique setting of coupled models on non-overlapping spatial subdomains, which leads to many interesting modeling and analysis questions. 

\subsection{Motivation}\label{subsec:motivation}
Coupling of different models on different parts of a spatial domain has been an extensively studied subject of research with the aim of combining the computational efficiency of some models with the accuracy of others. Popular approaches, in the context of models of mechanics, include energy-based and force-based formulations. Coupled local and nonlocal models are analyzed in \cite{acosta2022local,acosta2023domain,JuanPablo-Ciarlet2025,DosSantos2021, capodaglio2020energy,Delia2015a,DLLT18,Garriz2020,KiMa10,kriventsov2015regularity,Seleson2013a,Seleson2015a,TaTiDu19}, focusing on various mathematical aspects such as well-posedness, qualitative properties of solutions, solution regularity, and asymptotic behavior as functions of the parameters. See the review paper \cite{DLST21} for additional references. Diffusion models that couple different fractional models have also attracted attention, see \cite{Gal2017,Garriz-Ignat2021}.

The mathematical formulation of the particular coupling problem under consideration in this work adopts the energy-based approach, providing a physically consistent analogue of the classical PDE-based local interface problem. While the nonlocal transmission problem studied in \cite{capodaglio2020energy} serves as a motivation for our work, their approach makes use of a volumetric-interface region to couple the models, in contrast to the hypersurface interface typically used in local interface problems.

To briefly describe the problem of interest, we let $\Omega \subset \bbR^d$ be a bounded  Lipschitz domain. Let $\Gamma \subset \bbR^d$ be a hypersurface that is locally the graph of a Lipschitz function, and such that $\Sigma := \Omega \cap \Gamma$ partitions $\Omega$ into two disjoint Lipschitz domains $\Omega_1$ and $\Omega_2$. Illustrations of possible configurations of such domains are shown in \Cref{fig:domains}.
Notice that $\Omega = \Omega_1 \cup \Omega_2 \cup \Sigma$.

\begin{figure}[htbp]
\vspace{-1cm}
  \centering
\begin{minipage}{0.46\textwidth}
	\begin{tikzpicture}[scale=0.50]
		\draw[thin] (0, 0) ellipse (5cm and 3cm) node at (3, 3){$\Omega$};
		
		\fill[blue!10] (-1.5, 0) circle (1.5cm);
		\node at (-1.5, 0) {$\Omega_1$};
		
		\node at (-3.2, -0.65) {$\Sigma$};
		
		
		
		\tikzset{highlight/.style={fill opacity=.4,red}}
		
		
		
		\clip (-6,-6)
		rectangle (6cm ,6 cm );
		\node [below right]at(3,-2.3){$\partial \Omega_2$};
		\node at (2, -1) {$\Omega_2$};
	\end{tikzpicture}
\end{minipage}\hspace{.3cm}
\begin{minipage}{0.46\textwidth}
	\begin{tikzpicture}[scale=0.70]
		\draw[thin] (0, 0) rectangle (7, 4);
		\node at (3, 4.4) {$\Omega$};
		\node at (1.5, 2.25) {$\Omega_1$};
		\node at (4.5, 2.75) {$\Omega_2$};
		\draw[line width=0.2mm] (2, 0) to[out=15, in=200] (4.75, 4);
		\node[left] at (3.4, 2) {$\Sigma$};
	\end{tikzpicture}
\end{minipage}
\vspace{-1cm}
\caption{Illustration of two domain configurations.}
\label{fig:domains}
\end{figure}
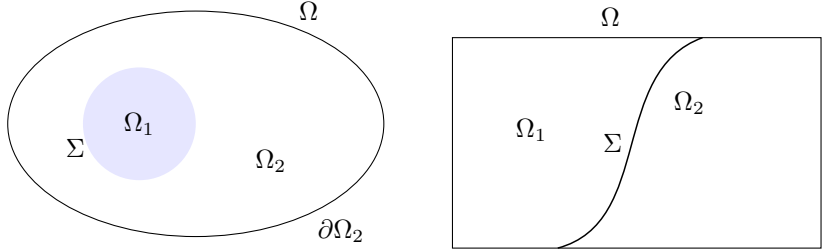


Given a function $f : \Omega \to \mathbb{R} $, we analyze the weak form of the transmission problem 
\begin{equation*}
	\begin{aligned}
		L^{s,\delta} u_1 &= f\quad \text{in $\Omega_1$},\\
		(-\Delta)_{\Omega_2}^s u_2  &= f\quad \text{in $\Omega_2$},\\
		\text{a specified transmission} &\text{ condition}  \text{ on $\Sigma$},\\
		\text{boundary } &\text{conditions }
		\text{on $\partial \Omega$},
	\end{aligned}
\end{equation*}
where $L^{s,\delta} $ is a specific nonlocal operator, $(-\Delta)_{\Omega_2}^s$ is the regional fractional Laplacian,  and the functions $u_1 : \Omega_1 \to \bbR$, $u_2 : \Omega_2 \to \bbR$ are regular enough for the operators to make sense. The transmission condition on $\Sigma$ refers to a continuity condition on $u(x) = u_1(x)\mathds{1}_{\Omega_1}(x) + u_2(x)\mathds{1}_{\Omega_2}$ in $\Omega$, in an appropriate sense, as well as a matching of suitably defined flux from $\Omega_1$ to $\Omega_2$ and vice verse. The nonlocal operator $L^{s,\delta}$ and the solution space of the transmission problem are defined in such a way that prescribing a transmission condition while using nonlocal operators is coherent and meaningful.
In addition to showing the existence of solutions to the transmission problem in a suitable space of functions, we will also study the behavior of the pair of solutions as a function of the parameter $s$  that determines the order of the differentiability and the parameter $\delta$ that measures the range of nonlocality.

\subsection{Statement of the transmission problem}
The specific class of transmission problems we study is derived as the first variation of a well-defined class of energies which will be detailed below. We take $1 < p < \infty$, $s \in (0,1)$ with $sp > 1$, $\delta\in (0, 1)$. Consider the following parameterized  energies 
\begin{equation}\label{energy}
	\begin{split}
		\cE_{s,\delta}(u_1,u_2) &:= \frac{\overline{C}_{d,p}}{p} \int_{\Omega_1} \int_{ B(\bx,\delta \sigma(\bx)) } \alpha(\bx)\frac{|u_1(\bx)-u_1(\by)|^p}{ \delta^{d+p} \sigma(\bx)^{d+sp} } \, \rmd \by \, \rmd \bx \\
		&\qquad + \frac{\kappa_{d,s,p}}{p} \int_{\Omega_2} \int_{\Omega_2}\beta(\bx) \frac{|u_2(\bx)-u_2(\by)|^p}{|\bx-\by|^{d+sp}} \, \rmd \by \, \rmd \bx.
	\end{split}
\end{equation}
where $\alpha$ and $\beta$ serve as coefficients and are bounded from below and above by positive constants. The normalizing constants $\overline{C}_{d,p}$ and   $\kappa_{d,s,p}$ will be specified later. The function $\sigma$ is the distance function from $\partial \Omega_1$, defined as
\[
\sigma(\bx) := \dist(\bx,\p \Omega_1).
\]
For any $\bx$ and $r>0$, $B(\bx, r)$ represents a ball centered at $\bx$ with radius $r.$ Note that the two terms of the energy are decoupled and model different processes in their respective domain. Indeed, for the first term of  $\cE_{s,\delta}(u_1,u_2)$, for $\bx\in \Omega_1$, the domain of interaction is $B(\bx, \delta\sigma(\bx))\subset\Omega_1$. Meanwhile, for the second term, the domain of interactions is given by the whole subdomain $\Omega_2$.  We minimize the energy $ \cE_{s,\delta}(u_1,u_2)$ over the class of pair of functions $(u_1, u_2)$ subject to a boundary condition on $\partial \Omega$
and 
a transmission condition on $\Sigma$ that introduces the coupling 
through a transmission condition across the interface $\Sigma$.
We begin by noting that for any admissible pair $(u_1, u_2)$, each of the terms of the energy $\cE_{s,\delta}(u_1,u_2)$ must be finite. As a consequence, since  $\alpha$ and $\beta$ have positive bounds,  $u_1$ belongs to the function space $\mathfrak{W}^{s,p}[\delta](\Omega_1)$ defined by 
{\small\[
\mathfrak{W}^{s,p}[\delta](\Omega_1) := \left\{ u \in L^p(\Omega_1) \, : [u]^{p}_{\mathfrak{W}^{s,p}[\delta](\Omega_1)} < \infty \right\}\]}where $$[u]^{p}_{\mathfrak{W}^{s,p}[\delta](\Omega_1)}:= \overline{C}_{d,p} \intdm{\Omega_1}{ \int_{B(\bx,\delta \sigma(\bx))}   \frac{\big| u(\bx)-u(\by) \big|^p}{ \delta^{d+p}  \sigma(\bx)^{d+sp} }  \, \rmd \by}{\bx} $$ serves as its seminorm,
and $u_2$ belongs to the classical fractional Sobolev space $W^{s, p}(\Omega_2)$,  given by 
\[
W^{s, p}(\Omega_2)=\left\{ u \in L^p(\Omega_2) \, : \,|u|^{p}_{W^{s,p}(\Omega_2)}<\infty\right\},
\]
where $$[u]^{p}_{W^{s,p}(\Omega_2)}:=\kappa_{d,s,p} \int_{\Omega_2} \int_{\Omega_2} \frac{|u(\bx)-u(\by)|^p}{|\bx-\by|^{d+sp}} \, \rmd \by \, \rmd \bx $$ serves as its seminorm. 

In addition, admissible pairs need to be regular enough to allow effective prescription of transmission and boundary conditions on $\Gamma$ and $\partial\Omega$ respectively. The latter can be facilitated by,  
as we will discuss in the sequel, choosing the interaction kernel comparable to $\mathds{1}_{B(\bx,\delta \sigma(\bx))}(\by) \frac{1}{\delta^{d+p}  \sigma(\bx)^{d+sp}}$ for the seminorm of $\mathfrak{W}^{s,p}[\delta](\Omega_1)$ and working with $s\in(0,1)$ such that $sp>1$.

The proper setup and analysis of the variational problem related to the minimization of the energy $\cE_{s,\delta}$ heavily rely on the functional and analytical properties of the functions in the spaces $\mathfrak{W}^{s,p}[\delta](\Omega_1)$ and $W^{s, p}(\Omega_2)$. While the structural properties of the fractional Sobolev space $W^{s, p}(\Omega_2)$ are well known, the space $\mathfrak{W}^{s,p}[\delta](\Omega_1)$ is nonstandard. The majority of the paper will be devoted to the careful study of this space.  As we will show in later sections the space $\mathfrak{W}^{s,p}[\delta](\Omega_1)$ is a Banach space equipped with the norm
$$
\Vnorm{u}_{\mathfrak{W}^{s,p}[\delta](\Omega_1)}^p := \Vnorm{u}_{L^p(\Omega_1)}^p + [u]_{\mathfrak{W}^{s,p}[\delta](\Omega_1)}^p\,.
$$
It is not difficult to see that functions in $\mathfrak{W}^{s,p}[\delta](\Omega_1)$ could be irregular in the interior of $\Omega_1$, required to be no better than being in $L^{p}_{loc}(\Omega_1)$. However, given the blowup of the effective interaction kernel $\mathds{1}_{B(\bx,\delta \sigma(\bx))}(\by) {1\over \delta^{d+p}  \sigma(\bx)^{d+sp}}$ as $\bx \to \partial \Omega_1$,  functions in $\mathfrak{W}^{s,p}[\delta](\Omega_1)$ need to behave regularly on the boundary to have a finite seminorm $[u]_{\mathfrak{W}^{s,p}[\delta](\Omega_1)}^p$.   In fact, for $sp>1$, functions in $\mathfrak{W}^{s,p}[\delta](\Omega_1)$ have a well-defined trace on the boundary of $\partial \Omega_1$ that belong to $W^{s-{1/p}, p}(\partial \Omega_1)$. Indeed, the  trace operator 
\[
T_1: \mathfrak{W}^{s,p}[\delta](\Omega_1) \to W^{s-{1/p}, p}(\partial \Omega_1)
\]
is well-defined, linear and bounded (see \Cref{subsec:trace} for precise statements and details). This should be compared with the well-known result on the existence of a linear and bounded trace operator on fractional Sobolev spaces that have similar differentiability and integrability exponents. To fix notation, for $sp>1$, we denote the usual trace operator on $W^{s,p}(\Omega_2)$  
by $$
T_2: W^{s,p}(\Omega_2) \to W^{s-{1/p}, p}(\partial \Omega_2).$$ 
With this background we are now ready to specify the admissible class of functions for the minimization of the energy $\cE_{s,\delta}$ given in \eqref{energy}.  As we indicated earlier, the minimization will be over a class of pair of functions $(u_1, u_2) \in \mathfrak{W}^{s,p}[\delta](\Omega_1) \times W^{s,p}(\Omega_2)$ subject to a transmission condition on $\Sigma$ and a boundary condition on $\partial \Omega$. For $\delta>0$ and $sp>1$, we then have the tool to express the transmission condition across $\Sigma$, namely that 
\begin{equation}\label{eq:TransmissionContinuity}
	(T_1 u_1 - T_2 u_2) \mathds{1}_{\Sigma} = 0 \quad \scH^{d-1}\text{-a.e. on } \Sigma,
\end{equation}
where $T_1$ and $T_2$ are the respective trace operators and $\scH^{d-1}$ denotes the Hausdorff measure of dimension $d-1$. We also require that $u_1$ and $u_2$ satisfy a homogeneous Dirichlet boundary condition, namely that $T_1u_{1}$ and $T_2u_{2}$ vanish on $\partial \Omega \cap \partial \Omega_1$  and $\partial \Omega \cap \partial \Omega_2$ respectively. The admissible function space is therefore
\begin{equation} 
	\resizebox{0.9\textwidth}{!}{$ \label{SPACES}
		\mathfrak{X}_{s,\delta} := \left\{ (u_1, u_2) \in \mathfrak{W}^{s,p}[\delta](\Omega_1) \times W^{s,p}(\Omega_2)\, : \, 
		\begin{gathered}
			\eqref{eq:TransmissionContinuity} \text{ holds, and for $i=1, 2$ } \\
			T_i u_i = 0 \,\, \scH^{d-1}\text{-a.e. on } \p \Omega_i \setminus \Sigma 
		\end{gathered}
		\right\},$}
\end{equation}
where we suppress the dependence of the space on $p$. 
Relying on the trace theorems on $\mathfrak{W}^{s,p}[\delta](\Omega_1)$ and $W^{s,p}(\Omega_2)$ that will be established later,  the space $\mathfrak{X}_{s,\delta} $ equipped with the norm
\begin{eqnarray*}
	\Vnorm{(u_1,u_2)}_{\mathfrak{X}}^p := \vnorm{u_1}_{L^p(\Omega_1)}^p  + \vnorm{u_2}_{L^p(\Omega_2)}^p + [(u_1,u_2)]_{\mathfrak{X}}^p,
\end{eqnarray*}
with the seminorm $[(u_1,u_2)]_{\mathfrak{X}}^p := [u_1]^{p}_{\mathfrak{W}^{s,p}[\delta](\Omega_1)}+[u_2]^{p}_{W^{s,p}(\Omega_2)}$ will be shown to be a separable reflexive Banach space, see \Cref{prop:X-separable-reflexive}. We are now ready to state the main problem in a proper way. 

\subsection{{Variational formulation of the transmission problem and main results}}
Given $\delta, s \in (0,1)$, $p \in (1,\infty)$ such that $sp>1$, and 
$\mathfrak{f}\in \mathfrak{X}'_{s,\delta}$, the dual space of $\mathfrak{X}_{s,\delta}$,
the problem of interest is to find $(u_1,u_2) \in \mathfrak{X}_{s,\delta}$ satisfying
\begin{equation}\label{eq:Minprob:Transmiss}
	\begin{gathered}
		(u_1,u_2) = \argmin_{(v_1,v_2) \in \mathfrak{X}_{s,\delta}} \cF_{s,\delta}(v_1,v_2), \\
		\text{ where } \cF_{s, \delta}(v_1,v_2) := \cE_{s,\delta}(v_1,v_2) - \langle\mathfrak{f}, (v_1, v_2)\rangle_{\mathfrak{X}'_{s, \delta}, \mathfrak{X}_{s, \delta}},
	\end{gathered}
\end{equation}
and where $\langle\cdot, \cdot\rangle_{\mathfrak{X}'_{s, \delta}, \mathfrak{X}_{s, \delta}}  $ is the duality pairing on $\mathfrak{X}_{s, \delta}$.  
The first objective of this work,
discussed in detail in \Cref{sec:variational-problem}, is to establish the well-posedness of the problem. 
\begin{theorem}\label{thm:wellposedness}
	For a given $\delta>0$ sufficiently small, $s\in (0, 1)$, $p\in (1, \infty)$ such that $sp>1,$ and  $\mathfrak{f}\in \mathfrak{X}'_{s,\delta}$, 
	there exists a unique $(u_1,u_2) \in \mathfrak{X}_{s,\delta}$ satisfying \eqref{eq:Minprob:Transmiss}.
\end{theorem}
The theorem is a consequence of the direct method of calculus of variations; see a detailed proof in \Cref{sec:variational-problem}. To apply the method,  in
addition to establishing some structural properties of the space $\mathfrak{X}_{s,\delta}$, we need to prove the coercivity of $\mathcal{F}_{s, \delta}$, which will follow from a Poincar\'e-type inequality that we will need to state and prove. As stated earlier, all of these rely on intermediate results we prove via careful analysis of the nonstandard function space $\mathfrak{W}^{s,p}[\delta](\Omega_1)$, see details presented in \Cref{sec:function-space}.  


The second objective of this paper is to understand the behavior of the sequence of solutions as we vary the parameters $\delta$ and $s$ to extreme values in several regimes, see \Cref{sec:variational-convergence}. 
As minimizers of a family of parameterized energies $\cF_{s,\delta}$, one way to understand their limiting behavior is through the limiting behavior of the sequence of energies. To that end, we use a variational convergence -- $\Gamma$-convergence -- to find limiting energies that will be finite in appropriate function spaces.  Let us begin describing the energy functionals that will be obtained as a limit of each of the parts of $\cE_{s,\delta}$ as $\delta\to 0$ or $s\to 1$. We start with the second integral energy that defines $\cE_{s,\delta}$, as it depends only on the parameter $s$. Our interest is in the regime of convergence when $s\to 1$. It follows from the well-known result in \cite{bourgain2001} that for $u_2 \in W^{1,p}(\Omega_2)$, we have 
\[
\lim_{s\to 1}\frac{\kappa_{d,s,p}}{p} \int_{\Omega_2} \int_{\Omega_2} \beta(\bx)\frac{|u_2(\bx)-u_2(\by)|^p}{|\bx-\by|^{d+sp}} \, \rmd \by \, \rmd \bx =    \frac{1}{p} \int_{\Omega_2} \beta(\bx){|\nabla u_2(\bx)|^p}  \, \rmd \bx
\]
where
\begin{equation}\label{constant-A}
	\text{ $\kappa_{d,s,p} = \frac{p-sp}{A_{d,p}}$  with } A_{d,p} := \int_{\bbS^{d-1}} |\bsomega \cdot \be|^p \, \rmd \scH^{d-1}(\bsomega) = \frac{ 2 \pi^{\frac{d-1}{2}}\Gamma(\frac{p+1}{2} ) }{ \Gamma(\frac{d+p}{2} ) }.
\end{equation}
The first part of $\cE_{s,\delta}$ has two parameters. We are interested in the two convergence regimes: fix $\delta\in (0,1)$ and let $s\to 1$, and fix $s\in (0,1)$ and let $\delta\to 0$. On the one hand, using standard convergence theorems, it is not difficult to verify  that  fixing $\delta\in (0,1)$ and letting $s\to 1$, we obtain for smooth enough $u_1$ defined in $\Omega_1$ that 
\begin{align*}\lim_{s\to 1} \int_{\Omega_1} &\int_{ B(\bx,\delta \sigma(\bx)) } \alpha(\bx)\frac{|u_1(\bx)-u_1(\by)|^p}{ \delta^{d+p} \sigma(\bx)^{d+sp} } \, \rmd \by \, \rmd \bx\\
	&=  \int_{\Omega_1} \int_{ B(\bx,\delta \sigma(\bx)) } \alpha(\bx)\frac{|u_1(\bx)-u_1(\by)|^p}{ \delta^{d+p} \sigma(\bx)^{d+p} } \, \rmd \by \, \rmd \bx. 
\end{align*}

On the other hand, fixing $s\in (0,1)$, we will show later in the paper that for smooth enough functions $u_1$
\begin{align*}
	\lim_{\delta\to 0} \frac{\overline{C}_{d,p}}{p} \int_{\Omega_1}\alpha(\bx)& \int_{ B(\bx,\delta \sigma(\bx)) } \frac{|u_1(\bx)-u_1(\by)|^p}{ \delta^{d+p} \sigma(\bx)^{d+sp} } \, \rmd \by \, \rmd \bx\\ &= \frac{1}{p} \int_{\Omega_1} \alpha(\bx) |\grad u_1(\bx)|^p \sigma(\bx)^{p-sp} \, \rmd \bx,
\end{align*}
where \[\overline{C}_{d,p} = \frac{d+p}{A_{d,p}}\] 
and $A_{d, p}$ is given as in \eqref{constant-A}. In fact, we will demonstrate that the above convergence holds true for a class of $L^{p}$ functions in $\Omega_1$ with a weak derivative that has a finite weighted-$L^{p}$ norm. To be precise,  we define this weighted space as 
\begin{equation}\label{Def-weighted-space}
	W^{1,p}(\Omega_1;p-sp) := \left\{ u \in L^p(\Omega_1) \, :\, \int_{\Omega_1} |\grad u(\bx)|^p \sigma(\bx)^{p-sp} \, \rmd \bx < \infty \right\}. 
\end{equation}
To simplify notations, when $s=1$, the space $W^{1,p}(\Omega_1;p-sp) = W^{1,p}(\Omega_1;0)$ is understood to be $W^{1,p}(\Omega_1)$. Such weighted Sobolev spaces have been previously investigated in the literature, see for example \cite{gol2009weighted,gurka1988continuous}, and more general weighted spaces \cite{Fabes-Kenig-Serapioni,kufner1980weighted}.

We are now ready to state the second main result of the paper, which is the $\Gamma$-convergence of the sequence of functionals $\cF_{s, \delta}$ defined in \eqref{eq:Minprob:Transmiss} with respect to the strong $L^{p}(\Omega_1)\times L^{p}(\Omega_2)$-topology. For brevity,
such types of $\Gamma$-convergence are denoted by $\text{$\Gamma$-$L^{p}$},$  in various limits taken with respect to the parameters $\delta$ and $s$. Without introducing additional overbearing notation, we assume that $\cF_{s, \delta}(u_1, u_2) = \infty$ for all $(u_1, u_2)\in \left(L^{p}(\Omega_1)\times L^{p}(\Omega_2) \right)\setminus \mathfrak{X}_{s,\delta}$.

\begin{theorem}\label{Main-gammaconvergence}
	Let $1<p<\infty$ {and let $p' = \frac{p}{p-1}$ denote the H\"older conjugate}. Given $f_1\in L^{p'}(\Omega_1)$ and $f_2\in L^{p'}(\Omega_2)$, 
    consider the class of functionals $\cF_{s, \delta}(u_1, u_2)$ given in \eqref{eq:Minprob:Transmiss} parameterized by $(\delta, s)\in (0, \underline{\delta}_0)\times (0, 1)$ for sufficiently small but fixed $\underline{\delta}_0>0$.
	Then we have the following $\Gamma$-convergence results with respect to strong $L^{p}(\Omega_1)\times L^{p}(\Omega_2)$ convergence:
	
	\begin{enumerate}[label=\textbf{\upshape \alph*)}]
		\item \label{item:a} Fix $s\in (0, 1)$. Then  as  $\delta \to 0$, 
		\[
		\cF_{s,\delta}(u_1,u_2) \xrightarrow{\text{$\Gamma$-$L^{p}$}} \cF_{s, 0}(u_1, u_2)
		\]
		where the limiting energy energy functional $\cF_{s, 0} : L^p(\Omega_1) \times L^p(\Omega_2) \to \bbR \cup \{+\infty\}$ is given by \[\cF_{s, 0}(u_1, u_2) = \cE_{s, 0}(u_1, u_2) - \int_{\Omega_1} f_1(\bx) u_1(\bx) \rmd\bx - \int_{\Omega_2} f_2(\bx) u_2(\bx) \rmd\bx\]  
		with 
		\begin{equation}\label{defn-E(s,0)}
			\begin{split}
				\cE_{s, 0}(u_1, u_2) &:= \frac{1}{p}\int_{\Omega_1} \alpha(\bx)|\nabla u_1(\bx)|^{p} \sigma(\bx)^{p-sp} \, \rmd\bx \\
				&\qquad + \frac{\kappa_{d,s,p}}{p} \int_{\Omega_2} \int_{\Omega_2}\beta(\bx) \frac{|u_2(\bx)-u_2(\by)|^p}{|\bx-\by|^{d+sp}} \, \rmd \by \, \rmd \bx
			\end{split}
		\end{equation}
		for $(u_1, u_2)\in \mathfrak{X}_{s, 0}$ and $\cF_{s, 0}(u_1,u_2) = +\infty$ otherwise, where   
		\[\mathfrak{X}_{s, 0}= \left\{\begin{gathered}(u_{1}, u_2)\in W^{1, p}(\Omega_1;p-sp)\times W^{s, p}(\Omega_2): (T_1u_{1} - T_2u_{2})|_{ \Sigma } = 0,\\
			T_iu_{i} = 0\quad \scH^{d-1}\text{-a.e. on } \partial \Omega_i\setminus  \Sigma, \text{ for } i=1,2
		\end{gathered}\right \}. \]
		
		\item \label{item:b} Fix $\delta\in (0, \underline{\delta}_0). $ Then as $s\to 1$, we have
		\[
		\cF_{s,\delta}(u_1,u_2)  \xrightarrow{\text{$\Gamma$-$L^{p}$}}  \cF_{1, \delta}(u_{1}, u_{2})
		\]
		where $\cF_{1, \delta} : L^p(\Omega_1) \times L^p(\Omega_2) \to \bbR \cup \{+\infty\}$ is given by $$\cF_{1, \delta}(u_{1}, u_{2}) =  \cE_{1, \delta}(u_1, u_2) -\int_{\Omega_1} f_{1}(\bx)u_1(\bx) \rmd\bx - \int_{\Omega_2}f_2(\bx) u_2(\bx)\rmd\bx
		$$
		with 
		\begin{equation}\label{defn-E(1,delta)}
			\begin{split}
				\cE_{1, \delta}(u_1, u_2) &:=  {\overline{C}_{d,p} \over p}\int_{\Omega_1} \int_{ B(\bx,\delta \sigma(\bx)) } \alpha(\bx)\frac{|u_1(\bx)-u_1(\by)|^p}{ \delta^{d+p} \sigma(\bx)^{d+p} } \, \rmd \by \, \rmd \bx \\
				&\qquad + \frac{1}{p} \int_{\Omega_2} \beta(\bx)|\nabla u_{2}(\bx)|^{p} \, \rmd \bx
			\end{split}
		\end{equation}
		for $(u_1, u_2)\in \mathfrak{X}_{1, \delta}$ and $\cF_{1,\delta}(u_1,u_2) = + \infty$ otherwise, where
		\[
		\mathfrak{X}_{1, \delta} = \left\{
		\begin{gathered}(u_{1}, u_{2})\in \mathfrak{W}^{1, p}[\delta](\Omega_1) \times W^{1, p}(\Omega_2): (T_1u_{1} - T_2u_{2})|_{\Sigma} = 0,\\
			T_iu_{i} = 0,\quad \scH^{d-1}\text{-a.e. on } \partial \Omega_i\setminus \Sigma, \text{ for }i=1, 2
		\end{gathered}
		\right\}.
		\]
		\item \label{item:c} For $s\in (0, 1)$, let $\cE_{s, 0}(u_1, u_2)$ be as defined in \eqref{defn-E(s,0)}. Then as $s\to 1$, 
		\[
		\cF_{s,0}(u_1,u_2) \xrightarrow{\text{$\Gamma$-$L^{p}$}} \cF_{1, 0}(u)
		\]
		where $u$ is defined as in \Cref{subsec:motivation} and $\cF_{1,0}(u) : L^p(\Omega) \to \bbR \cup \{+\infty\}$ is defined as
		$$
		\cF_{1, 0}(u) = \cE_{1, 0}(u) - \int_{\Omega_1} f_{1}(\bx)u(\bx)\rmd\bx -\int_{\Omega_2}f_2(\bx) u(\bx)\rmd\bx
		$$
		where if $\Lambda(\bx) := \alpha(\bx)\mathds{1}_{\Omega_1}(\bx) + \beta(\bx) \mathds{1}_{\Omega_2}(\bx),$ then 
		\begin{equation}\label{defn-E(1,0)}
			\cE_{1, 0}(u) := \frac{1}{p}\int_{\Omega} \Lambda(\bx) |\nabla u (\bx) |^{p} \rmd\bx, \quad 
		\end{equation}
		for $u \in \mathfrak{X}_{1,0}:= W^{1, p}_0(\Omega)$ and $\cF_{1,0}(u) = +\infty$ otherwise.
        
		\item \label{item:d} For $ \delta\in (0, \underline{\delta}_0)$, let $\cE_{1, \delta}$ be defined as in \eqref{defn-E(1,delta)}. Then as $\delta\to0$, we have 
		\[
		\cF_{1, \delta}(u_1, u_2) \xrightarrow{\text{$\Gamma$-$L^{p}$}}  \cF_{1, 0}(u),
		\]
		where $\cE_{1,0}$ is as defined in \eqref{defn-E(1,0)}.
		\item \label{item:e} Finally, for any $(s_k, \delta_k) \to (1, 0)$ as $k\to \infty$, we have 
		\[
		\cF_{s_k, \delta_k}(u_{1}, u_{2})  \xrightarrow{\text{$\Gamma$-$L^{p}$}}  \cF_{1,0}(u).
		\]
	\end{enumerate}

\end{theorem}


Notice that, in the above theorem, $p'$ is the H\"older conjugate of $p$, and for a given $(f_1, f_2)\in L^{p'}(\Omega_1) \times L^{p'}(\Omega_2)$ the functional $\mathfrak{f}$ given by  \[\langle\mathfrak{f}, (v_1, v_2)\rangle := \int_{\Omega_1} f_{1}(\bx) v_{1}(\bx) \rmd\bx +\int_{\Omega_2} f_{2}(\bx) v_{2}(\bx) \rmd\bx\]
defines an object in the dual space of  $\mathfrak{X}_{s, \delta}$ for all $s\in (0, 1]$, and $\delta\in [0, 1)$.

We remark that the above theorem, proved in \Cref{sec:variational-convergence}, shows that the diagram \Cref{fig:variational_convergence} commutes where the arrow represents $\Gamma$-convergence in the respective topology. 
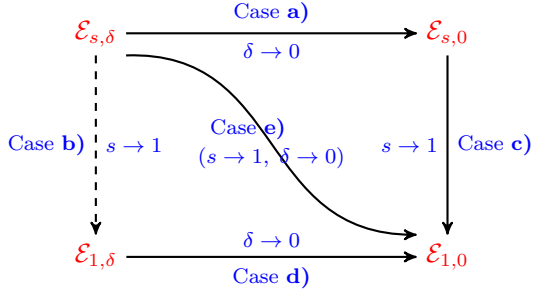
\begin{figure}
	\centering
	\begin{tikzpicture}[scale=0.85]
		\tikzset{to/.style={->,>=stealth',line width=.8pt}}   
		\node(v1) at (0,3.5) {\textcolor{red}{$\mathcal{E}_{s,\delta}$}};
		\node (v2) at (5.5,3.5) {\textcolor{red}{$\mathcal{E}_{s,0}$}};
		\node (v3) at (0,0) {\textcolor{red}{$\mathcal{E}_{1,\delta}$}};
		\node (v4) at (5.5,0) {\textcolor{red}{$\mathcal{E}_{1,0}$}};
		\draw[to] (v1.east) -- node[midway,above] {\footnotesize{\textcolor{blue}{  Case \ref{item:a}}}}  node[midway,below] {\footnotesize{\textcolor{blue}{$\delta\to0$}}}      
		(v2.west);
		\draw[to,dashed] (v1.south) -- node[midway,left] {\footnotesize{\textcolor{blue}{ Case \ref{item:b}}}} node[midway,right] {\footnotesize{\textcolor{blue}{$s\to1$}}} (v3.north);
		\draw[to] (v3.east) -- node[midway,below] {\footnotesize{\textcolor{blue}{ Case \ref{item:d}}}} node[midway,above] {\footnotesize{\textcolor{blue}{$\delta\to0$}}} (v4.west);
		\draw[to] (v2.south) -- node[midway,right] {\footnotesize{\textcolor{blue}{ Case \ref{item:c}}}} node[midway,left] {\footnotesize{\textcolor{blue}{$s\to1$}}}(v4.north);
		\draw[to] (v1.south east) to[out = 2, in = 180, looseness = 1.2] node[midway] {\footnotesize\hbox{\shortstack[l]{ {\textcolor{blue}{ $\;\;$Case \ref{item:e}}}\\ {\textcolor{blue}{($s\to1,\;\delta\to0$)}} }}} (v4.north west);
	\end{tikzpicture}
	\caption{Variational convergence of parameterized functionals 
	}
	\label{fig:variational_convergence}
\end{figure}
A corollary of the $\Gamma$-convergence of the parameterized energies is the convergence of the minimizers of the energies in the event 
where the energies are equicoercive. We 
thus have
\begin{corollary}\label{Convergence-minimizer}
	Under the assumptions of \Cref{Main-gammaconvergence}, the sequence of minimizers $(u^{s,\delta}_1, u^{s,\delta}_2) = \argmin_{\mathfrak{X}_{s, \delta}} \cF_{s,\delta}(v_1, v_2)$ converges in the strong $L^p(\Omega_1) \times L^p(\Omega_2)$ topology to a minimizer of the limiting energy functional in Case \ref{item:a}, Case \ref{item:c}, Case \ref{item:d} and Case \ref{item:e}.
\end{corollary}

In Case \ref{item:b}, on the one hand, although we have $\Gamma$-convergence of the sequence of energy functionals $\cF_{s, \delta}$  in the strong $L^{p}(\Omega_1) \times L^{p}(\Omega_2)$ topology, it is not clear whether sequence of minimizers converge to a minimizer in the same topology. Part of the difficulty is that it is not possible to show that the sequence of functionals is equicoercive, which is related to the compact embedding of  $\mathfrak{X}_{s,\delta}$ in $L^{p}(\Omega_1)\times L^{p}(\Omega_2)$. For a fixed $\delta,$ we do not necessarily have compact embedding for the function space $\mathfrak{W}^{s,p}[\delta](\Omega_1)$ in $L^{p}(\Omega_1)$
as it contains all functions in $L^{p}_{loc}(\Omega_1)$. 
On the other hand,  by weakening the topology, we can show convergence of minimizers in the weak $L^{p}$-topology  to the respective minimizer. 
The following proposition states this result. 

	\begin{proposition}[Case \ref{item:b}']\label{Convergence-minimizer-weakLp}
		Under the assumptions of \Cref{Main-gammaconvergence}, for a fixed $\delta\in (0, \underline{\delta}_0), $   we have that as $s\to 1$,
		\[
		\cF_{s,\delta}(u_1,u_2)  \xrightarrow{\text{$\Gamma$}}  \cF_{1, \delta}(u_{1}, u_{2})
		\]
		with respect to the weak topology in $L^{p}(\Omega_1)\times L^{p}(\Omega_2)$. Moreover, the sequence of minimizers $(u^{s,\delta}_1, u^{s,\delta}_2) = \argmin_{\mathfrak{X}_{s, \delta}} \cF_{s,\delta}(v_1, v_2)$ converges in the weak $L^p(\Omega_1) \times L^p(\Omega_2)$ topology, as $s\to 1$, to a minimizer of the limiting energy functional, $(u^{1,\delta}_1, u^{1,\delta}_2) = \argmin_{\mathfrak{X}_{1, \delta}} \cF_{1, \delta}(v_1, v_2)$.
\end{proposition}

The paper is organized as follows: \Cref{sec:function-space} introduces the function space framework for the variational problems, as well as some basic embedding properties. Further properties of the function spaces are established in \Cref{sec:properties}, such as the density of smooth functions, trace and extension theorems, and Poincar\'e-type inequalities; properties of the necessary analytical tools are reviewed beforehand in \Cref{sec:convolution}. \Cref{sec:variational-problem} contains the proof of the well-posedness of the variational problems, and the $\Gamma$-convergence results for the functional sequence under the different asymptotic regimes are in \Cref{sec:variational-convergence}. We conclude by mentioning some generalizations of our program in \Cref{sec:extensions}. In the appendix we collect well-known properties of fractional and weighted Sobolev spaces, for which we obtain further qualitative clarifications necessary for our main results.

\section{Function spaces}\label{sec:function-space}
In this section, we present all necessary intermediate results in connection with function spaces that we will need in the proof of the main results stated in the previous section. In particular, as a nonstandard space, we pay special attention to the space $\mathfrak{W}^{s,p}[\delta](\Omega_1)$, its structural properties and connections with the classical spaces, the fractional and weighted Sobolev spaces. Additional properties, such as the density of smooth functions in the space and the existence of a well-defined trace operator, will be established in \Cref{sec:properties}.
For ease of notation we will denote $\Omega_1$ by $\mathcal{D}$ throughout this section. 

\subsection{The function space \texorpdfstring{$\mathfrak{W}^{s,p}[\delta](\mathcal{D})$}{ }}
We recall that for $\delta, s\in(0,1)$ and $1<p<\infty$ the function space $\mathfrak{W}^{s,p}[\delta](\mathcal{D})$ collects measurable functions in $\mathcal{D}$ with a finite norm given by $
\Vnorm{u}_{\mathfrak{W}^{s,p}[\delta](\mathcal{D})} = \left(\Vnorm{u}_{L^p(\mathcal{D})}^p + [u]_{\mathfrak{W}^{s,p}[\delta](\mathcal{D})}^p\right)^{1\over p}. 
$ 
We recall that $[u]_{\mathfrak{W}^{s,p}[\delta](\mathcal{D})}^p$ is given by 
\[
[u]_{\mathfrak{W}^{s,p}[\delta](\mathcal{D})}^p = \overline{C}_{d,p}\intdm{\mathcal{D}}{ \int_{B(\bx,\delta \sigma(\bx))}   \frac{\big| u(\bx)-u(\by) \big|^p}{ \delta^{d+p}  \sigma(\bx)^{d+sp} }  \, \rmd \by}{\bx}
\]
with $\sigma(\bx) = \dist(\bx, \partial \mathcal{D})$. Following an argument in  \cite{Du2022Fractional} we can show that $[\cdot]_{\mathfrak{W}^{s,p}[\delta](\mathcal{D})}$,  indeed, defines a seminorm. In particular, if $[u]_{\mathfrak{W}^{s,p}[\delta](\mathcal{D})} = 0$, then $u$ is constant in connected subsets of $\mathcal{D}.$ 
We may use the symmetrized kernel
\[
\gamma^{s, \delta}_{p, \mathcal{D}}(\bx, \by) =  \frac{\overline{C}_{d,p}}{2\delta^{d+p}}\left( \frac{\mathds{1}_{B(\mathbf{0},\delta \sigma(\bx))}(\bx-\by)}{\sigma(\bx)^{d+sp}}+   \frac{\mathds{1}_{B(\mathbf{0},\delta \sigma(\by))}(\bx-\by)}{\sigma(\by)^{d+sp}}\right)
\]
to rewrite the seminorm as 
\[
[u]_{\mathfrak{W}^{s,p}[\delta](\mathcal{D})}^p = \int_{\mathcal{D}}\int_{\mathcal{D}} \gamma^{s, \delta}_{p, \mathcal{D}}(\bx, \by)|u(\bx) - u(\by)|^{p} \rmd \by \rmd\bx.
\]
\begin{proposition}
	For $\delta, s\in(0,1)$ and $1<p<\infty,$ the space $\mathfrak{W}^{s,p}[\delta](\mathcal{D})$ is a separable and reflexive Banach space. For $p=2$, $\mathfrak{W}^{s,2}[\delta](\mathcal{D})$ is a Hilbert space with inner product given by, for $u, v\in \mathfrak{W}^{s,2}[\delta](\mathcal{D})$, 
	\[
	\langle u, v\rangle = \int_{\mathcal{D}} u(\bx)v(\bx) \rmd\bx + \int_{\mathcal{D}}\int_{\mathcal{D}} \gamma^{s, \delta}_{2, \mathcal{D}}(\bx, \by)(u(\bx) - u(\by))(v(\bx) - v(\by))\rmd\by \rmd\bx.  
	\]
\end{proposition}
\begin{proof}
	We first show the completeness of the space. Suppose that $\{u_n\}$ is a Cauchy sequence in $\mathfrak{W}^{s,p}[\delta](\mathcal{D}).$ Then there exists $u\in L^{p}(\mathcal{D})$ such that $u_n\to u$ in $L^{p}(\mathcal{D})$. By Fatou's lemma, $u\in \mathfrak{W}^{s,p}[\delta](\mathcal{D})$. To conclude the proof it suffices to show $[u_n-u]_{\mathfrak{W}^{s,p}[\delta](\mathcal{D})}\to 0$ as $n\to \infty.$  Now for  $\tau>0$ given, we introduce the mapping $J_\tau: \mathfrak{W}^{s,p}[\delta](\mathcal{D}) \to L^{p'}(\mathcal{D})$, with $1/p + 1/p' = 1,$ given by 
	\[
	J_\tau(u)(\bx) = -2 \int_{\mathcal{D}} \mathds{1}_{\{\mathcal{D}_\tau \times \mathcal{D}_\tau\}}(\bx,\by)\gamma^{s, \delta}_{2, \mathcal{D}}(\bx, \by)|u(\bx) - u(\by)|^{p-2} (u(\bx) - u(\by)) \rmd\by 
	\]
	where $\mathcal{D}_\tau = \{\bx\in \mathcal{D}: \sigma(\bx)>\tau\}$. For $u\in \mathfrak{W}^{s,p}[\delta](\mathcal{D}) $,  we then see that
	\begin{equation*}
		\begin{split}
			\lim_{\tau\to 0}\int_{\mathcal{D}} J_\tau(u)\cdot u \, \rmd\bx  &= \lim_{\tau \to 0}\int_{\mathcal{D}_\tau}\int_{\mathcal{D}_\tau} \gamma^{s, \delta}_{p, \mathcal{D}}(\bx, \by)|u(\bx) - u(\by)|^{p} \rmd\by \rmd\bx =[u]^{p}_{\mathfrak{W}^{s,p}[\delta](\mathcal{D})}, 
		\end{split}
	\end{equation*}
	where we used Fubini's theorem and the dominated convergence theorem.  We apply this to the sequence ${u_n- u_m}$. Since $\{u_n\}$ is a Cauchy sequence in $\mathfrak{W}^{s,p}[\delta](\mathcal{D})$, for every $\veps>0$ there exists $K>0$ such that for all $n,m\geq K$ it holds that $[u_n-u_m]_{\mathfrak{W}^{s,p}[\delta](\mathcal{D})} <\veps$.
	Then for any $\tau>0$, and $m, n\geq K$
	\[
	\int_{\mathcal{D}} J_\tau(u_n - u_m)\cdot (u_n - u_m) \rmd\bx \leq [u_n-u_m]_{\mathfrak{W}^{s,p}[\delta](\mathcal{D})} < \veps.
	\]
	Fix $m\geq K$ and $\tau$, and let $n\to \infty$; we then have 
	$\int_{\mathcal{D}} J_\tau(u - u_m)\cdot (u - u_m) \, \rmd\bx <\veps.$
	We now let $\tau \to 0$ to conclude that for any $m\geq K$
	\[
	[u-u_m]_{\mathfrak{W}^{s,p}[\delta](\mathcal{D})} = \lim_{\tau \to 0} \int_{\mathcal{D}} J_\tau(u - u_m)\cdot (u - u_m) \rmd\bx \leq \veps. 
	\]
	The separability and reflexivity follows from the fact that we can define an isometry $T: \mathfrak{W}^{s,p}[\delta](\mathcal{D}) \to L^{p}(\mathcal{D}) \times L^{p}(\mathcal{D}\times \mathcal{D})$ given by 
	\[
	T(u) = (u, U),\quad \text{$U(\bx, \by):=\left(\gamma^{s, \delta}_{p, \mathcal{D}}(\bx, \by)\right)^{1/p}|u(\bx)-u(\by)|$}
	\]
	and that $T[\mathfrak{W}^{s,p}[\delta](\mathcal{D})]$ is a (strongly) closed subspace of the separable and reflexive space $L^{p}(\mathcal{D}) \times L^{p}(\mathcal{D}\times \mathcal{D})$, see \cite[Propositions 3.20 and 3.25]{Brez11}.
\end{proof}

The following theorem quantifies the continuous embedding of one nonlocal space into another based on the values of $\delta$ and gives a precise comparison estimate between corresponding seminorms. The proof adapts the one given \cite[Lemma 6.2]{TiDu17} and \cite[Lemma 2.2]{Du2022Fractional} to our setting. The case when $s=1$ is proved in \cite{scott2024nonlocal}.
\begin{theorem}
	\label{thm:InvariantHorizon}
	Suppose that $s\in (0,1]$ and 
	$p>1$. Let  $0 < \delta_1 \leq \delta_2 < \frac{1}{3}$. Then the  spaces $ \mathfrak{W}^{s,p}[\delta_2](\mathcal{D})$ and $\mathfrak{W}^{s,p}[\delta_1](\mathcal{D})$ are the same. Moreover, for any $u\in \mathfrak{W}^{s,p}[\delta_1](\mathcal{D}) = \mathfrak{W}^{s,p}[\delta_2](\mathcal{D}),$
	\begin{equation*}
		\left( \frac{\delta_1}{\delta_2} \right)^{1+{d\over p}} [u]_{\mathfrak{W}^{s,p}[\delta_1](\mathcal{D})} \leq [u]_{\mathfrak{W}^{s,p}[\delta_2](\mathcal{D})} \leq \left(\frac{2}{1-\delta_2}\right)^{1+{d\over p}}  (1+\delta_2)^{s+{d\over p}} [u]_{\mathfrak{W}^{s,p}[\delta_1](\mathcal{D})}. 
	\end{equation*}

\end{theorem}

\begin{proof}
The first inequality is trivial, so the proof is devoted to the second inequality. For the sake of brevity we omit the constant $\overline{C}_{d,p}$ in the calculations.
Let $n \in \bbN$.
To begin the estimate for $[u]_{\mathfrak{W}^{s,p}[\delta_2](\mathcal{D})}$, we apply the triangle inequality to the telescoping sum
\begin{equation*}
	|u(\bx+\bh) - u(\bx)| \leq \sum_{i = 1}^n \left| u \left( \bx + \frac{i}{n} \bh \right) - u \left( \bx + \frac{i-1}{n} \bh \right)\right|\,.
\end{equation*}
for $\bx \in \mathcal{D}$ and $\bh \in B(\mathbf{0}, \delta_2 \sigma(\bx))$. It then follows that $|\bx+\frac{i}{n}\bh - \bx| \leq |\bh| < \delta_2 \sigma(\bx)$, so thus $\bx +\frac{i}{n}\bh \in \mathcal{D}$ for $i = 0,1,\ldots,n$. Thus, setting $\bx_i := \bx + \frac{i-1}{n} \bh$ and using H\"older's inequality for sums
\begin{equation*}
	\begin{split}
		[u]_{ \mathfrak{W}^{s,p}[\delta_2](\mathcal{D}) }^p \leq n^{p-1} \sum_{i = 1}^n \int_{ \mathcal{D} } \int_{B(\mathbf{0},\delta_2 \sigma(\bx) ) } \frac{ |u(\bx_i+\frac{1}{n}\bh)-u(\bx_i)|^p }{ \delta_2^{d+p} \sigma(\bx)^{d+sp} } \, \rmd \bh \, \rmd \bx\,.
	\end{split}
\end{equation*}
Now, since $\sigma$ is Lipschitz with Lipschitz constant $\leq 1$,
\begin{equation*}
	\sigma(\bx_i) = \sigma(\bx + (i-1)\bh/n ) \leq |\bh| + \sigma(\bx) \leq (\delta_2 + 1)  \sigma(\bx)\,,
\end{equation*}
and
\begin{equation*}
	\sigma(\bx) \leq  |\bh| + \sigma(\bx_i) \leq \delta_2 \sigma(\bx) +  \sigma(\bx_i)\,,
\end{equation*}
so therefore
\begin{equation}\label{eq:VaryingLambda:SeminormComparison:Pf2}
	\frac{3}{4} \sigma(\bx_i)
	\leq \frac{\sigma(\bx_i)}{1+\delta_2 } \leq \sigma(\bx) \leq \frac{\sigma(\bx_i)}{1 - \delta_2 }
	\leq \frac{3}{2} \sigma(\bx_i)
\end{equation}
for all $\bx \in \mathcal{D}$.
Hence,	\begin{equation}\label{eq:VaryingLambda:SeminormComparison:Pf3}
	\begin{split}
		[u]_{ \mathfrak{W}^{s,p}[\delta_2](\mathcal{D}) }^p \leq n^{p-1} (1+\delta_2)^{d+sp} \sum_{i = 1}^n \int_{ \mathcal{D} } \int_{B(\mathbf{0},\frac{\delta_2}{1-\delta_2}  \sigma(\bx_i)) }
          \hspace{-.4cm} \frac{ |u(\bx_i+\frac{1}{n}\bh)-u(\bx_i)|^p }{ \delta_2^{d+p} \sigma(\bx_i)^{d+sp} } \, \rmd \bh \, \rmd \bx\,.
	\end{split}
\end{equation}
Now, for $\bx \in \mathcal{D}$ and $|\bh| \leq \frac{\delta_2}{1 - \delta_2} \sigma(\bx_i)$, we have
\begin{equation*}
	\sigma(\bx_i) \leq \sigma(\bx) + |\bh| \leq \sigma(\bx) + \frac{\delta_2}{1 - \delta_2} \sigma(\bx_i)\,.
\end{equation*}
Therefore $ \sigma(\bx_i) \leq \frac{ 1-\delta_2 }{1 - 2\delta_2 }  \sigma(\bx)$ and since $\delta_2 < \frac{1}{3}$ we can conclude that
\begin{equation*}
	|\bx-\bx_i| \leq |\bh| \leq \frac{\delta_2}{1-\delta_2} \sigma(\bx_i) \leq \frac{\delta_2}{1 - 2\delta_2} \sigma(\bx) < \sigma(\bx)\,,
\end{equation*}
i.e. $\bx_i \in \mathcal{D}$ for all $i = 1, \ldots, n$.
With this we can perform a change of variables in the outer integral; letting $\by = \bx_i = \bx + \frac{i-1}{n} \bh$ in \eqref{eq:VaryingLambda:SeminormComparison:Pf3} and using \eqref{eq:VaryingLambda:SeminormComparison:Pf2}, we get
\begin{equation*}
	\begin{split}
		[u]_{ \mathfrak{W}^{s,p}[\delta_2](\mathcal{D}) }^p 
		&\leq n^{p-1} (1+\delta_2)^{d+sp} \sum_{i = 1}^n \int_{ \mathcal{D} } \int_{B(\mathbf{0},\frac{\delta_2}{1-\delta_2} \sigma(\by) ) }   \hspace{-.2cm} \frac{ |u(\by+\frac{1}{n}\bh)-u(\by)|^p }{ \delta_2^{d+p} \sigma(\by)^{d+sp} } \, \rmd \bh \, \rmd \by \\
		&= n^{p} (1+\delta_2)^{d+sp} \int_{ \mathcal{D} } \int_{B(\mathbf{0},\frac{\delta_2}{1-\delta_2} \sigma(\by) ) } 
        \hspace{-.2cm}\frac{ |u(\by+\frac{1}{n}\bh)-u(\by)|^p }{ \delta_2^{d+p} \sigma(\by)^{d+sp} } \, \rmd \bh \, \rmd \by\,.
	\end{split}
\end{equation*}
Now perform a change of variables in the inner integral: let $\bz = \frac{\bh}{n}$, and set $\mu_n := \frac{\delta_2}{n(1-\delta_2)}$; we get
\begin{equation*}
	\begin{split}
		[u]_{ \mathfrak{W}^{s,p}[\delta_2](\mathcal{D}) }^p 
		&\leq n^{d+p} (1+\delta_2)^{d+sp} \int_{ \mathcal{D} } \int_{B(\mathbf{0},\frac{\delta_2}{1-\delta_2} \frac{ \sigma(\by) }{n}) } \frac{ |u(\by+\bz)-u(\by)|^p }{ \delta_2^{d+p} \sigma(\by)^{d+sp} } \, \rmd \bz \, \rmd \by \\
		&= n^{d+p} \mu_n^{d+p} \frac{(1+\delta_2)^{d+sp}}{\delta_2^{d+p}} \int_{ \mathcal{D} } \int_{B(\mathbf{0}, \mu_n \sigma(\by)) } \frac{ |u(\by+\bz)-u(\by)|^p }{ \mu_n^{d+p} \sigma(\by)^{d+sp} } \, \rmd \bz \, \rmd \by \\
		&= \frac{(1+\delta_2)^{d+sp}}{(1-\delta_2)^{d+p}} \int_{ \mathcal{D} } \int_{B(\mathbf{0}, \mu_n \sigma(\by)) } \frac{ |u(\by+\bz)-u(\by)|^p }{ \mu_n^{d+p} \sigma(\by)^{d+sp} } \, \rmd \bz \, \rmd \by \,.
	\end{split}
\end{equation*}
Now, fix $n \in \bbN$ such that
\begin{equation*}
	\frac{\delta_2}{\delta_1 ( 1 - \delta_2)} < n < \frac{2\delta_2}{\delta_1 ( 1 - \delta_2)} \,.
\end{equation*}
Therefore $\delta_1 / 2 < \mu_n < \delta_1$ and we have
\begin{equation*}
	\begin{split}
		[u]_{ \mathfrak{W}^{s,p}[\delta_2](\mathcal{D}) }^p 
		&\leq \frac{(1+\delta_2)^{d+sp} }{(1-\delta_2)^{d+p} } \int_{ \mathcal{D} } \int_{B(\mathbf{0},\delta_1 \sigma(\by) ) } \frac{ |u(\by+\bz)-u(\by)|^p }{ \mu_n^{d+p} \sigma(\by)^{d+sp} } \, \rmd \bz \, \rmd \by \\
		&\leq \left(\frac{2}{1-\delta_2}\right)^{d+p} (1+\delta_2)^{d+sp} \int_{ \mathcal{D} } \int_{B(\mathbf{0},\delta_1 \sigma(\by) } \frac{ |u(\by+\bz)-u(\by)|^p }{ \delta_1^{d+p} \sigma(\by)^{d+sp} } \, \rmd \bz \, \rmd \by\,,
	\end{split}
\end{equation*}
as desired.
\end{proof}
\subsection{Embeddings with fractional and weighted Sobolev spaces}
We first recall the observation 
that $W^{s,p}(\mathcal{D})\subset \mathfrak{W}^{s,p}[\delta](\mathcal{D})$ for any $0<s<1$, $p\geq 1$ and $\delta\in (0, 1]$ with the estimate of norms that for all $u\in W^{s,p}(\mathcal{D})$
\begin{equation}\label{frak-frac-estimate}
[u]_{\mathfrak{W}^{s,p}[\delta](\mathcal{D})} \leq C(d,s,p)
\delta^{s-1} [u]_{W^{s,p}(\mathcal{D})}.
\end{equation}
This follows from the fact that for all $\bx, \by \in \mathcal{D}$ with $|\bx-\by|\leq \delta \sigma(\bx)$, it holds that $\delta^{-d-p}|\sigma(\bx)|^{-d-sp} \leq \delta^{sp-p} |\bx-\by|^{-d-sp}.$
Similar estimates for the weighted Sobolev space $W^{1,p}(\cD; p-sp)$, introduced in \eqref{Def-weighted-space}, can be established, which we state and prove below.  
\begin{proposition}\label{lma:Embedding:Weighted:Pre}
Fix $p\in (1,\infty)$, $s \in (0,1]$ and a sufficiently small $\underline{\delta}_0 > 0$. Let $\cD \subset \mathbb{R}^{d}$ be a bounded domain.  Then for all $\delta<\underline{\delta}_0$ we have that  
\begin{equation*}
	[u]_{\mathfrak{W}^{s,p}[\delta](\cD)}^p \leq \frac{C(d, p)}{(1-\delta)^{p-sp+1}} [u]_{W^{1,p}(\cD;p-sp)}^p\,,
    \quad
     \forall\, u \in  C^{\infty}(\cD)\cap W^{1,p}(\cD;p-sp),
\end{equation*}
where $C(d,p) = \frac{\scH^{d-1}(\bbS^{d-1}) \overline{C}_{d,p} }{d+p}$. If $\cD$ is a Lipschitz domain, then the inequality holds for all $u\in C^\infty(\overline{\cD}) \cap W^{1,p}(\cD;p-sp).$
\end{proposition}

\begin{proof}
It suffices to prove the inequality  for $u\in C^{\infty}(\cD)\cap W^{1,p}(\cD;p-sp).$ In the event $\cD$ is a Lipschitz domain, the validity of the inequality for all $u \in W^{1,p}(\cD;p-sp)$ follows from the density of $C^{\infty} (\overline{\cD})$ in $W^{1,p}(\cD;p-sp),$ see \Cref{thm:FxnSpProp:Weighted}.

Now using Taylor's integral formula, 
for any $\bz \in B(\mathbf{0},\delta \sigma(\bx))$ we have 
\begin{equation*}
	|u(\bx+\bz) - u(\bx)|^p \leq |\bz|^p \int_0^1 |\grad u(\bx+t\bz)|^p \, \rmd t\,,
\end{equation*}
and so
\begin{equation*}
	\begin{split}
		[u]_{\mathfrak{W}^{s,p}[\delta](\cD)}^p 
		&= \int_{\cD} \int_{ B(\mathbf{0},\delta \sigma(\bx)) }  \frac{ \overline{C}_{d,p} }{\delta^{d+p} \sigma(\bx)^{d+sp}} |u(\bx+\bz) - u(\bx)|^p \, \rmd \bz \, \rmd \bx \\
		&\leq \int_{\cD} \int_{ B(\mathbf{0},\delta \sigma(\bx)) } \frac{ \overline{C}_{d,p} |\bz|^p}{\delta^{d+p} \sigma(\bx)^{d+sp}}\int_0^1 |\grad u(\bx+t\bz)|^p \, \rmd t \, \rmd \bz \, \rmd \bx \\
		&= \overline{C}_{d,p} \int_{\cD} \int_{ B(\mathbf{0},1) } \sigma(\bx)^{p-sp} |\bz|^p \int_0^1 |\grad u(\bx+t \delta \sigma(\bx) \bz)|^p \, \rmd t \, \rmd \bz \, \rmd \bx\,.
	\end{split}
\end{equation*}
Now we make a change of variables. 
Define the function $\bszeta_{\bz}^{t\delta}(\bx) :=\bx+t \delta \sigma(\bx) \bz$. It is proved in \cite[Lemma 3.2]{scott2024nonlocal}  that for all $\delta\in (0,\underline{\delta}_0), $   
$\bszeta_{\bz}^{t\delta}$ is invertible on $\cD$ with $\bszeta_{\bz}^{t\delta}(\cD) \subset \cD$, $\mathrm{det}(\nabla \bszeta_{\bz}^{t\delta}(\bx))\ge 1-t\delta$, and $\sigma(\bx)\le \sigma(\bszeta_{\bz}^{t\delta}(\bx))/(1-t\delta)$; changing coordinates and using Tonelli's theorem give
\begin{equation*}
	\begin{split}
		[u]_{\mathfrak{W}^{s,p}[\delta](\cD)}^p 
		&\leq \int_0^1 \int_{ B(\mathbf{0},1) } \frac{ \overline{C}_{d,p} |\bz|^p}{(1-  t\delta)^{p-sp+1}} \int_{\cD} \sigma(\bx)^{p-sp}  |\grad u(\bx)|^p  \, \rmd \bx  \, \rmd \bz \, \rmd t \\
		&\leq \frac{ \overline{C}_{d,p} }{(1-  \delta)^{p-sp+1}}  \int_0^1 \int_{ B(\mathbf{0},1) } |\bz|^p [u]_{W^{1,p}(\cD;p-sp)}^p  \, \rmd \bz \, \rmd t \\
		&= \frac{C(d,p)}{(1-\delta)^{p-sp+1}} [u]_{W^{1,p}(\Omega;p-sp)}^p\,.
	\end{split}
\end{equation*}
This completes the proof.
\end{proof}

We remark that $\frac{C(d,p)}{(1-  \delta)^{p-sp+1}} \leq \frac{C(d,p)}{(1-  \underline{\delta}_0)^{p-sp+1}} $ for all $\delta<\underline{\delta}_0$. Thus the constant in the estimate can be made independent of $\delta$. 

We close this subsection with a result about the embedding of the weighted  space $W^{s,p}(\cD; p-sp)$ introduced in \eqref{Def-weighted-space} into the fractional Sobolev space $W^{s,p}(\cD)$. 
\begin{theorem}\label{thm:WeightedEst}
Let $p \in (1,\infty)$ and $s \in (0,1]$, and let $\cD \subset \bbR^d$ be a bounded Lipschitz domain.  Then there exists a constant $C$ depending only on $d$, $p$ and the Lipschitz constant of $\cD$ such that for all $u \in W^{1,p}(\cD;p-sp)$ 
\begin{equation*}
	\kappa_{d,s,p} \int_{\cD} \int_{\cD} {|u(\by) - u(\bx)|^{p}\over|\by-\bx|^{d + sp} }\rmd\by \rmd\bx\leq C \int_{\cD} (\sigma(\bx))^{p-sp} |\grad u(\bx)|^p \, \rmd \bx\,.
\end{equation*} 
\end{theorem}

\begin{proof}
It suffices to show the inequality for $u \in C^{1}(\overline{\cD})$. By \cite{dyda2006comparability} there exists a constant $C = C(d,p,L)$, where $L$ is the Lipschitz constant of $\cD$, (the constant can be made independent of $s$) such that
\begin{equation*}
	[u]_{W^{s,p}(\cD)}^p \leq C \int_{\cD} \int_{B(\bx,\frac{1}{6} \sigma(\bx)) } \frac{ |u(\bx)-u(\by)|^p }{ |\bx-\by|^{d+sp} } \, \rmd \by \, \rmd \bx.
\end{equation*}
Then we have
\begin{equation*}
	\begin{split}
		[u]_{W^{s,p}(\cD)}^p &\leq C \int_{\cD} \int_{ B(\bx,\frac{1}{6} \sigma(\bx)) } \int_0^1 |\grad u(\bx + t(\by-\bx))|^p { |\bx-\by|^{-d-sp+p} } \, \rmd \by \, \rmd \bx \\
		&= C \int_{\cD} \int_{ B(0,1) } \int_0^1 |\grad u(\bx + t \sigma(\bx) \bz/6)|^p |\bz|^{-d-sp+p} \sigma(\bx)^{p-sp} \, \rmd t \, \rmd \bz \, \rmd \bx \\
		&= C \int_0^1 \int_{ B(0,1) }  \int_{\cD}  |\grad u(\bx + t \sigma(\bx) \bz/6)|^p \frac{ \sigma(\bx)^{p-sp} }{ |\bz|^{d+sp-p} } \, \rmd \bx  \, \rmd \bz \, \rmd t\,.
	\end{split}
\end{equation*}
Since $\sigma(\bx) \leq \sigma(\bx+t\sigma(\bx) \bz / 6) + t/6 \sigma(\bx)$, we have 
\begin{align*}
	&[u]_{W^{s,p}(\cD)}^p\\
	&\leq C (5/6)^{sp-p} \int_0^1 \int_{ B(0,1) }  \int_{\cD}  |\grad u(\bszeta({\bf x}, t, {\bf z}))|^p \sigma(\bszeta({\bf x}, t, {\bf z}))^{p-sp} \, \rmd \bx |\bz|^{-d-sp+p}  \, \rmd \bz \, \rmd t\,,
\end{align*}
where $\bszeta({\bf x}, t, {\bf z}) = \bx + t \sigma(\bx) \bz/6$. 
Changing coordinates in $\bx$, and then integrating out $\bz$ and $t$ gives the result.
\end{proof}

\section{Boundary-localized mollifier and convolution operator}\label{sec:convolution}

In this section, we review definitions and main properties of the {\it heterogeneous localization function, boundary-localized mollifier, and boundary-localized convolution operator} as developed in sections 3 and 4 of \cite{scott2024nonlocal} for studying the case $s=1$. Along the way, we present extensions of these results that will remain valid for any $s\in (0, 1]$. This analysis will be crucial in establishing properties of the nonlocal function spaces in the next section.

\subsection{Boundary-localized mollifier}\label{sec:blm}
To define the seminorm $[u]_{\mathfrak{W}^{s,p}[\delta](\mathcal{D})}$, the heterogeneous localization function used is $\delta \,\sigma(\bx)$.  In what follows we define a generalized heterogeneous localization function based on  generalized distance function $\eta : \overline{\mathcal{D}} \to [0,\infty)$ 
whose existence for any domain $\mathcal{D}$ is shown in \cite{Stein1970}. It has the following properties:
\begin{equation} \label{assump:Localization}
	\begin{aligned}
		& i) \, \text{ there exists a constant } \kappa_0 \geq 1 \text{ such that } \\
		&\qquad \frac{1}{\kappa_0} \dist(\bx,\p \mathcal{D}) \leq \eta(\bx) \leq \kappa_0 \dist(\bx,\p \mathcal{D}) \text{ for all } \bx \in \overline{\cD},\\
		& ii) \, \eta \in C^0(\overline{\cD}) \cap C^{\infty}(\mathcal{D}), \text{ and for each multi-index } \beta \in \bbN^d_0\,, \\
		&\qquad |D^\beta \eta(\bx)| \leq \kappa_{\beta} \dist(\bx,\p \mathcal{D})^{1-|\beta|}  \text{ for all } \bx \in \mathcal{D}, \text{ for constants } \kappa_{\beta} > 0.
	\end{aligned} \tag{\ensuremath{\rmA_{\eta}}}
\end{equation}
The constants $\kappa_\beta$ for all $|\beta| \geq 0$ can be chosen to depend only on $d$ and the derivatives of a standard mollifier given on $B(\mathbf{0},1)$; see \cite{Stein1970}. In particular, they are independent of $\cD$. Notice that $\eta(\bx) = \sigma(\bx)$ satisfies \eqref{assump:Localization} for $|\beta| \leq 1$.

To define a smooth boundary-localized dilation, we first introduce  a standard radial mollifier $\psi(|\bx|) \in C^{\infty}(\bbR^d)$, where $\psi: \bbR \to [0,\infty)$ is an even function that satisfies 
\begin{equation}  
\label{Assump:Kernel}
\left\{\begin{aligned}
	&\psi \in C^{\infty}(\bbR)\,,
  \quad [-c_\psi,c_{\psi}] \subset \supp \psi \Subset (-1,1) \text{ for fixed } c_{\psi} \in (0,1)\,,
	 \\  &	\quad \psi \geq 0\,,\quad \text{and}\quad
	\int_{\bbR^d} \psi(|\bx|) \, \rmd \bx = 1\,. 
    \end{aligned}\right.
\end{equation}
The boundary-localized dilation is then given by 
\begin{equation}\label{eq:OperatorKernelDef}
	\begin{split}
		\psi_{\delta}(\bx,\by) &:= \frac{1}{(\delta \eta(\bx))^{d}} {\psi} \left( \frac{|\by-\bx|}{\delta \eta(\bx)} \right) \,.
	\end{split}
\end{equation}
It then follows from \eqref{Assump:Kernel} that $\int_{\cD} \psi_\delta(\bx,\by) \, \rmd \by = 1$ for all $\bx \in \cD$ and for all $\delta < \underline{\delta}_0$, where the threshold $\underline{\delta}_0$ depends on $\kappa_0, \kappa_1$, and will be taken as 
\begin{equation}\label{bound-for-delta}
	\underline{\delta}_0 := \frac{1}{3} \min \left\{  1/\kappa_0 \,, 1/\kappa_1 \right\}
\end{equation} so that ${\delta}\eta(\bx) \leq \frac{1}{3}$ and $\delta|\grad \eta(\bx)| \leq \frac{1}{3}$ for any $\delta < \underline{\delta}_0$. The choice of this particular threshold is influenced by the estimate that we will establish in the sequel. 

The function $\psi_\delta(\bx, \by)$ can be viewed as a mollified version of 
$\frac{\mathds{1}_{B(\bx, \delta \sigma(\bx))} (\by)}{(\delta \sigma(\bx))^d}$ which is precisely a boundary-localized dilation of the unit characteristic function $\mathds{1}_{B_1(0)}(\bx)$ used in the seminorm of $\mathfrak{W}^{s,p}[\delta](\mathcal{D})$. 
It turns out we can use $\psi_\delta(\bx, \by)$ to define a new seminorm on $\mathfrak{W}^{s,p}[\delta](\mathcal{D})$ that will be  equivalent to the seminorm on $\mathfrak{W}^{s,p}[\delta](\mathcal{D})$. The following proposition establishes this equivalence.

\begin{proposition}\label{thm:EnergySpaceIndepOfKernel}
	For $p \in (1,\infty)$ and $s \in (0,1]$, there exist positive constants $C = C(d, p, \psi, \kappa_0)$ and $c = c(d, p, \psi, \kappa_0, c_\psi)$  
	such that the inequalities
	\begin{equation*}
		c\, [u]_{\mathfrak{W}^{s,p}[\delta](\mathcal{D})}^p \leq    \intdm{\mathcal{D}}{ \int_{\mathcal{D}} \psi_{\delta}(\bx,\by)  \frac{\big| u(\bx)-u(\by) \big|^p}{ \delta^{p} \eta(\bx)^{sp} }  \, \rmd \by}{\bx} \leq C [u]_{\mathfrak{W}^{s,p}[\delta](\mathcal{D})}^p \,
	\end{equation*}
	hold for all $u \in \mathfrak{W}^{s,p}[\delta](\mathcal{D})$.
\end{proposition}

\begin{proof}
	First, it follows from \eqref{assump:Localization} and \eqref{Assump:Kernel} that, there exists $C(\psi)>0$ such that 
	\begin{gather*}
		\frac{  \sigma(\bx) }{\kappa_0} \leq \eta(\bx) \leq \kappa_0  \sigma(\bx)\,, \qquad C(\psi)^{-1} \mathds{1}_{ B(\mathbf{0},{c_\psi\over 2}) }(\bx) \leq \psi(|\bx|) \leq C(\psi) \mathds{1}_{ B(\mathbf{0},1) }(\bx)\,.
	\end{gather*}
	Therefore,  from the assumptions on $\delta$ and $\kappa_0 \geq 1$ and \Cref{thm:InvariantHorizon}  we have that
	\begin{equation*}
		\begin{split}
			\int_{\mathcal{D}}\int_{\mathcal{D}} &\psi_{\delta}(\bx,\by)  \frac{\big| u(\bx)-u(\by) \big|^p}{ \delta^p \eta(\bx)^{sp} }  \, \rmd \by d{\bx}\\
			&\leq C(\psi) \kappa_0^{d+sp} \int_{\mathcal{D}} \int_{ \{|\by-\bx| < \kappa_0 \delta\,\sigma(\bx) \} } \frac{|u(\bx)-u(\by)|^p}{ \delta^{d+p} \sigma(\bx)^{d+sp} } \, \rmd \by \, \rmd \bx \\
			&\leq C(\psi,d,p,\kappa_0) [u]_{ \mathfrak{W}^{s,p}[\delta](\cD) }^p.
    \end{split}
	\end{equation*}
	Conversely, we also have  
	\begin{equation*}
		\begin{split}
			[u]_{ \mathfrak{W}^{s,p}[\delta](\mathcal{D}) }^p\,  &\leq  \frac{ \overline{C}_{d,p} }{C(\psi) } \int_{\mathcal{D}} \int_{ \{|\by-\bx| < \frac{c_\psi}{2\kappa_0} \delta \sigma(\bx) \} } \frac{|u(\bx)-u(\by)|^p}{ \kappa_0^{d+sp} \delta^{d+p} \sigma(\bx)^{d+sp} } \, \rmd \by \, \rmd \bx\\ & \leq \tilde{C} \int_{\mathcal{D}}\int_{\mathcal{D}} \psi_{\delta}(\bx,\by)  \frac{\big| u(\bx)-u(\by) \big|^p}{ \delta^p \eta(\bx)^{sp} }  \, \rmd \by\, \rmd \bx\,.
		\end{split}
	\end{equation*}
	We then take $c = {1\over \tilde{C}}.$
\end{proof}
Henceforth, we use both seminorms on $ \mathfrak{W}^{s,p}[\delta](\mathcal{D}) $ interchangeably.

As a composition of two differentiable functions, the boundary-localized dilation function $\psi_\delta(\bx, \by)$ is differentiable in $\bx$ and by direct computation we verify that  
\begin{equation*}
	\begin{split}
		\grad_{\bx} {\psi}_{\delta}(\bx,\by) &= {(\psi')_\delta(\bx,\by)\over \delta\eta(\bx)} \left( \frac{\bx-\by}{|\bx-\by|} - |\bx-\by| \frac{\grad \eta(\bx)}{\eta(\bx)} \right) - d \psi_{\delta}(\bx,\by)  \frac{ \grad \eta(\bx) }{\eta(\bx)}, 
	\end{split}
\end{equation*}
from which it follows from the properties of $\eta$ and the bounds of $\delta$  that  for some constant $C=C(\psi, d)$
\begin{equation}\label{eq:KernelDerivativeEstimate}
	|\grad _{\bx} \psi_{\delta}(\bx,\by)| \leq \frac{C}{\delta\eta(\bx) } \Big( \psi_{\delta}(\bx,\by) 
	+ (|\psi'|)_{\delta}(\bx,\by) \Big)\,,  \quad \forall \bx, \by \in\mathcal{D}.
\end{equation}
It is important to note that $\psi_\delta$ is not a symmetric function, i.e. $\psi_\delta(\bx, \by) \neq \psi_{\delta}(\by,\bx)$. For $\bx\in \mathcal{D},$ we  define 
\[
\Psi_{\delta}(\bx):= \int_{\mathcal{D}} \psi_{\delta}(\by,\bx) \rmd\by.     \]
Noting that  for $\bx, \by \in \mathcal{D}$ such that $|\bx-\by|\leq \delta \eta(\bx)$, we have 
\begin{equation}(1- \kappa_1 \delta ) \delta \eta(\bx) \leq\delta \eta(\by) \leq (1+\kappa_1 \delta ) \delta \eta(\bx),
	\label{eq:ComparabilityOfDistanceFxn2}
\end{equation}
and so we obtain that for all $\delta\leq \underline{\delta}_0$,
\begin{equation} \label{eq:KernelIntFunction:UpperBound:SqDist}
\|\Psi_\delta\|_{L^{\infty}(\mathcal{D})} \leq C
\end{equation}
with a bound $C$ depending only on $\psi$, $d$, and $\kappa_1$ (independent of $\delta$).   
Combining \eqref{eq:ComparabilityOfDistanceFxn2} and \eqref{eq:KernelDerivativeEstimate} we obtain that corresponding to $\alpha\in \mathbb{R}$, there exists $C = C(d, \psi, k_1, \alpha)$ such that for all $\delta \leq  \underline{\delta}_0$,
\begin{equation}\label{eq:KernelIntegralDerivativeEstimate}
	\int_{\mathcal{D}} |\delta \eta(\by)|^{\alpha} |\grad_{\bx} \psi_{\delta}(\bx,\by)| \, \rmd \by
	\leq \frac{C}{(\delta \eta(\bx))^{1-\alpha}}\,, \quad \forall \, \bx \in \mathcal{D}.
\end{equation}
Applying similar calculations to $\psi_{\delta}(\by,\bx)$, where $\bx, \by $ are interchanged, we obtain that for all $\bx, \by \in \mathcal{D}$ 
\begin{equation*}
	|\grad _{\bx} \psi_{\delta}(\by,\bx)| \leq \frac{1}{\delta \eta(\by) }  (|\psi'|)_{\delta}(\by,\bx) \,,  
\end{equation*} 
and that for any $\alpha \in \mathbb{R}$,  there exists $C = C(d,\psi,\kappa_1,\alpha)$  such that
\begin{equation*}
	\int_{\mathcal{D}} |\delta \eta(\by)|^\alpha |\grad_{\bx} \psi_{\delta}(\by,\bx) | \, \rmd \by \leq \frac{C}{(\delta \eta(\bx))^{1-\alpha}} \,, \quad \forall \, \bx \in \mathcal{D},
\end{equation*}
which follows from combining \eqref{eq:ComparabilityOfDistanceFxn2}, \eqref{eq:KernelDerivativeEstimate}, and \eqref{eq:KernelIntFunction:UpperBound:SqDist}. 
\subsection{Boundary-localized convolution operator}
For a boundary-localized mollifier, $\psi_\delta$, we define the boundary-localized convolution operator $K_{\delta}$ by 
\begin{equation*}
	K_{\delta}u (\bx) := \int_{\mathcal{D}} \psi_{\delta}(\bx,\by) u(\by) \, \rmd \by\,
    \quad \forall\, \bx \in \mathcal{D}.
\end{equation*}
It is clear from the properties of $\eta$ and $\psi$ that $K_{\delta} u \in C^{\infty}(\cD)$ for any $u \in L^1_{loc}(\cD)$. After a change of variables $\bz = {\bx-\by\over \delta \eta(\bx)}$, we may write, for $\bx\in \cD$, that
\begin{equation}\label{eq:conv-expression}
	K_{\delta}u (\bx)= \frac{1}{(\delta \eta(\bx))^{d}} \int_{\mathcal{D}}  \psi\left( \frac{|\by-\bx|}{\delta \eta(\bx)} \right) u(\by) \, \rmd \by = \int_{B(\mathbf{0},1)} \psi(|\bz|) u(\bx - \delta \eta(\bx) \bz) \, \rmd \bz.
\end{equation}
For $u\in C^0(\overline{\cD})$, the above formulation allows us to continuously extend $K_{\delta}u$ up to the boundary. Indeed, for any $\bx\in \partial \cD$ and any sequence $\bx_n\in \cD$ such that $\bx_n \to \bx$,  we have the limit 
\[ \lim_{n\to \infty}K_{\delta}u (\bx_n)=\int_{B(\mathbf{0},1)} \psi(|\bz|) u(\bx)\rmd\bz = u(\bx).
\]
Thus, for $u\in C^0(\overline{\cD})$ defining $K_{\delta}u (\bx) = u(\bx)$ for $\bx\in \partial \cD$ implies that $K_\delta u\in C^0(\overline{\cD})$. 
The following proposition summarizes this and a list of other important properties that are proved in \cite{scott2024nonlocal}.  The results are analogues of well-known results for the standard convolution operator. 

\begin{proposition}\label{continuity-of-Kdelta}
    For any $\delta < \underline{\delta}_0$, we have the following:
	\begin{enumerate}[\upshape 1)]
		\item For any $u \in L^1_{loc}(\cD)$, $K_{\delta} u \in C^{\infty}(\cD)$ with derivatives given by 
		\[
		D_{\bx}^{\ell}K_{\delta} u (\bx) = \intdm{\cD}{D^{\ell}_{\bx} \psi_{\delta}(\bx,\by) u(\by)}{\by}, \quad \ell\in \bbN_0^{d}.  
		\]
		If $u$ has compact support, so does $K_{\delta} u$.
		
		\item If $u \in C(\overline{\cD})$, so is $K_{\delta} u$,  and by definition $K_{\delta} u(\bx) = u (\bx)$,  $\forall \bx \in \p \cD$.
		
		\item For $1\leq p\leq \infty$ and $u\in L^{p}(\cD)$, we have $K_{\delta}u\in L^{p}(\cD)$ with the estimate 
		\[
		\|K_{\delta}u\|_{L^{p}(\cD)} \leq C_0 \|u\|_{L^{p}(\cD)}.
		\]
		
		\item For $1\leq p< \infty$ and $u\in W^{1, p}(\cD)$, we have $K_{\delta}u\in W^{1, p}(\cD)$ with the estimate 
		\[
		\|\nabla K_{\delta}u\|_{L^{p}(\cD)} \leq C_1 \|u\|_{W^{1, p}(\cD)}.
		\]
	\end{enumerate}
	The constants $C_0$ and $C_1$ depend only on $d, p, \psi$, and $\kappa_1$.  
\end{proposition}
In connection with function spaces of interest in this paper, we can prove similar continuity results. We recall the weighted Sobolev space $W^{1,p}(\cD;p-sp)$ introduced previously in \eqref{Def-weighted-space} as 
\[
W^{1,p}(\cD;p-sp) = \left\{ u \in L^p(\cD) \, :\, [u]^{p}_{W^{1,p}(\cD,p-sp)}:= \int_{\cD} |\grad u(\bx)|^p \sigma(\bx)^{p-sp} \, \rmd \bx < \infty \right\}. 
\]

\begin{theorem}\label{thm:Convolution:DerivativeEstimate}
	Given $p \in [1,\infty)$ and $s \in (0,1]$. Then for any $u \in \mathfrak{W}^{s,p}[\delta](\cD)$, $K_{\delta} u\in W^{1,p}(\cD;p-sp)$. Moreover, there exists $C > 0$ depending only on $d$, $\psi$, $p$,  $\kappa_0$ and $\kappa_1$ such that 
	\begin{equation}\label{eq:ConvEst:Deriv}
		\begin{split}
			[ K_{\delta} u ]_{W^{1,p}(\cD;p-sp)}^p
			\leq
			C 
			[u]_{\mathfrak{W}^{s,p}[\delta](\cD)}^p\,, \qquad \forall u \in \mathfrak{W}^{s,p}[\delta](\cD).
		\end{split}
	\end{equation}
\end{theorem}

\begin{proof}
	Fix $u \in \mathfrak{W}^{s,p}[\delta](\cD)$.  
	Since $\int_{\cD} \psi_\delta(\bx,\by) \, \rmd \by = 1$ for all $\bx \in \cD$, its $\bx$-gradient is zero, and so we have
	\begin{equation*}
		\begin{split}
			\grad K_{\delta}u (\bx) &= \intdm{\cD}{\grad_{\bx} \psi_{\delta}(\bx,\by) u(\by)}{\by}  = \intdm{\cD}{\grad_{\bx} \psi_{\delta}(\bx,\by) (u(\by)-u(\bx))}{\by}\,.
		\end{split}
	\end{equation*}
	Therefore, by H\"older's inequality 
	\begin{equation*}
		\begin{split}
			&\Vnorm{\eta^{1-s} \grad K_{\delta} u }_{L^p(\cD)}^p \\
			&\leq \int_{\cD}  \left( \delta\eta(\bx)  \intdm{\cD}{ |\grad_{\bx} \psi_{\delta}(\bx,\bz)|}{\bz}
			\right)^{p-1}
			\intdm{\cD}{\frac{|\grad_{\bx} \psi_{\delta}(\bx,\by)|}{\delta^{p-sp}  (\delta\eta(\bx))^{sp-1}} |u(\by)-u(\bx)|^p}{\by} \rmd \bx. 
		\end{split}
	\end{equation*}
	We use \eqref{eq:KernelIntegralDerivativeEstimate} with $\alpha = 1$
	to get
	\begin{equation}\label{eq:ConvEst:Deriv:Pf1}
		\begin{split}
			\Vnorm{\eta^{1-s}\grad K_{\delta} u}_{L^p(\cD)}^p &\leq C \int_{\cD} \intdm{\cD}{\frac{|\grad_{\bx} \psi_{\delta}(\bx,\by)|}{\delta^{p-sp}  (\delta\eta(\bx))^{sp-1} } |u(\by)-u(\bx)|^p}{\by}  \, \rmd \bx\,.
		\end{split}
	\end{equation}
	Using the inequality in \eqref{eq:KernelDerivativeEstimate} we have
	\begin{equation*}
		\frac{|\grad_{\bx} \psi_{\delta}(\bx,\by)|}{ (\delta\eta(\bx))^{sp-1}  } \leq C \frac{ (|\psi'|)_{\delta}(\bx,\by) + \psi_{\delta}(\bx,\by) }{  (\delta\eta(\bx))^{sp} }\,.
	\end{equation*}
	Inequality \eqref{eq:ConvEst:Deriv} now follows after using the above inequality in \eqref{eq:ConvEst:Deriv:Pf1} and then applying \Cref{thm:EnergySpaceIndepOfKernel} and the subsequent remark.
\end{proof}

\begin{corollary}\label{rmk:contKfraktofrac}
	Let $p\in (1, \infty),$   $s\in (0,1]$,  $\delta\in (0, \underline{\delta}_0)$ and $\cD \subset \bbR^d$ be a bounded Lipschitz domain. Then we have the following. 
	
	\begin{enumerate}[\upshape 1)]
		\item There exists a constant $C>0$ depending only on $d$, $p$, $\psi$, $\kappa_0$ and $\kappa_1$ such that
		\begin{equation*}
			[ K_{\delta} u ]_{W^{1,p}(\cD;p-sp)}
			\leq
			C
			[u]_{W^{1,p}(\cD;p-sp)}\,, \qquad \forall u \in W^{1,p}(\cD;p-sp),
		\end{equation*}
		
		\item There exists a constant $C>0$ depending only on $d$, $p$, $\psi$, $\kappa_0$, $\kappa_1$ and the Lipschitz constant of $\cD$ such that
		\begin{equation}\label{eq:ConvEst:Deriv:s:Opt}
			\begin{split}
				[ K_{\delta} u ]_{W^{s,p}(\cD)}^p
				\leq
				C [u]_{\mathfrak{W}^{s,p}[\delta](\cD)}^p\,, \qquad \forall u \in \mathfrak{W}^{s,p}[\delta](\cD).
			\end{split}
		\end{equation}
		
		\item There exists a constant $C>0$ depending only on $d$, $p$, $\psi$, and $\kappa_1$ such that
		\begin{equation}\label{eq:ConvEst1:Wsp}
		\Vnorm{K_{\delta} u}_{W^{s,p}(\cD)} \leq C \Vnorm{u}_{W^{s,p}(\cD)} \qquad \forall u \in W^{s,p}(\cD).
		\end{equation}
	\end{enumerate}
\end{corollary}

\begin{proof}
	Item 1) follows from \Cref{thm:Convolution:DerivativeEstimate} and \Cref{cor:Embedding:Weighted}.
	
	Item 2) follows from the previous item and the embedding $W^{1,p}(\cD;p-sp) \subset W^{s,p}(\cD)$ which is stated in \Cref{thm:WeightedEst}.
	
	Item 3) can be proved by interpolation, since $W^{s,p}$ is obtained by interpolating $L^{p}(\cD)$ and $W^{1,p}(\cD)$ and $K_\delta$ is a bounded linear operator on $L^{p}(\cD)$ and $W^{1,p}(\cD)$ by \Cref{continuity-of-Kdelta}. We may also use a direct proof which is straightforward, using a change of variables. 
	To be precise, we can show that for $1\leq p< \infty$
	\[
	\|K_{\delta}u\|_{W^{s,p}(\cD)} \leq C(d,p,\delta,s,\kappa_1) \|u\|_{W^{s,p}(\cD)},\quad\forall u\in W^{s,p}(\cD),
	\]
	where $C(d,p,\delta,s,\kappa_1)=\frac{(1+\kappa_1\delta)^{d/p+s}}{(1-\kappa_1\delta)^{2/p}}$, which can be made uniform in $\delta\in (0,\underline{\delta}_0)$ and $s\in (0,1]$ by taking $C = \frac{(1+\kappa_1\underline{\delta}_0)^{d/p+1}}{(1-\kappa_1\underline{\delta}_0)^{2/p}}$. 
	Indeed, first note that 
		\[
		K_\delta u(\bx) - K_\delta u(\by) = \int_{B(\mathbf{0}, 1)} \psi(\bz) (u(\bx - \delta \eta(\bx) \bz) - u(\by - \delta \eta(\by) \bz)) \rmd \bz.
		\]
		Using the change of variables $\bx \mapsto \bszeta^{-\delta}_{\bz}(\bx):=\bx - \delta \eta(\bx) \bz$ and $\by\mapsto \bszeta^{-\delta}_{\bz}(\by)=\by - \delta \eta(\by) \bz$ and inequalities from \cite[Lemma~3.2]{scott2024nonlocal}, namely that $\mathrm{det}(\nabla \bszeta^{-\delta}_{\bz}(\bx))\ge 1-\kappa_1 \delta$ and $|\bszeta^{-\delta}_{\bz}(\bx)- \bszeta^{-\delta}_{\bz}(\by)|\le (1+\kappa_1\delta)|\bx-\by|$, one obtains
		\begin{align*}
			[K_{\delta}u]_{W^{s,p}(\cD)} ^p		&=\kappa_{d,s,p}\int_{\cD}\int_{\cD} \frac{\left|K_\delta u(\bx) - K_\delta u(\by)\right|^p}{|\bx - \by|^{d+sp}} \rmd \bx \rmd \by\\
			&\le \kappa_{d,s,p}\int_{B(\mathbf{0}, 1)} \psi(\bz) \int_{\cD}\int_{\cD}\frac{\left|u(\bx - \delta \eta(\bx) \bz) - u(\by - \delta \eta(\by) \bz)\right|^p}{|\bx - \by|^{d+sp}} \rmd \bx \rmd \by \rmd \bz\\
			&\le \kappa_{d,s,p}\frac{(1+\kappa_1\delta)^{d+sp}}{(1-\kappa_1\delta)^2} \int_{B(\mathbf{0}, 1)} \hspace{-.2cm}  \psi(\bz) \int_{\cD}\int_{\cD}\hspace{-.1cm} \frac{|u(\bszeta^{-\delta}_{\bz}(\bx)) - u(\bszeta^{-\delta}_{\bz}(\by))|^p}{|\bszeta^{-\delta}_{\bz}(\bx) - \bszeta^{-\delta}_{\bz}(\by)|^{d+sp}}\\
            &\quad\quad\quad \times \mathrm{det}(\nabla \bszeta^{-\delta}_{\bz}(\bx))\mathrm{det}(\nabla \bszeta^{-\delta}_{\bz}(\by))\rmd \bx \rmd \by \rmd \bz\\
			&\le \kappa_{d,s,p}\frac{(1+\kappa_1\delta)^{d+sp}}{(1-\kappa_1\delta)^2} \int_{B(\mathbf{0}, 1)} \hspace{-.2cm} \psi(\bz) \int_{\cD}\int_{\cD}\hspace{-.1cm} \frac{|u(\bx) - u(\by)|^p}{|\bx - \by|^{d+sp}} \rmd \bx \rmd \by \rmd \bz\\
			&\le \frac{(1+\kappa_1\delta)^{d+sp}}{(1-\kappa_1\delta)^2} [u]_{W^{s,p}(\cD)} ^p.
		\end{align*}
\end{proof}

\subsection{Behavior of \texorpdfstring{$K_\delta$}{Kdelta} with respect to \texorpdfstring{$\delta$}{delta}}
We begin with a summary of results that are proved in \cite{scott2024nonlocal} in connection with the localizing effect of $K_\delta$  as $\delta \to0$ in classical spaces.  
\begin{proposition}\label{prop:summary-conv-properties}
	Let $p \in (1,\infty)$. Then the following hold:
	\begin{enumerate}[\upshape 1)]
		\item If $u \in C^0(\overline{\cD})$, then $K_{\delta} u \to u$ uniformly on $\overline{\cD}$ as $\delta \to 0$. 
		
		\item If $u\in L^{p}(\cD),$ then $K_\delta u \to u$ in $L^{p}(\cD)$ as $\delta \to 0$.
		
		\item If $u\in W^{1,p}(\cD),$ then $K_\delta u \to u$ in  $W^{1,p}(\cD)$ as $\delta \to 0$.
	\end{enumerate}
\end{proposition}
The same localizing effect can be established in the fractional and weighted Sobolev spaces with respect to their respective norms, as demonstrated below. 

\begin{proposition}
\label{prop:conv-convergence-Sobolev-spaces}
	Fix $p\in (1, \infty)$ and  $s\in (0, 1].$ The following hold:
\begin{equation}\label{eq:ConvergenceOfConv:H1}
\lim\limits_{\delta \to 0} \Vnorm{K_{\delta} u - u}_{W^{s,p}(\cD)} = 0, \quad\forall\,  u \in W^{s,p}(\cD).
	\end{equation}
\begin{equation}\label{eq:ConvergenceOfConv:H1Weighted}
\lim\limits_{\delta \to 0} \Vnorm{K_{\delta} u - u}_{W^{1,p}(\cD;p-sp)} = 0,\quad\forall\,  u \in W^{1,p}(\cD;p-sp). 
	\end{equation}
\end{proposition}

\begin{proof} As the case $s=1$ has been covered by \Cref{prop:summary-conv-properties} item 3, we shall focus on $s\in (0, 1)$. We start by proving \eqref{eq:ConvergenceOfConv:H1}.
	Thanks to the density of smooth functions in $W^{s,p}(\cD)$ and the boundedness of $K_\delta$ on $W^{s,p}(\cD)$ proved in \Cref{rmk:contKfraktofrac}, it suffices to show that for $u\in C^\infty(\overline{\cD})$, 
	\[
	|K_\delta u - u|_{W^{s,p}(\cD)}\to 0, \quad \delta\to 0.
	\]
	By \eqref{eq:conv-expression}, for $\bx, \by\in \cD$, 
	\begin{equation*}
		\begin{aligned}
			&\ K_\delta u(\bx) - u(\bx) - (K_\delta u(\by) - u(\by)) \\ 
			=& \int_{B(\mathbf{0}, 1)} \psi(|\bz|)[u(\bx-\delta\eta(\bx)\bz)- u(\bx) - (u(\by-\delta\eta(\by)\bz)- u(\by))]\rmd\bz 
		\end{aligned}
	\end{equation*}
	Using Taylor expansion we have
	\begin{gather*}
		u(\bx)-u(\by) = \int_0^1 \nabla u(\by+t(\bx-\by)) \cdot (\bx-\by)\rmd t,\\
		u(\bx-\delta\eta(\bx)\bz) - u(\by - \delta \eta(\by) \bz) = \int_0^1 \nabla u(\by - \delta \eta(\by) \bz + t(\bx-\by)) \cdot (\bx - \by) \rmd t.
	\end{gather*}
	Thus, using the fact that $\eta(\bx) \le C(\cD, \kappa_0, \kappa_1)$ uniformly in $\bx\in \cD$, one derives 
	\begin{equation*}
		\begin{aligned}
			&\ \left|u(\bx-\delta\eta(\bx)\bz)- u(\bx) - (u(\by-\delta\eta(\by)\bz)- u(\by))\right| \\
			= & \left| \int_0^1 [\nabla u(\by+t(\bx - \by) -\delta \eta(\by)\bz) - \nabla u(\by + t(\bx - \by))] \cdot (\bx - \by)\rmd t\right|\\
			\le &\int_0^1 \int_0^1 \left| \nabla^2 u(\by + t(\bx -\by)-t'\eta(\by)\bz)\delta \eta(\by)\bz \right| \rmd t' |\bx-\by|\rmd t\\
			\le & C(\|u\|_{C^2(\cD)}, \kappa_0, \kappa_1, \cD)\delta |\bx - \by|.
		\end{aligned}
	\end{equation*}
	Hence by Jensens' inequality, as $\delta\to 0$,
	\begin{align*}
		& [K_\delta u - u]_{W^{s,p}(\cD)}
		\le  C(\|u\|_{C^2(\cD)}, \kappa_0, \kappa_1, \cD)\delta \left(\int_{\cD}\int_{\cD}\frac{1}{|\bx-\by|^{d+sp-p}}\rmd\by \rmd\bx\right)^{1/p}\to 0.
	\end{align*}
	Thus \eqref{eq:ConvergenceOfConv:H1} is established.
	
	To prove \eqref{eq:ConvergenceOfConv:H1Weighted}, thanks to the density result \Cref{thm:FxnSpProp:Weighted} and the boundedness of $K_\delta$ on $W^{s,p}(\cD; p-sp)$ proved in \Cref{rmk:contKfraktofrac}, it suffices to show for $u\in C^\infty(\overline{\cD})$, \[\int_{\cD} \left|\nabla K_\delta u(\bx)- \nabla u(\bx)\right|^p\sigma(\bx)^{p-sp}\rmd\bx \to 0 \quad \text{ as } \quad \delta\to 0.\]
	Denoting $\bszeta^{-\delta}_{\bz}(\bx):=\bx-\delta\eta(\bx)\bz$, 
	\[
	\nabla K_\delta u(\bx)=\int_{B(\mathbf{0}, 1)} \psi(\bz)\nabla_{\bx}[\bszeta^{-\delta}_{\bz}](\bx) \nabla u(\bszeta^{-\delta}_{\bz}(\bx))\rmd\bz.
	\]
	Let $\bA:=\nabla [\bszeta^{-\delta}_{\bz}](\bx)=\bI-\delta \bz\otimes \nabla\eta(\bx)$, where $\bI$ denotes the identity matrix and ``$\otimes$'' denotes the tensor product. Using that $\eta(\bx), |\nabla\eta(\bx)|\le C(\cD, \kappa_0, \kappa_1)$ for all $\bx\in \cD$, together with the Taylor expansion \[
	\nabla u(\bx-\delta\eta(\bx)\bz) = \nabla u(\bx)-\int_0^1 \nabla^2 u(\bx-t\delta \eta(\bx)\bz)\delta\eta(\bx)\bz \, \rmd t,
	\]  we obtain that
	\begin{align*}
		\left|\bA \nabla u(\bszeta^{-\delta}_{\bz}(\bx)) -\nabla u(\bx)\right|
		&= \left|\bA \left(\nabla u(\bszeta^{-\delta}_{\bz}(\bx)) - \nabla u(\bx)\right)+(\bA-\bI)\nabla u(\bx)\right|\\
		&\le C\int_0^1\left|\nabla^2 u(\bx-t\delta \eta(\bx)\bz)\delta\eta(\bx)\bz\right|\rmd t+C\delta|\bz||\nabla u(\bx)|\\
		&\le C(\|u\|_{C^2(\cD)}, \kappa_0, \kappa_1, \cD)\delta.
	\end{align*}
	Then by Jensen's inequality,
	\[
	\left|\nabla K_\delta u(\bx)- \nabla u(\bx)\right|^p\le
	\int_{B(\mathbf{0},1)}\psi(\bz)\left|\bA \nabla u(\bszeta^{-\delta}_{\bz}(\bx)) -\nabla u(\bx)\right|^p\rmd\bz\le C\delta^p.
	\]
	Note that $\sigma(\bx)$ is bounded since $|\cD| < \infty$, so we have 
	\[\int_{\cD} \left|\nabla K_\delta u(\bx)- \nabla u(\bx)\right|^p\sigma(\bx)^{p-sp} \rmd\bx \to 0 \quad \text{ as } \quad \delta\to 0.
	\]
	
\end{proof}

In the following proposition, we prove a rate for the convergence $K_\delta u\to u$ in $L^{p}(\cD)$ as $\delta\to 0$ for $u\in \mathfrak{W}^{s,p}[\delta](\cD)$.
\begin{proposition}\label{thm:diffuKu:Weighted}
	Let $p \in (1,\infty)$ and $s \in (0,1]$ be given, and let $\delta < \underline{\delta}_0$. Then there exists a constant $C$ depending only on $d$,  $p$, $\psi$, and $\kappa_0$ such that 
	\begin{equation}\label{eq:Kdeltaerror1:Weighted}
		\Vnorm{ (u - K_{\delta} u) \sigma^{-s} }_{L^p(\cD)} \leq C \delta [u]_{\mathfrak{W}^{s,p}[\delta](\cD)} \,, \qquad \forall u \in \mathfrak{W}^{s,p}[\delta](\cD).
	\end{equation}
	Consequently, there exists a constant $C$ depending only on $d$,  $p$, $\psi$, and $\kappa_0$ such that
\begin{equation}\label{eq:KdeltaError}
		\Vnorm{u - K_{\delta} u }_{L^p(\cD)} \leq 
		C \delta \diam(\cD)^s [u]_{\mathfrak{W}^{s,p}[\delta](\cD)}\,, \quad\forall u \in \mathfrak{W}^{s,p}[\delta](\cD).
	\end{equation}
\end{proposition}

\begin{proof}
	We have by H\"older's inequality
	\begin{equation*}
		\Vnorm{(K_\delta u - u) \sigma^{-s} }_{L^p(\cD)}^p
			\leq \int_{\cD} \frac{1}{\sigma(\bx)^{sp}} \int_{\cD} \psi_\delta(\bx,\by) |u(\bx)-u(\by)|^p \, \rmd \by \, \rmd \bx.
	\end{equation*}
    Then by the properties of $\eta$ and \Cref{thm:EnergySpaceIndepOfKernel}
	\[
		\begin{aligned}
		\Vnorm{(K_\delta u - u) \sigma^{-s} }_{L^p(\cD)}^p &\leq \delta^{p} \kappa_0^{sp}\int_{\cD} \frac{1}{\delta^p \eta(\bx)^{sp}} \int_{\cD} \psi_\delta(\bx,\by) |u(\bx)-u(\by)|^p \, \rmd \by \, \rmd \bx \\
			&\leq C \delta^{p}[u]_{\mathfrak{W}^{s,p}[\delta](\cD)}^p\,.
		\end{aligned}
	\]
    This concludes the proof.
\end{proof}

\subsection{Behavior of \texorpdfstring{$K_\veps$}{Kepsilon} with respect to \texorpdfstring
{$\delta$}{delta}}

We close this section with a discussion of the boundary-localized convolution $K_\veps$ with bulk horizon parameter $\veps$ decoupled from the horizon parameter $\delta$ defining the nonlocal space $\mathfrak{W}^{s,p}[\delta](\cD)$.
The following technical lemma 
will be used in the proof of the density of smooth functions.

\begin{lemma} \label{lem:with-phiepsilon}
	Let $p \in (1,\infty)$, $s \in (0,1]$, and $\delta \in (0,\underline{\delta}_0)$. Let $\varrho \in [0,\diam(\cD))$. Then there exist 
	continuous functions $\theta : [0,\underline{\delta}_0) \to (0,1]$ and $\phi : [0,\underline{\delta}_0) \to [1,\infty)$ with $\theta(0) = \phi(0) = 1$ determined only from $d$, $p$, $s$, $\kappa_0$ and $\kappa_1$ such that the following holds: 
	For all $\veps \in (0,\underline{\delta}_0)$ and $u \in \mathfrak{W}^{s,p}[\delta](\cD)$,
\begin{equation}\label{eq:Comp:ConvolutionEstimate1}
		\begin{split}
			\int_{\cD \cap \{ \sigma(\bx) < \varrho \} } & \int_{B(\bx, \theta(\veps)\delta \sigma(\bx))}  \frac{ |K_\veps u(\bx) - K_\veps u(\by)|^p }{ (\theta(\veps)\delta)^{d+p} \sigma(\bx)^{d+sp} } \, \rmd \by \, \rmd \bx \\
			&\leq \phi(\veps) \int_{\cD \cap \{ \sigma(\bx) < \phi(\veps) \varrho \} } \int_{B(\bx, \delta \sigma(\bx))} \frac{ |u(\bx) - u(\by)|^p }{ \delta^{d+p} \sigma(\bx)^{d+sp} } \, \rmd \by \, \rmd \bx,
		\end{split}
	\end{equation}
	and
\begin{equation}\label{eq:Comp:ConvolutionEstimate2}
		\begin{split}
			\int_{\cD \cap \{ \sigma(\bx) > \varrho \} }  &\int_{B(\bx, \theta(\veps)\delta \sigma(\bx))} \frac{ |K_\veps u(\bx) - K_\veps u(\by)|^p }{ (\theta(\veps)\delta)^{d+p} \sigma(\bx)^{d+sp} } \, \rmd \by \, \rmd \bx \\
			&\leq \phi(\veps) \int_{\cD \cap \{ \sigma(\bx) > \frac{ \varrho}{\phi(\veps)} \} } \int_{B(\bx, \delta \sigma(\bx))} \frac{ |u(\bx) - u(\by)|^p }{ \delta^{d+p} \sigma(\bx)^{d+sp} } \, \rmd \by \, \rmd \bx.
		\end{split}
	\end{equation}
\end{lemma}

\begin{proof}
	We write the proof of \eqref{eq:Comp:ConvolutionEstimate1}; the proof of \eqref{eq:Comp:ConvolutionEstimate2} is similar.
	Set $\theta(\veps) := \frac{1-\kappa_0 \veps}{1+ \kappa_1 \veps}$. For $\veps \in (0,\underline{\delta}_0)$ and $\bz \in B(\mathbf{0},1)$, define $\bszeta^\veps_\bz(\bx) := \bx + \veps \eta(\bx)\bz$. Then by Jensen's inequality
	\begin{equation*}
		|K_\veps u(\bx) - K_\veps u(\by)|^p \leq \int_{B(\mathbf{0},1)} \psi(|\bz|) |u(\bszeta^\veps_\bz(\bx)) - u(\bszeta^\veps_\bz(\by))|^p \, \rmd \bz.
	\end{equation*}
	Now, we apply the inequalities in \cite[Lemma 3.2]{scott2024nonlocal} to obtain
    
    \begingroup\makeatletter\def\f@size{9}\check@mathfonts
    \def\maketag@@@#1{\hbox{\m@th\large\normalfont#1}}%
        \begin{equation*}
			\begin{split}
				&\int_{\cD \cap \{ \sigma(\bx) < \varrho \} } \int_{B(\bx, \theta(\veps)\delta \sigma(\bx))} \frac{ |K_\veps u(\bx) - K_\veps u(\by)|^p }{ (\theta(\veps)\delta)^{d+p} \sigma(\bx)^{d+sp} }  \rmd \by \rmd \bx \\
				\leq& \tilde{\phi}(\veps)
				\int_{B(\mathbf{0},1)} \psi \left( |\bz| \right) \int_{\cD}  \int_{ \cD } \chi_\veps(\bx,\by,\bz) \frac{ \left|  u( \bszeta_{\bz}^\veps(\bx) ) - u( \bszeta_{\bz}^\veps(\by)  ) \right|^p }{\delta^{d+p} \sigma( \bszeta_{\bz}^\veps(\bx)  )^{d+sp}} \det \grad \bszeta_\bz^\veps(\bx) \det \grad \bszeta_\bz^\veps(\by) \, \rmd \by \, \rmd \bx \rmd \bz,
			\end{split}
	    \end{equation*}
    \endgroup
	where we denote
	\begin{equation*}
		\chi_{\veps}(\bx,\by,\bz) = \mathds{1}_{ \{ \sigma(\bszeta_\bz^\veps(\bx)) < (1 + \kappa_0 \veps) \varrho \} } 
		\mathds{1}_{  \{ |\bszeta_\bz^\veps(\by) - \bszeta_\bz^\veps(\bx)| \leq \delta \sigma( \bszeta_\bz^\veps(\bx) ) \} }
	\end{equation*}
	and the constant $\tilde{\phi}(\veps) :=  \frac{ (1+\kappa_0 \veps)^{d+sp} }{ \theta(\veps)^{d+p} (1-\kappa_1 \veps)^{2}}$.
	Now, apply the change of variables $\bar{\by} = \bszeta_{\bz}^\veps(\by) $ and $\bar{\bx} = \bszeta_{\bz}^\veps(\bx)$.
	Therefore we obtain 
	\begin{equation*}
		\begin{split}
			&\int_{\cD \cap \{ \sigma(\bx) < \varrho \} } \int_{B(\bx, \theta(\veps)\delta \sigma(\bx))} \frac{ |K_\veps u(\bx) - K_\veps u(\by)|^p }{ (\theta(\veps)\delta)^{d+p} \sigma(\bx)^{d+sp} } \, \rmd \by \, \rmd \bx \\
			&\leq \tilde{\phi}(\veps) \int_{B(\mathbf{0},1)} \psi \left( |\bz| \right) \int_{ 
				\cD \cap \{ \sigma(\bar{\bx}) < (1+\kappa_0 \veps) \varrho \} 
			} \int_{B(\bar{\bx}, \delta\sigma(\bar{\bx})) } \frac{ \left|  u(\bar{\bx}) - u(\bar{\by}) \right|^p }{\delta^{d+p} \sigma(\bar{\bx})^{d+sp}} \, \rmd \bar{\by} \, \rmd \bar{\bx} \, \rmd \bz\,.
		\end{split}
	\end{equation*} 
	We obtain \eqref{eq:Comp:ConvolutionEstimate1} by setting $\phi(\veps) := \max \{ \tilde{\phi}(\veps), \frac{1+\kappa_0 \veps}{1-\kappa_0 \veps} \}$. We conclude by noting that \eqref{eq:Comp:ConvolutionEstimate2} can be established by the same method of argument, with the same choice of $\phi(\veps)$.
\end{proof}

As a corollary, we obtain that $K_\veps$ is a bounded linear operator on the nonlocal space.

\begin{corollary}\label{cor:ConvEst:NonlocalSpace}
	Let $p \in (1,\infty)$, $s \in (0,1)$ and $\delta \in (0, \underline{\delta}_0)$. Then there exists a constant $C >0 $ depending only on $d, p,\underline{\delta}_0$, $\kappa_0$ and $\kappa_1$ such that for  all $\veps \in (0,\underline{\delta}_0)$,
	\begin{equation*}
		[K_\veps u]_{\mathfrak{W}^{s,p}[\delta](\cD)}  \leq C [u]_{\mathfrak{W}^{s,p}[\delta](\cD)},\quad \forall\, u \in \mathfrak{W}^{s,p}[\delta](\cD).
	\end{equation*}
\end{corollary}

\begin{proof}
	Noting that $ \lim\limits_{\veps \to 0} \theta(\veps) = 1$, we may 
	apply \Cref{thm:InvariantHorizon} and then \eqref{eq:Comp:ConvolutionEstimate2} with $\varrho = 0$ to conclude that 
	\begin{equation*}
		[K_\veps u]_{\mathfrak{W}^{s,p}[\delta](\cD)} \leq C [K_\veps  u]_{\mathfrak{W}^{s,p}[\theta(\veps)\delta](\cD)} \leq C [u]_{\mathfrak{W}^{s,p}[\delta](\cD)}
	\end{equation*}
    for all $u\in  \mathfrak{W}^{s,p}[\delta](\cD)$.
\end{proof}

\section{Properties of function spaces}
\label{sec:properties}
In addition to those 
presented in \Cref{sec:function-space}, additional
properties of the function spaces are established in this section with the help of the results in \Cref{sec:convolution}; these properties are essential for the analysis of the variational problem.
\subsection{Density of smooth functions in \texorpdfstring{$\mathfrak{W}^{s,p}[\delta](\mathcal{D})$}{the nonlocal space}}\label{sec:density}
The main focus of this subsection is to show that smooth functions are dense in $\mathfrak{W}^{s,p}[\delta](\mathcal{D})$. 
This gives us the option in subsequent discussions to demonstrate properties of $\mathfrak{W}^{s,p}[\delta](\mathcal{D})$ first for smooth functions, and then pass to the general case.

\begin{theorem}[Density of smooth functions]\label{thm:Density}
	Let $1 < p < \infty$, let $s \in (0,1]$, let $\mathcal{D} \subset \bbR^d$ be a bounded domain. Corresponding to any $u \in \mathfrak{W}^{s,p}[\delta](\mathcal{D})$,  there exists a sequence $\{u_n \} \subset C^{\infty}(\mathcal{D}) \cap \mathfrak{W}^{s,p}[\delta](\mathcal{D})$ such that $\Vnorm{u_n - u}_{\mathfrak{W}^{s,p}[\delta](\mathcal{D})} \to 0$ as $n \to \infty$. 
	In addition, if $\mathcal{D}$ is a Lipschitz domain then the sequence can be chosen to be in $C^{\infty}(\overline{\mathcal{D}})$.
\end{theorem}

\begin{corollary}\label{cor:Embedding:Weighted}
	\Cref{lma:Embedding:Weighted:Pre} holds for general $u \in W^{1,p}(\cD;p-sp)$.
\end{corollary}

Typically, to demonstrate the density of functions within a given function space, one approximates a specific function by its convolution with compactly supported smooth functions, and then uses the compatibility of the convolution operator with translation-invariant kernels to verify the smallness of the difference. However, in the case of $ \mathfrak{W}^{s,p}[\delta](\mathcal{D})$ this is not feasible. The main reason is that while the kernel $\gamma^{s, \delta}_{p, \mathcal{D}}$ associated with the seminorm is symmetric, it lacks translation invariance. Noting that the kernel $\gamma^{s, \delta}_{p, \mathcal{D}}$ is singular near boundary points, in \cite{scott2024nonlocal} a boundary-localized convolution operator is used to facilitate the construction of approximate smooth functions in the case $s=1$. The same technique can be used (thanks to the results of the previous section) to prove the density of smooth functions in the case of general $0<s\leq 1$.


\begin{proof} Fix $\delta\in (0, \underline{\delta}_0)$. 
	Define, for $0 < \veps <  \underline{\delta}_0$, 
	the sequence
	\begin{equation*}
		v_{\veps}(\bx) := K_{\veps} u(\bx)\,.
	\end{equation*}
	By the properties of $\psi$ and $\eta$ this function is in $C^{\infty}(\cD)$. 
	Moreover, $K_\veps u \in \mathfrak{W}^{s,p}[\delta](\cD)$ by \Cref{cor:ConvEst:NonlocalSpace}, and $K_{\veps} u \to u$ in $L^p(\cD)$ as $\veps \to 0$ by \Cref{prop:summary-conv-properties}. 
	To conclude the first density result we just need to show that $[K_\veps u - u]_{\mathfrak{W}^{s,p}[\delta](\cD)} \to 0$ as $\veps \to 0$. 
	To this end, we first use \Cref{thm:InvariantHorizon} to obtain
	\begin{equation*}
		[K_\veps u - u]_{\mathfrak{W}^{s,p}[\delta](\cD)} \leq C(d,s,p,\underline{\delta}_0) [K_\veps u - u]_{\mathfrak{W}^{s,p}[\theta(\veps)\delta](\cD)}.
	\end{equation*}
	Now we split the integral defining the seminorm; for $\varrho \in (0,\diam(\cD))$, we write
	\begin{equation*}
		\begin{split}
			&[K_\veps u - u]_{\mathfrak{W}^{s,p}[\theta(\veps)\delta](\cD)}^p \\
			&= \int_{\cD \cap \{ \sigma(\bx) > \varrho \} } \int_{B(\bx, \theta(\veps)\delta \sigma(\bx))} \frac{ |K_\veps u(\bx) - K_\veps u(\by) - (u(\bx)-u(\by))|^p }{ (\theta(\veps)\delta)^{d+p} \sigma(\bx)^{d+sp} } \, \rmd \by \, \rmd \bx \\
			&\quad + \int_{\cD \cap \{ \sigma(\bx) < \varrho \} } \int_{B(\bx, \theta(\veps)\delta \sigma(\bx))} \frac{ |K_\veps u(\bx) - K_\veps u(\by) - (u(\bx)-u(\by))|^p }{ (\theta(\veps)\delta)^{d+p} \sigma(\bx)^{d+sp} } \, \rmd \by \, \rmd \bx \\
			&:= \textrm{I}_{\veps,\varrho} + \textrm{II}_{\veps,\varrho}.
		\end{split}
	\end{equation*}
	To estimate $\textrm{I}_{\veps,\varrho}$, we use the inequalities
	$$\mathds{1}_{ \{ |\bx-\by| < \theta(\veps) \delta \sigma(\bx) \} } \mathds{1}_{ \{ \sigma(\bx) > \varrho \} } 
	\leq \mathds{1}_{ \{ |\bx-\by| < \frac{ \theta(\veps) \delta }{ 1 - \theta(\veps) \delta } \sigma(\by) \} } \mathds{1}_{ \{ \sigma(\bx) > \varrho \} } \mathds{1}_{ \{ \sigma(\by) > (1- \theta(\veps) \delta) \varrho \} }  
	$$
	and  $\sigma(\bx) \leq (1+\theta(\veps) \delta)\sigma(\by)$ which hold on the domains of integration. With that, we obtain
	\begin{equation*}
		\textrm{I}_{\veps,\varrho} \leq \frac{C(d,s,p,\underline{\delta}_0)}{ \delta^{p} } \int_{\cD \cap  \{ \sigma(\bx) > (1-\theta(\veps) \delta) \varrho \} } \frac{|K_\veps u(\bx) - u(\bx)|^p}{ \sigma(\bx)^{sp} } \, \rmd \bx \leq \frac{C}{\delta^p} {
			\frac{ \vnorm{K_\veps u - u}_{L^p(\cD)}^p }{ (1-\underline{\delta}_0)^{sp} \varrho^{sp} }.
		}
	\end{equation*}
	To estimate $\textrm{II}_{\veps,\varrho}$, we use \eqref{eq:Comp:ConvolutionEstimate1} to get that for all $\veps>0$ sufficiently small 
	\begin{equation*}
		\begin{split}
			\textrm{II}_{\veps,\varrho} &\leq 2^{p-1}(1+\phi(\veps)) \int_{\cD \cap \{ \sigma(\bx) < \phi(\veps) \varrho \} } \int_{B(\bx, \delta \sigma(\bx))} \frac{ |u(\bx) - u(\by)|^p }{ \delta^{d+p} \sigma(\bx)^{d+sp} } \, \rmd \by \, \rmd \bx \\
			&\leq C(d,s,p,\underline{\delta}_0) \int_{\cD \cap \{ \sigma(\bx) < \phi(\underline{\delta}_0) \varrho \} } \int_{B(\bx, \delta \sigma(\bx))} \frac{ |u(\bx) - u(\by)|^p }{ \delta^{d+p} \sigma(\bx)^{d+sp} } \, \rmd \by \, \rmd \bx.
		\end{split}        
	\end{equation*}
	Therefore we can apply the dominated convergence theorem 
	to the right-hand side of this inequality, and get that $\lim\limits_{\varrho \to 0} \sup_{\veps \in (0,\underline{\delta}_0)}\textrm{II}_{\veps,\varrho}  = 0$.
	Hence, for arbitrary $\tau > 0$ there exists $\bar{\varrho} > 0$ such that 
	\begin{equation*}
		\sup_{\veps \in (0,\underline{\delta}_0)}\textrm{II}_{\veps,\bar{\varrho}} < \tau.
	\end{equation*}
	We now use this $\bar{\varrho}$ and let $\veps \to 0$ in the estimate of $\textrm{I}_{\veps,\bar{\varrho}}$ to obtain that $\limsup_{\veps \to 0} \textrm{I}_{\veps,\bar{\varrho}} = 0$. It then follows that 
	\begin{equation*}
		\limsup \limits_{\veps \to 0}[K_\veps u-u]_{\mathfrak{W}^{s,p}[\delta](\cD)}^p \leq \tau.
	\end{equation*}
	The first convergence result follows since $\tau > 0$ is arbitrary.

	Now suppose that $\cD$ is Lipschitz. 
	Define $K_{\veps} u$ just as above for $\veps < \delta$; note that $K_{\veps} u \in W^{s,p}(\cD)$ by \eqref{eq:ConvEst:Deriv:s:Opt}. 
	So we can regard $K_\veps u$ as belonging to $W^{s,p}(\bbR^d)$ by standard extension theorems for fractional Sobolev spaces.
	Let $\varphi$ be a standard mollifier, and define for $0 < \gamma \ll \veps$
	\begin{equation*}
		v_{\gamma,\veps}(x) = \varphi_{\gamma} \ast K_{\veps} u(x)\,.
	\end{equation*}
	Then $v_{\gamma,\veps} \in C^{\infty}(\overline{\cD})$. 
	Moreover, by \eqref{frak-frac-estimate}, we have 
	\begin{equation*}
		\lim\limits_{\gamma \to 0} \Vnorm{v_{\gamma,\veps} - K_{\veps} u}_{\mathfrak{W}^{s,p}[\delta](\cD)}^p \leq \lim\limits_{\gamma \to 0} C \delta^{s-1} \Vnorm{v_{\gamma,\veps} - K_{\veps} u}_{W^{s,p}(\cD)}^p = 0\,.
	\end{equation*}
	For each $n \in \bbN$, choose $\{\veps_n\}_n$ to be strictly decreasing sequence that satisfies $\vnorm{ K_{\veps_n} u - u }_{\mathfrak{W}^{s,p}[\delta](\cD)} < \frac{1}{2n}$. Then for each $n$, there exists $\gamma_n(\veps_n)$ depending on $\veps_n$ such that $\vnorm{ v_{\gamma_n,\veps_n} - K_{\veps_n} u }_{\mathfrak{W}^{s,p}[\delta](\cD)} < \frac{1}{2n}$. We can choose the sequence $\{\gamma_n\}_n$ to be strictly decreasing as well. Define $\{w_n= v_{\gamma_n,\veps_n} \}_n\subset C^{\infty}(\overline{\cD})$; we conclude with
	\begin{equation*}
		\begin{split}
			\vnorm{ w_n - u }_{\mathfrak{W}^{s,p}[\delta](\cD)}
			&\leq \vnorm{ w_n - K_{\veps_n} u }_{\mathfrak{W}^{s,p}[\delta](\cD)} + \vnorm{ K_{\veps_n} u - u }_{\mathfrak{W}^{s,p}[\delta](\cD)} < \frac{1}{n}.
		\end{split}
	\end{equation*}
\end{proof}

\subsection{Trace and extension theorems for \texorpdfstring{$\mathfrak{W}^{s,p}[\delta](\cD)$}{the nonlocal space}}\label{subsec:trace}
The papers \cite{TiDu17,Du2022Fractional,Foss2021} have established that for $sp>1$ and for Lipschitz domains $\cD$, functions that belong to the nonlocal function space $\mathfrak{W}^{s,p}[\delta](\cD)$ admit a trace that belongs to the space $W^{s-{1\over p} }(\partial \cD)$. Here we give an alternate proof of the result using the boundary-localized convolutions, the density of smooth functions and trace theorems of classical fractional and weighted Sobolev spaces.
{We do this for two reasons. First, the proof given here is shorter than those appearing in \cite{TiDu17,Du2022Fractional,Foss2021} and utilizes different tools; second, our main results require the dependence of the upper bound of the trace operator norm on parameters such $s$, $p$, and $\delta$ to be explicitly stated. Such details are not readily available in the existing work.} We state the trace theorem as follows.

\begin{theorem}\label{thm:TraceTheorem}
	Let $1 < p < \infty$ and $s \in (0,1]$ with $sp > 1$, and $\cD \subset \bbR^d$ be a Lipschitz domain.
	Let $T$ denote the trace operator, i.e. for $u \in C^{\infty}(\overline{\cD})$,
	$ T u = u \big|_{\p \cD}$.
	Then for each $\delta < \underline{\delta}_0$ the trace operator extends to a bounded linear operator $T : \mathfrak{W}^{s,p}[\delta](\cD) \to W^{s-1/p,p}(\p \cD)$. Moreover
	there exists $C = C(d,p,\cD)$ such that
	$$
	\Vnorm{Tu}_{W^{s-1/p,p}(\p \cD)} \leq \frac{C}{sp-1} \Vnorm{u}_{\mathfrak{W}^{s,p}[\delta](\cD)} ,\qquad\forall u \in \mathfrak{W}^{s,p}[\delta](\cD)\,.
	$$
\end{theorem}
Before giving the proof of the theorem, we state and prove a lemma that asserts that for $u$  in the fractional Sobolev space $W^{s,p}(\cD)$, the trace of $u$ and that of its boundary-localized convolution $K_\delta u$ agree. 
\begin{lemma}\label{lem:TraceOfConv}
	Let $1 < p < \infty$ and $s \in (0,1]$ with $sp > 1$ and let $\cD \subset \bbR^d$ be a Lipschitz domain, and suppose that $u \in W^{s,p}(\cD)$. If $T: W^{s,p}(\cD) \to W^{s-1/p,p}(\p \cD)$ is the classical trace operator, then 
	\begin{equation*}
		T K_{\delta} u = T u \quad \text{ in the sense of functions in } W^{s-1/p,p}(\p \cD)\,.
	\end{equation*}
\end{lemma}

\begin{proof}
	Let $\tau > 0$, and let $\wt{u} \in C^{\infty}(\overline{\cD})$ with $\Vnorm{\wt{u}-u}_{W^{s,p}(\cD)} < \tau$. Then by \Cref{continuity-of-Kdelta}, we have  $K_\delta \wt{u} = \wt{u}$ on $\p \cD$. Thus by adding and subtracting $T K_{\delta}\wt{u} $ and $T\wt{u}$, and noting that  $\Vnorm{ T K_{\delta} \wt{u} - T\wt{u} }_{W^{s-1/p,p}(\p \cD)}=0$ we apply \Cref{rmk:contKfraktofrac}, \eqref{eq:ConvEst1:Wsp} to obtain that
	\begin{equation*}
		\begin{split}
			\Vnorm{ T K_{\delta} u - T u }_{W^{s-1/p,p}(\p \cD)} 
			&\leq \Vnorm{ TK_{\delta} \wt{u} - T K_\delta u }_{W^{s-1/p,p}(\p \cD)} + \Vnorm{ T\wt{u} - Tu }_{W^{s-1/p,p}(\p \cD)} \\
			&\leq C\Vnorm{  K_{\delta} \wt{u} - K_\delta u  }_{W^{s,p}(\cD)} + C\Vnorm{\wt{u} - u}_{W^{s,p}(\cD)} \\
			&\leq C \Vnorm{\wt{u}-u}_{W^{s,p}(\cD)} \leq C \tau\,.
		\end{split}
	\end{equation*}
	where $C$ is independent of $\delta.$ 
	Since $\tau$ is arbitrary, the theorem is proved.
\end{proof}
\begin{proof}[Proof of \Cref{thm:TraceTheorem}]
	
	Let $T$ be the classical trace operator on $W^{s,p}(\cD)$ as used in \Cref{lem:TraceOfConv}. Then by \eqref{eq:ConvEst:Deriv:s:Opt} in \Cref{rmk:contKfraktofrac}, we have that for any $u\in \mathfrak{W}^{s,p}[\delta](\cD)$, $K_\delta u\in W^{s,p}(\cD)$. Therefore, $T K_\delta u \in W^{s-1/p,p}(\p \cD)$ and the following estimate holds per \eqref{eq:ConvEst:Deriv:s:Opt} and \Cref{thm:Trace:Fractional}:  
	\begin{equation*}
		\Vnorm{T K_\delta u}_{W^{s-1/p,p}(\p \cD)} \leq \frac{C(d,p,\cD)}{sp-1} \Vnorm{K_\delta u}_{W^{s,p}(\cD)} \leq \frac{C(d,p,\cD)}{sp-1} \Vnorm{u}_{\mathfrak{W}^{s,p}[\delta](\cD)}\,.
	\end{equation*}
	Now let $\{ u_n \} \subset C^{\infty}(\overline{\cD})$ be a sequence converging to $u$ in $\mathfrak{W}^{s,p}[\delta](\cD)$. Then by \Cref{lem:TraceOfConv}, 
	\begin{equation*}
		\begin{split}
			\Vnorm{ T u_n - T u_m }_{W^{s-1/p,p}(\p \cD)} 
			&= \Vnorm{ T K_\delta u_n - T K_\delta u_m }_{W^{s-1/p,p}(\p \cD)} \\
			&\leq C \Vnorm{ u_n - u_m }_{\mathfrak{W}^{s,p}[\delta]( \cD)}\,.
		\end{split}
	\end{equation*}
	Therefore, $\{T u_n\}$ is Cauchy in  $W^{s-1/p,p}(\p \cD)$ and therefore has a limit. Define 
	\[Tu := \lim_{n\to \infty} T u_n \quad \text{in $W^{s-1/p,p}(\p \cD)$}. \]
	It is now clear from standard arguments that the trace operator is a well-defined bounded linear mapping from $\mathfrak{W}^{s,p}[\delta](\cD)$ to $W^{s-1/p,p}(\p \cD)$.
\end{proof}

Using the density of smooth functions in the nonlocal function space $\mathfrak{W}^{s,p}[\delta](\cD)$, and the boundedness of $K_\delta$ on $\mathfrak{W}^{s,p}[\delta](\cD)$ proved in \Cref{cor:ConvEst:NonlocalSpace},  the following result can be proved in the same way as \Cref{lem:TraceOfConv} for functions in $\mathfrak{W}^{s,p}[\delta](\cD)$.

\begin{corollary}\label{cor:Trace:NonlocalSpace}
	Let $1 < p < \infty$ and let $s\in (0,1]$ with $sp > 1$, and let $\cD \subset \bbR^d$ be a Lipschitz domain. Suppose that $u \in \mathfrak{W}^{s,p}[\delta](\cD)$. Then 
	\begin{equation*}
		T K_{\delta} u = T u \quad \text{ in the sense of functions in } W^{s-1/p,p}(\p \cD)\,.
	\end{equation*}
\end{corollary}

A Lebesgue point property can also be proved for $Tu$ that relates the a.e. pointwise value of $Tu$ with that of $u$ in $\cD$ for all $u$ in $\mathfrak{W}^{s,p}[\delta](\cD)$.  Indeed, for all $u \in \mathfrak{W}^{s,p}[\delta](\cD)$, i.e.
\begin{equation*}
	Tu(\bx) = \lim\limits_{\veps \to 0} \fint_{B(\bx,\veps)} u(\by) \, \rmd \by\,, \qquad \text{ for } \scH^{(d-sp)_+}\text{-a.e. } \bx \in \p \cD\,.
\end{equation*}
This property is proved in \cite{Foss2021}, but a shorter proof is given in \cite{du2024weighted}.
\begin{remark}\label{rmk:StrongConvTrace}
	Observe that for $s\in (0, 1)$ and $p \in (1, \infty)$ such that $sp >1$, if $\{u_n\}$ is a weakly convergent sequence in $\mathfrak{W}^{s,p}[\delta](\cD)$ with limit point $u$, then $T u_n \rightharpoonup T u$ weakly in $W^{s-1/p,p}(\p \cD)$. Therefore, by the compact embedding theorem for fractional Sobolev spaces, it follows that $\vnorm{T u_n - Tu}_{L^p(\p \cD)} \to 0$ as $n \to \infty$.
	This same observation holds for weakly convergent sequences in the spaces $W^{s,p}(\cD)$ and $W^{1,p}(\cD;p-sp)$.
\end{remark}

We close this subsection with the following extension result whose proof follows from the extension theorems from $W^{s-1/p,p}(G)$  to both the fractional and the weighted Sobolev spaces ${W}^{s,p}(\cD) $ and $W^{1,p}(\cD;p-sp)$. The statement and sketch of the later results are given in the Appendix, \Cref{thm:FxnSpProp:Weighted}.

\begin{proposition}\label{lma:Extension}
	Let $\cD \subset \bbR^d$ be a Lipschitz domain, and let $G \subset \p \cD$ be a $\scH^{d-1}$-measurable closed set that is diffeomorphic to a finite union of the closures of $W^{s-1/p,p}$-extension domains.
	Then there exists a bounded linear operator $E : W^{s-1/p,p}(G) \to \mathfrak{W}^{s,p}[\delta](\cD)$, with $C = C(d,p,\cD,G) > 0$ 
	such that 
\begin{equation}\label{eq:ExtensionBd}
		\vnorm{Ev}_{ \mathfrak{W}^{s,p}[\delta](\cD) } \leq \frac{C}{s} \vnorm{v}_{W^{s-1/p,p}(G)}.
	\end{equation}
\end{proposition}

\begin{proof}
	First, let $v \in C^\infty(G)$. By \Cref{thm:FxnSpProp:Weighted} item 3), there exists an extension operator $E:W^{s-1/p,p}(G) \to W^{1,p}(\cD;p-sp)$ such that 
	\begin{equation*}
		\vnorm{ Ev }_{ W^{1,p}(\cD;p-sp) } \leq \frac{C}{s} \vnorm{v}_{W^{s-1/p,p}(G)}.
	\end{equation*}
	Then by \Cref{lma:Embedding:Weighted:Pre} we have
	\begin{equation*}
		\vnorm{ Ev }_{ \mathfrak{W}^{s,p}[\delta](\cD) } \leq \frac{C}{s} \vnorm{Ev}_{W^{1,p}(\cD;p-sp)}
		\leq \frac{C}{s} \vnorm{v}_{W^{s-1/p,p}(G)}.
	\end{equation*}
	This establishes the bound for $v \in C^\infty(G)$.
	Finally, we can make use of \Cref{thm:Density} to extend $E$ from $C^\infty(G)$ to a well-defined bounded linear operator on $W^{s-1/p,p}(G)$, and we can also conclude that the estimate \eqref{eq:ExtensionBd} holds for all $v \in W^{s-1/p,p}(G)$. 
\end{proof}

\subsection{\texorpdfstring{$L^{p}$}{Lp}-compactness criterion}
As discussed earlier, for a fixed $\delta>0$, all functions in $\mathfrak{W}^{s,p}[\delta](\cD)$ are only required to be $L^{p}_{loc}(\cD)$, but exhibit sufficient regularity on the boundary to have a well-defined trace. As such, the space $\mathfrak{W}^{s,p}[\delta](\cD)$ is not expected to be compactly contained in $L^{p}(\cD)$. However, as the kernel $\gamma^{s, \delta}_{p, \mathcal{D}}(\bx, \by)$ becomes more singular as $\delta \to 0$, one might expect compactness of classes of functions whose nonlocal seminorms remain finite.
In this subsection we make this rigorous and provide a compactness criterion for a sequence $\{ u_\delta\}\subset L^{p}(\cD)$ with bounded $\mathfrak{W}^{s,p}[\delta](\cD)$-norm.  We first start with a nonlocal characterization of a weighted Sobolev space in the spirit of \cite{bourgain2001}.  
\begin{proposition}\label{thm:LocalizationOfSeminorm}
	Let $ p \in (1, \infty)$, $s \in (0,1]$, $\{s_n\}_{n=1}^\infty \subset (0, 1]$ and $\{\delta_n\}_{n=1}^\infty\subset (0,\underline{\delta}_0)$ such that $s_n\to s$ and $\delta_n\to 0$ as $n\to \infty$. Then 
	if a sequence $\{u_n\}_n$ converges to $u$ in $C^2(\overline{V})$ for any $V \Subset \cD$ as $n \to \infty$, then
	\begin{equation*}
		\lim\limits_{n\to\infty} \overline{C}_{d,p} \int_{V} \int_{V \cap B(\bx,\delta_n \sigma(\bx))}  \frac{|u_n(\by)-u_n(\bx)|^p}{ \delta_n^{d+p} \sigma(\bx)^{d+s_n p} } \, \rmd \by \, \rmd \bx = \int_{V} |\grad u(\bx)|^p \sigma(\bx)^{p-sp} \, \rmd \bx\,.
	\end{equation*}
			If in addition, $ s_n\le s,\ \forall n$, then 
			\begin{equation*}
				\lim\limits_{n\to \infty} [u]_{\mathfrak{W}^{s_n,p}[\delta_n](\cD)}^p =
				\begin{cases}
					[u]_{W^{1,p}(\cD;p-sp)}^p\,, & 
					\text{if }  u \in W^{1,p}(\cD;p-sp) ,\\
					+ \infty, & \text{if } u \in L^p(\cD) \setminus W^{1,p}(\cD;p-sp).
				\end{cases}
			\end{equation*}    
			The previous statements hold if $s_n=s\in (0, 1]$ for all $n$.
		\end{proposition}

		\begin{proof}
			The proof of the first statement follows exactly the same steps as \cite[Proposition 4.1, Remarks 4.1 and 4.2]{ponce2004new}. The heterogeneous localization $\sigma(\bx)$ gives no additional difficulty.  The proof of the second statement in the case $u \in W^{1,p}(\cD;p-sp)$ follows exactly the same steps as \cite[Theorem 1.1]{ponce2004new}, 
			Indeed, for $u\in W^{1,p}(\cD;p-sp)$, by density for any $\veps >0$, there exists $\varphi\in C^\infty(\overline{\cD})$ such that $[u-\varphi]_{W^{1,p}(\cD;p-sp)}<\veps$. Then we have the estimate 
			\begin{align*}
				&\left|[u]_{\mathfrak{W}^{s_n,p}[\delta_n](\cD)} - [u]_{W^{1,p}(\cD;p-sp)} \right|\\
				& \qquad\leq  \left|[u]_{\mathfrak{W}^{s_n,p}[\delta_n](\cD)} - [\varphi]_{\mathfrak{W}^{s_n,p}[\delta_n](\cD)}\right|+\left|[\varphi]_{\mathfrak{W}^{s_n,p}[\delta_n](\cD)} - [\varphi]_{W^{1,p}(\cD;p-sp)}\right|\\
				& \qquad\qquad+\left| [\varphi]_{W^{1,p}(\cD;p-sp)} - [u]_{W^{1,p}(\cD;p-sp)}\right|\\
				&  \qquad \leq  [u-\varphi]_{\mathfrak{W}^{s_n,p}[\delta_n](\cD)}+\left|[\varphi]_{\mathfrak{W}^{s_n,p}[\delta_n](\cD)} - [\varphi]_{W^{1,p}(\cD;p-sp)}\right|+[u-\varphi]_{W^{1,p}(\cD;p-sp)}.
			\end{align*}
			But by \Cref{cor:Embedding:Weighted}, we have 
			\begin{align*}
				[u-\varphi]^{p}_{\mathfrak{W}^{s_n,p}[\delta_n](\cD)} &\leq  \frac{C(d)}{(1-\delta_n)^{p-s_n p+1}} [u-\varphi]_{W^{1,p}(\cD;p-s_n p)}^p\\
				&\leq \frac{C}{(1-\underline{\delta}_0^{p+1})}\int_{\cD}|\nabla (u-\varphi)(\bx)|^p \sigma(\bx)^{p-s p}\rmd\bx\\
				&= \frac{C}{(1-\underline{\delta}_0^{p+1})} [u-\varphi]^p_{W^{1, p}(\cD; p-sp)},
			\end{align*}
			since $s_n\leq s$ and $\sigma(\bx)^{sp-s_np}\le C$ uniformly in $n$ and $\bx$. 
			As a consequence, we have that 
			\begin{align*}
				\left|[u]_{\mathfrak{W}^{s_n,p}[\delta_n](\cD)} - [u]_{W^{1,p}(\cD;p-sp)} \right|
				&\leq  2C [u-\varphi]_{W^{1, p}(\cD; p-sp)} 
				\\
				&\qquad + \left|[\varphi]_{\mathfrak{W}^{s_n,p}[\delta_n](\cD)} - [\varphi]_{W^{1,p}(\cD;p-sp)}\right|.
			\end{align*}
			We now let $n\to \infty$ and use the first statement of this proposition to obtain that $\limsup_{n\to \infty} \left|[u]_{\mathfrak{W}^{s_n,p}[\delta_n](\cD)} - [u]_{W^{1,p}(\cD;p-sp)} \right| \leq 2C\veps$. The second statement in the theorem for the case $u \in W^{1,p}(\cD;p-sp)$ now follows since $\veps$ is arbitrary.
			
			To prove the remainder of the second statement, it suffices to show that for any $u \in L^p(\cD)$
			\begin{equation*}
				\liminf_{n\to\infty} [u]_{\mathfrak{W}^{s_n,p}[\delta_n](\cD)} < \infty \quad \Rightarrow \quad u \in W^{1,p}(\cD;p-sp);
			\end{equation*}
			we present a simple argument using the boundary-localized convolution. 
			Let $M := \liminf_{n\to\infty} [u]_{\mathfrak{W}^{s_n,p}[\delta_n](\cD)}$, and choose a subsequence $\{ [u]_{\mathfrak{W}^{s_{n'},p}[\delta_{n'}](\cD)} \}_{n'}$ that converges to $M$.
			Fix $\bar{s} \in (0,s)$. Then since $s_n\to s$, \Cref{thm:Convolution:DerivativeEstimate} implies that there exists $N \in \bbN$ depending on $\bar{s}$ such that for all $n'\geq N$ 
			\begin{equation}\label{eq:uniform-bound}
				[ K_{\delta_{n'}} u]_{W^{1,p}(\cD;p-\bar{s} p)}
				\leq C \,[ K_{\delta_{n'}} u]_{W^{1,p}(\cD;p-s_{n'} p)}
				\leq
				C 
				[u]_{\mathfrak{W}^{s_{n'},p}[\delta_{n'}](\cD)}.
			\end{equation} 
			Thus $\{ K_{\delta_{n'}} u\}_{{n'} \geq N}$ is uniformly bounded in $W^{1,p}(\cD;p-\bar{s}p)$.
			Hence, there exists a subsequence $\{K_{\delta_{n''}} \}_{n''}$ that converges weakly to a function $v$  in $W^{1,p}(\cD;p-\bar{s}p)$ and by \Cref{thm:FxnSpProp:Weighted} item 4) $K_{\delta_{n''}} u$ converges to $v$ strongly in $L^p(\cD)$ as $n'' \to \infty$. 
			But we know from \Cref{thm:diffuKu:Weighted} that $\vnorm{K_{\delta_{n''}}  u - u}_{L^p(\cD)} \leq C \delta_{n''}$, hence $K_{\delta_{n''}} u$ converges to $u$ strongly in $L^p(\cD)$ as $n \to \infty$. Thus $v = u$ in the sense of functions in $L^p(\cD)$. 
			Therefore, their weak derivatives coincide, and by \eqref{eq:uniform-bound}
			\begin{equation*}
				\begin{split}
					[u]_{W^{1,p}(\cD;p-\bar{s}p)} 
					= [v]_{W^{1,p}(\cD;p-\bar{s}p)} 
					&\leq \liminf_{n'' \to \infty} [K_{\delta_{n''}} u]_{W^{1,p}(\cD;p-\bar{s}p)} \\
					&\leq C \liminf_{n'' \to \infty} [u]_{\mathfrak{W}^{s_{n''},p}[\delta_{n''}](\cD)}
					= C M.
				\end{split}
			\end{equation*}
			The constants $C$ and $M$ are independent of $\bar{s}$; thus we conclude that $u\in W^{1,p}(\cD;p-sp)$ using Fatou's lemma by taking $\bar{s}\to s$ in the above inequality.
		\end{proof}
		The following theorem gives an $L^{p}$-compactness criterion in the spirit of \cite{bourgain2001,ponce2004new}. 
		\begin{theorem}\label{thm:Compactness}
			Let $p \in (1,\infty)$ and $s \in (0,1]$. Let $\{\delta_n \}_{n \in \bbN}$ be a sequence that converges to $0$, $\{s_n\}_{n\in\bbN}$ be a sequence in $(0,1]$ that converges to $s$, and let $\{ u_n \}_n \subset \mathfrak{W}^{s_n,p}[\delta_n](\cD)$ be a sequence such that for constants $B$ and $C$ independent of $n$ 
			\begin{equation*}
				\sup_{n} \vnorm{u_n}_{L^p(\cD)} \leq C < \infty \quad \text{ and } \quad \sup_{n } [u_n]_{\mathfrak{W}^{s_n,p}[\delta_n](\cD)} := B < \infty.
			\end{equation*}
			Then $\{ u_n \}_n$ is precompact in the strong topology of $L^p(\cD)$. Moreover, any limit point $u$ belongs to $W^{1,p}(\cD;p-sp)$, with $[u]_{W^{1,p}(\cD;p-sp)} \leq B$.

		\end{theorem}
		
		The proof for $s_n=s = 1$ is in \cite{scott2024nonlocal}; we write here a similar proof for general $s \in (0,1)$.

		\begin{proof}
			It suffices to show that a subsequence of $\{u_n\}$ is Cauchy in $L^p(\cD)$. 
			Choose $\psi$ to satisfy \eqref{Assump:Kernel}-\eqref{Assump:Kernel} and $\eta$ to satisfy \eqref{assump:Localization}, and define $K_{\delta_n} u$ accordingly.
			First, we use \eqref{eq:KdeltaError} to see that
			\begin{equation*}
				\Vnorm{ u_n -  K_{\delta_n} u_n }_{L^{p}(\cD)} \leq C\delta_n [u_n]_{\mathfrak{W}^{s_n,p}[\delta_n](\cD)} \leq C B\delta_n.
			\end{equation*}
			Using \Cref{cor:Trace:NonlocalSpace}
			\begin{equation*}
				\Vnorm{ K_{\delta_n} u_n }_{W^{s_n,p}(\cD)} \leq C \Vnorm{ u_n }_{\mathfrak{W}^{s_n,p}[\delta_n](\cD)}.
			\end{equation*}    
			Therefore, the sequence $\{ K_{\delta_n} u_n \}_{n \in \bbN}$ is bounded in $W^{s_n,p}(\cD)$. Fix any $s'\in (0, s)$ so that for $n$ large enough $s_n > s'$. Then $\{ K_{\delta_n} u_n \}_{n \in \bbN}$ is bounded in $W^{s',p}(\cD)$, 
			hence by the compact embedding of $W^{s',p}(\cD)$ into $L^p(\cD)$  it is precompact in $L^p(\cD)$, and so for a convergent subsequence $\{ K_{\delta_n} u_n \}$ in $L^p(\cD)$ (not relabeled) and for $n$, $m \in \bbN$
			\begin{equation*}
				\begin{split}
					\Vnorm{ u_n - u_m }_{L^p(\cD)} 
					&\leq \Vnorm{ K_{\delta_n} u_n - u_n }_{L^p(\cD)} \\
					&\qquad + \Vnorm{ K_{\delta_m} u_m - u_m }_{L^p(\cD)} + \Vnorm{ K_{\delta_n} u_n - K_{\delta_m} u_m }_{L^p(\cD)} \\
					&\leq C B( \delta_m + \delta_n) + \Vnorm{ K_{\delta_n} u_n - K_{\delta_m} u_m }_{L^p(\cD)} \to 0 \text{ as } m,n \to \infty\,.
				\end{split}
			\end{equation*}
			To see that any limit $u$
			belongs to $W^{1,p}(\cD;p-sp)$, we use an argument similar to the one used in the proof of \Cref{thm:Density}.
			Let $\veps \in (0,\underline{\delta}_0)$ and $0 < \varrho < \diam(\cD)$ be given.
			Applying the estimate \eqref{eq:Comp:ConvolutionEstimate2} from \Cref{lem:with-phiepsilon} with fixed $\varrho$ yields
			\begin{align*}
				& \overline{C}_{d,p} \int_{\cD \cap \{ \sigma(\bx) > \varrho  \} }\int_{\cD  \cap \{ \sigma(\bx) > \varrho  \} \cap B(\bx, \theta(\veps)\delta_n \sigma(\bx))} \frac{\left|K_\veps u_n(\bx)- K_\veps u_n(\by)\right|^p}{(\theta(\veps)\delta_n)^{d+ p}\sigma(\bx)^{d + s_n p}} \rmd\by \rmd\bx \\
				\le & \overline{C}_{d,p} \phi(\veps)\int_{\cD\cap \{\sigma(\bx) > \frac{\varrho}{\phi(\veps)}\}}\int_{B(\bx, \delta_n \sigma(\bx))} \frac{\left|u_n(\bx)- u_n(\by)\right|^p}{\delta_n^{d+p}\sigma(\bx)^{d+s_n p}}\rmd\by \rmd\bx\\
				\le & \overline{C}_{d,p} \phi(\veps)\int_{\cD}\int_{B(\bx, \delta_n \sigma(\bx))} \frac{\left|u_n(\bx)- u_n(\by)\right|^p}{\delta_n^{d+p}\sigma(\bx)^{d+s_n p}}\rmd\by \rmd\bx \le \phi(\veps)B^p.
			\end{align*}
			Now, for any fixed $\veps > 0$ the sequence $\{K_\veps u_n \}_{n}$ converges to $K_\veps u$ in $C^2( \{ \bx \in \cD \, : \, \sigma(\bx) \geq \varrho\} )$ as $n \to \infty$, since $\cD \cap \{ \sigma(\bx) > \varrho \} \subset \subset\cD$. Therefore we can use \Cref{thm:LocalizationOfSeminorm} when taking $n \to \infty$ in the previous inequality to get
			\begin{equation*}
				\int_{\cD \cap \{ \sigma(\bx) > \varrho \} } |\grad K_\veps u|^p \sigma(\bx)^{p-sp} \, \rmd \bx \leq \phi(\veps) B^p.
			\end{equation*}
			Letting $\varrho\to 0$ yields \begin{equation}\label{eq:cmp-estimate}
				\int_{\cD } |\grad K_\veps u|^p \sigma(\bx)^{p-sp} \, \rmd \bx \leq \phi(\veps) B^p.
			\end{equation}
			That estimate is uniform in $\veps$ for any $\veps \in (0, \underline{\delta}_0)$. 
			Thus by \Cref{thm:FxnSpProp:Weighted} $\{K_\veps u\}_{\veps}$ converges to a function $v$ weakly in $W^{1,p}(\cD;p-sp)$ and strongly in $L^p(\cD)$. But since $K_\veps u \to u$ strongly in $L^p(\cD)$ by \Cref{thm:diffuKu:Weighted}, it follows that $v = u$ as functions in $L^p(\cD)$, and moreover their weak derivatives coincide. Therefore $u \in W^{1,p}(\cD;p-sp)$. Finally, we can take $\veps\to 0$ in \eqref{eq:cmp-estimate} by using using \Cref{prop:conv-convergence-Sobolev-spaces}, and noting from \Cref{lem:with-phiepsilon} that $\phi(\veps)\to 1$ as $\veps \to 0$, we obtain
			$[u]_{W^{1,p}(\cD;p-sp)} \leq B$. 
		\end{proof}

\subsection{Poincar\'e-type inequalities} 
In this subsection we prove Poincar\'e-type inequalities that will be valid for all functions in the nonlocal space $\mathfrak{W}^{s,p}[\delta](\cD)$ satisfying certain boundary conditions. 
We start with a space that we will use to study a  Dirichlet boundary value problem.

For a bounded Lipschitz domain $\cD \subset \bbR^d$ and a closed set $\p \cD_D \subset \p \cD$ with $\scH^{d-1}(\p\cD_D) > 0$, we define the Banach space
\begin{equation*}
	\begin{split}
		\mathfrak{W}^{s,p}_{0,\p \cD_D}[\delta](\cD) :=  \text{ closure of } C^{\infty}_c(\overline{\cD} \setminus \p \cD_D) \text{ with respect to the norm } \Vnorm{\cdot}_{\mathfrak{W}^{s,p}[\delta](\cD)}.
	\end{split}
\end{equation*}
In the event that $\p\cD_D =\p \cD$, we write $\mathfrak{W}^{s,p}_{0}[\delta](\cD) := \mathfrak{W}^{s,p}_{0,\p \cD}[\delta](\cD)$. We may apply the weak continuity of the trace operator discussed in \Cref{rmk:StrongConvTrace} to prove that functions in $\mathfrak{W}^{s,p}_{0,\p \cD_D}[\delta](\cD)$ have zero trace on $\p \cD_D$.  
\begin{lemma}\label{cor:TraceZero}
	Let $1 < p < \infty$ and $s \in (0,1]$ with $sp > 1$. Then a function $u$ belongs to $\mathfrak{W}^{s,p}_{0,\p \cD_D}[\delta](\cD)$ if and only if $u \in \mathfrak{W}^{s,p}[\delta](\cD)$ and $T u = 0$ on $\p \cD_D$\,.
\end{lemma}

\begin{proof}
	The forward implication is clear thanks to continuity of the trace; we will provide a proof for the reverse implication. Assume that $T u = 0$ on $\p \cD_D$.
	Then, for each $n \in \bbN$ there exists $\veps_n > 0$ such that the boundary localized convolution $K_{\veps_n} u$ satisfies $\vnorm{K_{\veps_n} u - u}_{\mathfrak{W}^{s,p}[\delta](\Omega)} < \frac{1}{n}$.
	The choice of such $\veps_n$ is possible by the argument in the proof of \Cref{thm:Density}.
	Moreover, for each $n$, $K_{\veps_n} u$ satisfies $T K_{\veps_n} u = 0$ on $\p \cD_D$ by \Cref{cor:Trace:NonlocalSpace}. Hence $K_{\veps_n} u \in W^{1,p}_{0,\p \cD_D}(\Omega;p-sp)$ by \eqref{eq:ConvEst:Deriv}, \Cref{thm:InvariantHorizon}, and \Cref{thm:FxnSpProp:Weighted} item 5). Additionally by \Cref{thm:FxnSpProp:Weighted} item 5), for each $n \in \bbN$ there exists $v_n \in C^\infty_c(\overline{\cD} \setminus \p \cD_D)$ such that $\vnorm{ v_n - K_{\veps_n} u }_{W^{1,p}(\Omega;p-sp)} < \frac{1}{n}$. The sequence $\{v_n\}_n$ is the desired sequence; indeed, by \Cref{cor:Embedding:Weighted}
	\begin{equation*}
		\vnorm{v_n-u}_{\mathfrak{W}^{s,p}[\delta](\Omega)} \leq \vnorm{K_{\veps_n} u - u}_{\mathfrak{W}^{s,p}[\delta](\Omega)} + C \vnorm{v_n-K_{\veps_n} u}_{W^{1,p}(\Omega;p-sp)} \leq \frac{1+C}{n},
	\end{equation*}
	which completes the proof.
\end{proof}

The following Poincar\'e-type inequality is valid for functions in $\mathfrak{W}^{s,p}_{0,\p \cD_D}[\delta](\cD)$. 
\begin{theorem}\label{thm:PoincareDirichlet}
	Let $1 \leq p < \infty$ and $s \in (0,1]$ with $sp > 1$. Then there exists a constant $C_D=C_D(d,p,\cD,\kappa_0,\kappa_1)$
	such that for all $\delta < \underline{\delta}_0$ and all $s \in (0,1)$
	\begin{equation*}
		\Vnorm{u}_{L^p(\cD)} \leq C_D [u]_{\mathfrak{W}^{s,p}[\delta](\cD)}\,
		,\quad \forall u \in \mathfrak{W}^{s,p}_{0,\p \cD_D}[\delta](\cD)\,.
	\end{equation*}
\end{theorem}

\begin{proof}[Proof]
	Let $\psi$ satisfy \eqref{Assump:Kernel} and let $\eta$ satisfy \eqref{assump:Localization}. For any $u\in \mathfrak{W}^{s,p}_{0,\p \cD_D}[\delta](\cD)$, we have that $K_{\delta} u \in W^{s,p}(\cD)$, see \Cref{rmk:contKfraktofrac}.
	Moreover, by \Cref{cor:Trace:NonlocalSpace} and \Cref{cor:TraceZero} we have that $T K_{\delta} u = 0$ on $\p \cD_D$. 
	Hence $K_{\delta} u \in W^{s,p}_{0,\p \cD_D}(\cD)$, and we can apply the fractional Poincar\'e inequality \Cref{thm:Poincare:Fractional}:
	\begin{equation}\label{eq:PoincareDirichlet:Pf1}
		\Vnorm{ K_{\delta} u}_{L^p(\cD)} \leq C(d,p,\cD, \p \cD_D) [ K_{\delta} u]_{W^{s,p}(\cD)}\,.
	\end{equation}
	Then since 
	$[K_{\delta} u]_{W^{s,p}(\cD)} \leq C [u]_{\mathfrak{W}^{s,p}[\delta](\cD)}$, by \eqref{eq:ConvEst:Deriv:s:Opt}, we have that 
	\begin{equation*}
		\begin{split}
			\Vnorm{u}_{L^p(\cD)} \leq \Vnorm{ K_{\delta} u}_{L^p(\cD)} + \Vnorm{ u - K_{\delta} u}_{L^p(\cD)} \leq (C + C \delta) [u]_{\mathfrak{W}^{s,p}[\delta](\cD)}
		\end{split}
	\end{equation*}
	where the last inequality follows from \Cref{thm:diffuKu:Weighted}. 
\end{proof}

\section{Well-posedness of the variational problem involving coupled nonlocal and fractional energies}\label{sec:variational-problem}
In this section we will prove \Cref{thm:wellposedness} on the well-posedness of the variational problem of the minimization of a coupled energy. This energy consists of a nonlocal energy with a boundary-localized heterogeneous kernel on a part of a domain and a $p$-fractional energy on the remaining portion, subject to a transmission condition. We recall that  $\Omega \subset \bbR^d$ is taken to be a bounded  Lipschitz domain. Let $\Gamma \subset \bbR^d$ be a hypersurface that is locally the graph of a Lipschitz function, and such that $\Sigma := \Omega \cap \Gamma$ partitions $\Omega$ into two disjoint Lipschitz domains $\Omega_1$ and $\Omega_2$. 
With this set up we notice that $\Omega = \Omega_1 \cup \Omega_2 \cup \Sigma$. 

We recall again that for $1<p<\infty$,  $s, \delta \in (0, 1)$, and $\overline{C}_{d,p}$ and $\kappa_{d, s,p}$ given by \eqref{constant-A}, the energy functional we will be minimizing is 
\begin{equation}\label{energy-section3}
	\begin{split}
		\cE_{s,\delta}(u_1,u_2) &:= \frac{\overline{C}_{d,p}}{p} \int_{\Omega_1} \int_{ B(\bx,\delta \sigma(\bx)) } \alpha(\bx)\frac{|u_1(\bx)-u_1(\by)|^p}{ \delta^{d+p} \sigma(\bx)^{d+sp} } \, \rmd \by \, \rmd \bx \\
		&\qquad + \frac{\kappa_{d,s,p}}{p} \int_{\Omega_2} \int_{\Omega_2} \beta(\bx)\frac{|u_2(\bx)-u_2(\by)|^p}{|\bx-\by|^{d+sp}} \, \rmd \by \, \rmd \bx,
	\end{split}
\end{equation}
where $\sigma(\bx) = \text{dist}(\bx, \Omega_1).$ We notice that $\cE_{s,\delta}(u_1,u_2) <\infty$ if and only if $(u_1, u_2)\in \mathfrak{W}^{s,p}[\delta](\Omega_1) \times W^{s,p}(\Omega_2)$ where the space $\mathfrak{W}^{s,p}[\delta](\Omega_1)$ is studied in the previous sections. We restrict $\delta$ to be in $(0, \underline{\delta}_0)$, where $\underline{\delta}_0$ is given by \eqref{bound-for-delta} and $s$ and $p$ such that $sp > 1$ so as to apply the theory for traces developed in the previous section. In this regime of parameters both spaces $\mathfrak{W}^{s,p}[\delta](\Omega_1)$ and $ W^{s,p}(\Omega_2)$ have traces with $T_1$ and $T_2$ being the respective trace operators. The transmission condition we require any minimizer of \eqref{energy-section3} to satisfy is a matching of traces of $u_1$ and $u_2$ on the shared interface $\Sigma $ as given by \eqref{eq:TransmissionContinuity} which we recall as 
\begin{equation}\label{eq:TransmissionContinuity-sec3}
	(T_1 u_1 - T_2 u_2) \mathds{1}_{\Sigma} = 0 \quad \scH^{d-1}\text{-a.e. on } \Sigma, 
\end{equation}
a homogeneous Dirichlet boundary condition, namely $T_1u_{1}$ and $T_2u_{2}$ vanish on $\partial \Omega \cap \partial \Omega_1$  and $\partial \Omega \cap \partial \Omega_2$ respectively. We summarize the admissible class of functions for the variational problem of minimizing \eqref{energy-section3} 
\begin{equation}\label{Admissible}
	\mathfrak{X}_{s,\delta} = \left\{ (u_1, u_2) \in \mathfrak{W}^{s,p}[\delta](\Omega_1) \times W^{s,p}(\Omega_2) : \, 
	\begin{gathered}
		\eqref{eq:TransmissionContinuity-sec3} \text{ holds, and for $i=1, 2$ } \\
		T_i u_i = 0 \,\text{ a.e. on } \p \Omega_i \setminus \Sigma 
	\end{gathered}
	\right\}.
\end{equation}
For $s\in (0,1)$, the space $\mathfrak{X}_{s,0}$ can be considered as the limit space of $\mathfrak{X}_{s,\delta}$ as $\delta\to 0$ and is given by
\begin{equation}\label{Admissible-delta0}
	\mathfrak{X}_{s, 0}= \left\{\begin{gathered}(u_{1}, u_2)\in W^{1, p}(\Omega_1;p-sp)\times W^{s, p}(\Omega_2): \eqref{eq:TransmissionContinuity-sec3} \text{ holds, and }\\
		\text{for $i=1, 2$ }\,  T_iu_{i} = 0,\,\text{ a.e. on } \partial \Omega_i\setminus  \Sigma
	\end{gathered}\right\}.
\end{equation}
Finally when $s=1$ and $\delta=0$, 
the admissible class is given by
\begin{equation}\label{Admissible-s1}
	\left\{\begin{gathered}(u_{1}, u_2)\in W^{1, p}(\Omega_1)\times W^{1, p}(\Omega_2): \eqref{eq:TransmissionContinuity-sec3} \text{ holds, and for $i=1, 2$ }\\
		T_iu_{i} = 0,\,\text{ a.e. on } \partial \Omega_i\setminus  \Sigma
	\end{gathered}\right\}.
\end{equation}
Notice that if $(u_{1}, u_2)$ is in the class given in \eqref{Admissible-s1}, then $u = u_1
\mathds{1}_{\Omega_1}+ u_2
\mathds{1}_{\Omega_2}$ belongs to $W^{1,p}_0(\Omega)=\mathfrak{X}_{1,0}$, as defined in \Cref{item:c} of \Cref{Main-gammaconvergence}. This follows from the matching of the traces and the vanishing of each of the functions on $\partial \Omega$. Conversely, if $u\in W^{1,p}_0(\Omega),$  then $(u|_{\Omega_1}, u|_{\Omega_2})$ is in the class given in \eqref{Admissible-s1}. This creates a natural identification of $\mathfrak{X}_{1, 0}=W^{1,p}_0(\Omega)$ with the class given in \eqref{Admissible-s1}. With this identification, it is now clear that $\mathfrak{X}_{1, 0}$ can be thought of as the limit space of $\mathfrak{X}_{s, 0}$, as $s\to 1$.

Notice that $ \mathfrak{X}_{s,\delta}$ is a linear space. The following statement says it is also a Banach space with useful properties.
\begin{proposition}\label{prop:X-separable-reflexive}
	For $1<p<\infty$, $s\in (0, 1)$ such that $sp>1$, and $0\leq \delta <1$, we have that  
	$\mathfrak{X}_{s,\delta}$ is a closed subspace of $\mathfrak{W}^{s,p}[\delta](\Omega_1) \times W^{s,p}(\Omega_2).$ 
	As a consequence,  $\mathfrak{X}_{s,\delta}$ is 
	a separable reflexive Banach space with the norm 
	\begin{eqnarray*}
		\Vnorm{(u_1,u_2)}_{\mathfrak{X}_{s,\delta}}^p := \vnorm{(u_1, u_2)}_{L^p(\Omega_1)\times L^{p}(\Omega_2)}^p  +  [(u_1,u_2)]_{\mathfrak{X}_{s,\delta}}^p,
	\end{eqnarray*}
	where  $ [(u_1,u_2)]_{\mathfrak{X}_{s,\delta}}^p := [u_1]^{p}_{\mathfrak{W}^{s,p}[\delta](\Omega_1)}+[u_2]^{p}_{W^{s,p}(\Omega_2)}$, and for $\delta=0$ and $s>0$, we replace the space  $\mathfrak{W}^{s,p}[\delta](\Omega_1)$ by $W^{1, p}(\Omega_1;p-sp)$. 
\end{proposition}
\begin{proof}
	We prove the theorem for the case when $\delta\in (0,1)$. The remaining case follows from a similar argument. 
	It suffices to show that $\mathfrak{X}_{s,\delta}$ is a closed subspace of the separable reflexive Banach product space $\mathfrak{W}^{s,p}[\delta](\Omega_1) \times W^{s,p}(\Omega_2).$  To this end, suppose that the sequence $\{(u_1^{n}, u^{n}_2)\}_n \subset \mathfrak{X}_{s,\delta}$ converges to  $(u_1, u_2)$ in $\mathfrak{W}^{s,p}[\delta](\Omega_1) \times W^{s,p}(\Omega_2).$ It is then clear that $u_1^{n} \to u_1$ in $\mathfrak{W}^{s,p}[\delta](\Omega_1)$ and $u^{n}_2 \to u_2$ in $W^{s,p}(\Omega_2).$ It then follows from the trace theorem \Cref{thm:TraceTheorem} that $T_{1} u^{n}_{1} \to T_1u_1$ in $W^{s-{1\over p},p}(\p \Omega_1)$ and $T_{2} u^{n}_{2} \to T_2u_2$ in $W^{s-{1\over p},p}(\p \Omega_2)$. Since for each $n$,  $T_1u^{n}_1 = T_{2} u^{n}_{2}$  a.e. on $\Sigma= \p\Omega_1\cap \p\Omega_2\cap \Omega$ from the strong convergence we have $T_1u_1 = T_2u_2$ a.e. on $\Sigma$ as well. Similarly, since for each $n$, $T_i u^n_i = 0$ a.e. $\p\Omega_1\setminus \Sigma$, we have from the strong convergence of traces that the limit satisfies   $T_i u_i = 0$ a.e. $\p\Omega_i\setminus \Sigma$. We then conclude that $(u_1, u_2)\in \mathfrak{X}_{s,\delta}.$
\end{proof}
\begin{remark} Using the weak continuity of the trace operators on each of the domains, \Cref{rmk:StrongConvTrace},  the stronger result that $\mathfrak{X}_{s,\delta}$ is weakly closed in the product space $\mathfrak{W}^{s,p}[\delta](\Omega_1) \times W^{s,p}(\Omega_2)$ can be easily proved.   
\end{remark}
The above proposition makes it possible to apply the direct method of the calculus of variations to show the existence of a minimizer for $\cE_{s,\delta}(u_1,u_2)$ in $\mathfrak{X}_{s,\delta}.$ We note that since $\alpha(\bx)$ and $\beta(\bx)$ are bounded from below and above by positive constants, the seminorms $[(u_1, u_2)]_{\mathfrak{X}_{s, \delta}}$ and $\cE_{s,\delta}(u_1, u_2)$ are equivalent.

Next, we prove a Poincar\'e-type inequality for the space $\mathfrak{X}_{s,\delta}.$

\begin{lemma}[Poincar\'e-type inequality on $\mathfrak{X}_{s, \delta}$]\label{PI-onX}
	Suppose $p\in (1, \infty)$ and that $\underline{\delta}_0>0$ is as given \eqref{bound-for-delta}. Then there exists a constant $C = C(d,p,\Omega_1,\Omega_2,\kappa_0,\kappa_1,\underline{\delta}_0)>0$ such that for all $s\in (1/p, 1]$ and $\delta\in [0, \underline{\delta}_0)$ 
	\begin{equation*}
		\vnorm{(u_1, u_2)}_{L^p(\Omega_1) \times L^p(\Omega_2) }^p 
		\leq C \cE_{s,\delta}(u_1,u_2), \qquad \forall (u_1,u_2) \in \mathfrak{X}_{s, \delta}.
	\end{equation*}
\end{lemma}
\begin{proof}
	We prove this only for the case $\delta>0$, $s < 1$; the proof for either $\delta = 0$ or $s = 1$ is simpler.
	Let $(u_1,u_2) \in \mathfrak{X}_{s, \delta}$. Without loss of generality assume that $\scH^{d-1}(\p \Omega_2 \setminus \Sigma) > 0$.
	Then 
	$u_2\in W^{s,p}_{0,\p\Omega_2\backslash\Sigma}(\Omega_2)$, the space of functions in $W^{s,p}(\Omega_2)$ whose traces vanish on the closed set $\p\Omega_2\backslash\Sigma$. 
	Since $\scH^{d-1}(\p\Omega_2\backslash\Sigma)>0$, by the Poincar\'e inequality \Cref{thm:Poincare:Fractional} for $W^{s,p}_{0,\p\Omega_2\backslash\Sigma}(\Omega_2)$ we have 
	\begin{equation}\label{ineq3}
		\|u_2\|_{L^p(\Omega_2)}\le C [u_2]_{W^{s,p}(\Omega_2)}.
	\end{equation}
	Since \(\Omega_2\) is Lipschitz, using the trace inequality \Cref{thm:Trace:Fractional} for \(u_2\in W^{s,p}(\Omega_2)\) and \eqref{ineq3} yields
	\begin{equation}\label{eq:transmission:Poinc:pf1}
		\|T_2 u_2\|_{W^{s-1/p,p}(\p\Omega_2)}\le C [u_2]_{W^{s,p}(\Omega_2)}.
	\end{equation}
	
	Let $G \subset \Sigma$ be a $\scH^{d-1}$-measurable closed set in $\bbR^d$, with positive $\scH^{d-1}$-measure, that is diffeomorphic to a finite union of the closures of $W^{s-1/p,p}$-extension domains,
	and let $E : W^{s-1/p,p}(G) \to \mathfrak{W}^{s,p}[\delta](\Omega_1)$ 
	be the extension operator defined in \Cref{lma:Extension}. Then applying the extension theorem on $T_{1}u_1$ restricted to $G$ we have $ET_1u_1\in \mathfrak{W}^{s,p}[\delta](\Omega_1)$ and 
	\begin{equation}\label{eq:transmission:Poinc:pf2}
		\|ET_1 u_1\|_{\mathfrak{W}^{s,p}[\delta](\Omega_1)}
		\leq C \|T_1 u_1\|_{W^{s-1/p,p}(G)}.
	\end{equation}
	Since $(u_1,u_2) \in \mathfrak{X}_{s, \delta}$, we have 
	\begin{equation}\label{eq:transmission:Poinc:pf3}
		\begin{split}
			\|T_1 u_1\|_{W^{s-1/p,p}(G)} &= \|T_2 u_2\|_{W^{s-1/p,p}(G)} \leq C[u_2]_{W^{s,p}(\Omega_2)},
		\end{split}
	\end{equation}
	where we used \eqref{eq:transmission:Poinc:pf1} on the last line.
	Combining \eqref{eq:transmission:Poinc:pf2} with \eqref{eq:transmission:Poinc:pf3} 
	yields 
	\begin{equation}\label{eq:transmission:Poinc:pf4}
		\|ET_1 u_1\|_{\mathfrak{W}^{s,p}[\delta](\Omega_1)}\le C [u_2]_{W^{s,p}(\Omega_2)}.
	\end{equation}
	Since \(u_1-ET_1 u_1\in \mathfrak{W}_{0,G}^{s,p}[\delta](\Omega_1)\), we may apply the Poincar\'e-type inequality, \Cref{thm:PoincareDirichlet}, to conclude that there exists a constant $C>0$ such that
	\begin{equation}\label{eq:transmission:Poinc:pf5}
		\|u_1-ET_1 u_1\|_{L^p(\Omega_1)} \le C [u_1-ET_1 u_1]_{\mathfrak{W}^{s,p}[\delta](\Omega_1)}.
	\end{equation}
	Then by \eqref{eq:transmission:Poinc:pf5} and \eqref{eq:transmission:Poinc:pf4}
	\begin{equation*}
		\begin{aligned}
	\|u_1\|_{L^p(\Omega_1)}&\le \|u_1-ET_1 u_1\|_{L^p(\Omega_1)}+\|ET_1 u_1\|_{L^p(\Omega_1)}\\
			&\le C [u_1-ET_1 u_1]_{\mathfrak{W}^{s,p}[\delta](\Omega_1)}+\|ET_1 u_1\|_{L^p(\Omega_1)}\\
			&\le C \left([u_1]_{\mathfrak{W}^{s,p}[\delta](\Omega_1)}+\|ET_1 u_1\|_{\mathfrak{W}^{s,p}[\delta](\Omega_1)}\right)\\
			&\le C\left([u_1]_{\mathfrak{W}^{s,p}[\delta](\Omega_1)}+[u_2]_{W^{s,p}(\Omega_2)}\right).
		\end{aligned}
	\end{equation*}
	Combining the last inequality with \eqref{ineq3} yields
	\begin{equation*}
\|u_1\|_{L^p(\Omega_1)}+\|u_2\|_{L^p(\Omega_2)}\le C\left([u_1]_{\mathfrak{W}^{s,p}[\delta](\Omega_1)}+[u_2]_{W^{s,p}(\Omega_2)}\right) = C [(u_1, u_2)]_{\mathfrak{X}_{s, \delta}}.
	\end{equation*}
	Therefore, using the equivalence of $[(u_1, u_2)]_{\mathfrak{X}_{s, \delta}}$ and $\cE_{s,\delta}(u_1, u_2)$, we have  \[\vnorm{(u_1, u_2)}_{L^p(\Omega_1)\times L^p(\Omega_2)}^p \leq C \cE_{s,\delta}(u_1, u_2), \qquad \forall u \in \mathfrak{X}_{s, \delta}.\]
\end{proof}
We close this section by proving \Cref{thm:wellposedness} on the existence of a solution to the variational problem. 

\begin{proof}[Proof of \Cref{thm:wellposedness}]
	We apply the direct method of calculus of variations. 
    
    Given $\mathfrak{f}\in  \mathfrak{X}'_{s, \delta}$, we consider the energy $\cF_{s, \delta}: \mathfrak{X}_{s, \delta}\to \mathbb{R}$ given by 
	\[
	\cF_{s, \delta}(v_1,v_2) = \cE_{s,\delta}(v_1,v_2) - \langle\mathfrak{f}, (v_1, v_2)\rangle_{\mathfrak{X}'_{s, \delta}, \mathfrak{X}_{s, \delta}}.   
	\]
	We will show that $\cF_{s, \delta}$ is coercive and weakly lower semicontinuous. These two assertions follow from the Poincar\'e inequality proved in \Cref{PI-onX}. Indeed, using the inequality it follows that $\cE_{s,\delta}(v_1,v_2)$ now  serves as a norm on $\mathfrak{X}_{s, \delta}$ and from which is weakly lower semicontinuity of $\cF_{s, \delta}$ follows. 
	Moreover, for any $(v_1, v_2)$
	\begin{equation*}
		\begin{split}
			\cF_{s, \delta}(v_1,v_2) &\geq \cE_{s, \delta}(v_1,v_2) -\|\mathfrak{f}\|_{\mathfrak{X}_{s,p}'} \|(v_1, v_2)\|_{\mathfrak{X}_{s,p}}\\
			&\geq  \cE_{s, \delta}(v_1,v_2) -C \left(\cE_{s, \delta}(v_1,v_2)\right)^{1\over p},
		\end{split}
	\end{equation*}
	where is $C$ is the product of  $\|\mathfrak{f}\|_{\mathfrak{X}_{s,p}'}$ and the Poincar\'e constant from \Cref{PI-onX}. 
	As a consequence, $\cF_{s, \delta}(v_1,v_2) \to \infty$ whenever $\cE_{s, \delta}(v_1,v_2) \to \infty$. Finally, the uniqueness of the minimizer follows from the strict convexity of $|\cdot|^p$ when $p\in (1,\infty)$.
\end{proof}

\section{Convergence of parameterized variational problems}\label{sec:variational-convergence}
In this section we prove the second main result of the paper: \Cref{Main-gammaconvergence}.

\subsection{\texorpdfstring{$\Gamma$}{Gamma}-Convergence of parameterized energies}
Here we establish the $\Gamma$-convergence of the parameterized energies defined over the space of unconstrained functions in different convergence regimes of $s$ and $\delta$. We recall that without introducing additional notation, we assume that the energy ${\cE}_{s,\delta}$ are understood as maps 
${\cE}_{s,\delta}: L^{p}(\Omega_1)\times L^{p}(\Omega_2) \to \mathbb{R}\cup \{\infty\}$; that is, for $s\in (0, 1)$ and $\delta \in (0, 1)$ ${\cE}_{s,\delta}$
is equal to the expression \eqref{energy-section3} for $(u_1, u_2)\in \mathfrak{W}^{s,p}[\delta](\Omega_1) \times {W}^{s,p}(\Omega_2)$ and equal to $+\infty$ otherwise.

\begin{theorem}\label{Gamma-converge-unconstrained}
	We have the following $\Gamma$-convergence results:
	
	\begin{enumerate}[\upshape i)]
		\item \label{item:i} Fix $s\in (0, 1)$. Then as $\delta\to 0$ the  sequence of energies  ${\cE}_{s,\delta}$ $\Gamma$-converges   with respect to $L^{p}(\Omega_1)\times L^{p}(\Omega_2)$ to 
		${\cE}_{s, 0}$ 
		where ${\cE}_{s, 0}(u_1, u_2)$ is as given in \eqref{defn-E(s,0)} for $(u_1, u_2)\in W^{1,p}(\Omega_1, p-sp) \times {W}^{s,p}(\Omega_2)$ and $+\infty$ otherwise.
		
		\item \label{item:ii} Fix $\delta>0$. Then as $s\to 1$  the  sequence of energies  ${\cE}_{s,\delta}$ $\Gamma$-converges   with respect to $L^{p}(\Omega_1)\times L^{p}(\Omega_2)$ to 
		${\cE}_{1, \delta}$ where   ${\cE}_{1, \delta}(u_1, u_2)$ is as given in \eqref{defn-E(1,delta)} for $(u_1, u_2)\in \mathfrak{W}^{1,p}[\delta](\Omega_1) \times {W}^{1,p}(\Omega_2) $ and $+\infty$ otherwise.
		
		\item \label{item:iii} As $s\to 1,$ the sequence of energies ${\cE}_{s, 0}$, from \Cref{item:i},  $\Gamma$-converges with respect to $L^{p}(\Omega_1)\times L^{p}(\Omega_2)$ to 
		${\cE}_{1, 0}$ where
		\[
		{\cE}_{1, 0}(u_1, u_2) =  
		{1\over p}\left(\int_{\Omega_1} \alpha(\bx)|\nabla u_1 (\bx)|^{p} \rmd\bx +\int_{\Omega_2} \beta(\bx)|\nabla u_2 (\bx)|^{p} \rmd\bx \right)\]
		for $(u_1, u_2)\in W^{1,p}(\Omega_1)\times W^{1,p}(\Omega_2) $
		and $+\infty$ otherwise. 

		\item \label{item:iv} As $\delta\to 0$ the sequence of energies ${\cE}_{1, \delta}$ $\Gamma$-converges with respect to $L^{p}(\Omega_1)\times L^{p}(\Omega_2)$ to 
		${\cE}_{1, 0}$ as given in \Cref{item:iii}.
		
		\item \label{item:v} Finally, for any $(s_k, \delta_k) \to (1, 0)$ as $k\to \infty$, the sequence of energies ${\cE}_{s_k, \delta_k}$ $\Gamma$-converges with respect to $L^{p}(\Omega_1)\times L^{p}(\Omega_2)$ to 
		${\cE}_{1, 0}$ as given in \Cref{item:iii}.

	\end{enumerate}

\end{theorem}

The theorem effectively says that the diagram \Cref{fig:variational_convergence_unconstrained} commutes. 
\begin{figure}
\centering
\begin{tikzpicture}[scale=0.75]
\tikzset{to/.style={->,>=stealth',line width=.8pt}}   
\node(v1) at (0,3.5) {\textcolor{red}{$\mathcal{E}_{s,\delta}$}};
\node (v2) at (5.5,3.5) {\textcolor{red}{$\mathcal{E}_{s,0}$}};
\node (v3) at (0,0) {\textcolor{red}{$\mathcal{E}_{1,\delta}$}};
\node (v4) at (5.5,0) {\textcolor{red}{$\mathcal{E}_{1,0}$}};
\draw[to] (v1.east) -- node[midway,above] {\footnotesize{\textcolor{blue}{ Case \ref{item:i}}}}  node[midway,below] {\footnotesize{\textcolor{blue}{$\delta\to0$}}}      
(v2.west);
\draw[to] (v1.south) -- node[midway,left] {\footnotesize{\textcolor{blue}{ Case \ref{item:ii}}}} node[midway,right] {\footnotesize{\textcolor{blue}{$s\to1$}}} (v3.north);
\draw[to] (v3.east) -- node[midway,below] {\footnotesize{\textcolor{blue}{ Case \ref{item:iv}}}} node[midway,above] {\footnotesize{\textcolor{blue}{$\delta\to0$}}} (v4.west);
\draw[to] (v2.south) -- node[midway,right] {\footnotesize{\textcolor{blue}{ Case \ref{item:iii}}}} node[midway,left] {\footnotesize{\textcolor{blue}{$s\to1$}}}(v4.north);
\draw[to] (v1.south east) to[out = 2, in = 180, looseness = 1.2] node[midway] {\footnotesize\hbox{\shortstack[l]{ {\textcolor{blue}{ $\;\;$Case \ref{item:v}}}\\ {\textcolor{blue}{($s\to1,\;\delta\to0$)}} }}} (v4.north west);
\end{tikzpicture}
\caption{Variational convergence of parameterized functionals}
\label{fig:variational_convergence_unconstrained}
\end{figure}
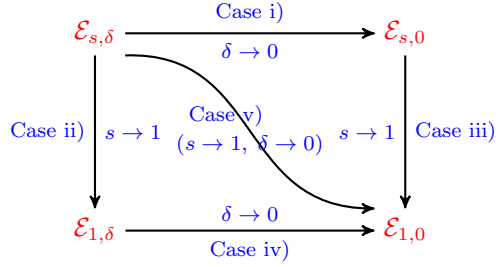

In the remaining part of this section, we prove each of the parts of the theorem separately. 
For each of the $\Gamma$-convergence results, we establish the corresponding liminf
and limsup inequalities. To that end, we prove the following inequality for the nonlocal energy.

\begin{proposition}\label{Prop:LIMINF-INEQUALITY}
	Let $s\in (0, 1]$ and $\alpha\in L^{\infty}(\cD)$ such that for some $\alpha_0>0,$ $\alpha_0 \leq\alpha(\bx) \leq {1\over \alpha_0},$ for all $\bx \in \cD$. Then if $\{u_{n}\}_n$ converges to $u$ in $L^p(\cD)$ and $\delta_n \to 0$ as $n\to\infty$,  then we have
\begin{equation*}
    \begin{aligned}
&		\int_{\cD} \alpha(\bx)|\nabla u(\bx)|^p \sigma(\bx)^{p-sp}\rmd\bx\\
& \qquad \le \liminf_{n\to \infty} \overline{C}_{d,p} \int_{\cD}\int_{B(\bx,\delta_n\sigma(\bx))} \alpha(\bx)\frac{|u_n(\bx)-u_n(\by)|^p}{\delta_n^{d+p}\sigma(\bx)^{d+sp}}\rmd\by \rmd\bx.
\end{aligned}
	\end{equation*}
\end{proposition}

\begin{proof}
	We may assume that the right-hand side of the inequality is finite. Thus, we can apply \Cref{thm:Compactness} to conclude that $u \in W^{1,p}(\cD; p-sp).$
	We prove the inequality in several steps. 
	
	Step 1)  Assume that $\alpha(\bx) = \mathds{1}_{\cB}(\bx)$ where $\cB\subset \cD$ is arbitrary. 
	For $\veps \in (0,1/\kappa_0)$ define the set $\cB^{\veps}:=\{\bx\in \cB:\dist(\bx,\p\cB)> \veps\kappa_0 \diam(\cD) \}$.
	Then using an argument similar to the proof of \Cref{lem:with-phiepsilon}, we can prove the analog of \eqref{eq:Comp:ConvolutionEstimate2} for  $\cB$ with $\varrho$ arbitrary:	\begin{equation*}
		\begin{split}
			\int_{\cB^{\veps} \cap \{ \sigma(\bx) > \varrho \} }  &\int_{B(\bx, \theta(\veps)\delta \sigma(\bx))} \frac{ |K_\veps u(\bx) - K_\veps u(\by)|^p }{ (\theta(\veps)\delta)^{d+p} \sigma(\bx)^{d+sp} } \, \rmd \by \, \rmd \bx \\
			&  \leq \phi(\veps) \int_{\cB \cap \{ \sigma(\bx) > \frac{ \varrho}{\phi(\veps)} \} } \int_{B(\bx, \delta \sigma(\bx))} \frac{ |u(\bx) - u(\by)|^p }{ \delta^{d+p} \sigma(\bx)^{d+sp} } \, \rmd \by \, \rmd \bx.
		\end{split}
	\end{equation*}
	With this, one can argue similarly as in the proof of \Cref{thm:Compactness} to obtain
	\begin{equation}\label{eq:LIMINF-INEQUALITY:Pf1}
		\begin{split}
			&\int_{\cB^\veps} |\grad K_\veps u|^p \sigma(\bx)^{p-sp} \, \rmd \bx \\
			\leq& \overline{C}_{d,p} \liminf_{n\to\infty} \int_{\cB^\veps}\int_{B(\bx,\delta_n\theta(\veps) \sigma(\bx))} \frac{|K_\veps u_n(\bx)-K_\veps u_n(\by)|^p}{(\theta(\veps)\delta_n)^{d+p}\sigma(\bx)^{d+sp}}\rmd\by \rmd\bx\\
			\leq& \overline{C}_{d,p} \phi(\veps) \liminf_{n\to\infty} \int_{\cB}\int_{B(\bx,\delta_n\sigma(\bx))} \frac{|u_n(\bx)-u_n(\by)|^p}{\delta_n^{d+p}\sigma(\bx)^{d+sp}}\rmd\by \rmd\bx.
		\end{split}
	\end{equation}
	
	Now, for $\bz \in B(\mathbf{0},1)$ define $\bszeta_\bz^\veps(\bx) = \bx + \veps \sigma(\bx) \bz$, and we note using Lipschitz continuity of $\dist(\cdot, \p \cB)$ that
	\begin{equation*}
		\mathds{1}_{ \cB }(\bx) \mathds{1}_{ \{ \dist(\bx, \p \cB) < \veps \kappa_0 \diam(\cD) \} } (\bx) 
		\leq \mathds{1}_{\cD} (\bx) \mathds{1}_{ \{ \dist( \bszeta_\bz^\veps(\bx), \p \cB) < 2 \veps \kappa_0 \diam(\cD)  \}  } (\bx).
	\end{equation*}
	Thus, we can use \cite[Lemma 3.2]{scott2024nonlocal} and a change of variables to get
	\begin{equation}\label{eq:LIMINF-INEQUALITY:Pf2}
		\begin{split}
			&\int_{\cB \backslash \cB^\veps} |\nabla K_\veps u(\bx)|^p \sigma(\bx)^{p-sp}\rmd\bx \\
			\leq& \int_{\cB \backslash \cB^\veps} \int_{B(\mathbf{0},1)} \frac{ \psi(|\bz|) (1+\kappa_1 \veps)^p}{ (1-\kappa_0 \veps)^{p-sp} } |\grad u(\bx + \veps \sigma(\bx) \bz)|^p  \sigma(\bx + \veps \sigma(\bx) \bz)^{p-sp} \, \rmd \bz \, \rmd \bx \\
			\le& \tilde{\phi}(\veps) \int_{\cB_{2\veps}} |\nabla u(\bx)|^p \sigma(\bx)^{p-sp}\rmd\bx
		\end{split}
	\end{equation}
	for any $\cB \subset \cD$, where $\tilde{\phi}(\veps):=\frac{(1+\kappa_1 \veps)^p}{(1-\kappa_1 \veps) (1-\kappa_0 \veps)^{p-sp} }$ and $\cB_{2\veps}:=\{\bx\in \cD : \dist(\bx,\p\cB)<2\veps \kappa_0 \diam(\cD) \}$. 
	Combining \eqref{eq:LIMINF-INEQUALITY:Pf1}-\eqref{eq:LIMINF-INEQUALITY:Pf2} and 
	letting $\veps\to 0$ yields 
	\begin{equation*}
		\int_{\cB } |\grad u(\bx)|^p \sigma(\bx)^{p-sp} \, \rmd \bx \leq \overline{C}_{d,p}  \liminf_{n\to\infty} \int_{\cB}\int_{B(\bx,\delta_n\sigma(\bx))} \frac{|u_n(\bx)-u_n(\by)|^p}{\delta_n^{d+p}\sigma(\bx)^{d+sp}}\rmd\by \rmd\bx,
	\end{equation*}
	where we used $\tilde{\phi}(\veps) \int_{ \cB_{2\veps}} |\nabla u(\bx)|^p \sigma(\bx)^{p-sp}\rmd\bx \to 0$ as $\varepsilon \to 0$ since $u\in W^{1,p}(\cD; p-sp)$. 
	
	Step 2) Assume that $\alpha(\bx)$ is a step function and $ \alpha(\bx)= \sum_{i=1}^{\ell} \alpha_i \mathds{1}_{\cB_{i}}(\bx)$, where for all $i=1, 2, \cdots, \ell$, $\alpha_i$ are positive constants and $\cB_{i} \subset \cD$. In this case, apply Step 1) for each $i=1, 2, \cdots, \ell$  to conclude that 
	\begin{align*}
		&\int_{\cD } \alpha(\bx)|\grad  u(\bx)|^p \sigma(\bx)^{p-sp} \, \rmd \bx  \\
		=& \sum_{i=1}^{\ell} \int_{\cB_i } \alpha_i |\grad  u(\bx)|^p \sigma(\bx)^{p-sp} \, \rmd \bx \\
		\leq& \overline{C}_{d,p} \sum_{i=1}^{\ell}\liminf_{n\to\infty} \int_{\cB_i} \int_{B(\bx,\delta_n\sigma(\bx))} \alpha_i\frac{|u_n(\bx)-u_n(\by)|^p}{\delta_n^{d+p}\sigma(\bx)^{d+sp}}\rmd\by \rmd\bx \\
		\leq& \overline{C}_{d,p} \liminf_{n\to\infty} \int_{\cD} \int_{B(\bx,\delta_n\sigma(\bx))} \alpha(\bx)\frac{|u_n(\bx)-u_n(\by)|^{p}}{\delta_n^{p+d}\sigma(\bx)^{d+sp}}\rmd\by \rmd\bx,
	\end{align*} 
	where we have used the superadditivity of the liminf operation. 
	
	Step 3) For $0\leq \alpha(\bx) \in L^{\infty}(\cD)$, we can choose a nondecreasing sequence of  nonnegative simple functions $a_{k}(\bx)$ such that $a_{k} \to \alpha$ in $L^{\infty}(\cD)$ and $0\leq a_k \leq a_{k+1} \leq \alpha$ for all $k$ and $x\in \cD$. We apply Step 2) for each $a_k$ and then monotonicity of the sequence to obtain that 
	\begin{align*}
		&\int_{\cD } a_k(\bx) |\grad  u(\bx)|^p \sigma(\bx)^{p-sp} \, \rmd \bx\\
		\leq& \overline{C}_{d,p} \liminf_{n\to\infty} \int_{\cD}\int_{B(\bx,\delta_n\sigma(\bx))} a_k(\bx) \frac{|u_n(\bx)-u_n(\by)|^p}{\delta_n^{d+p}\sigma(\bx)^{d+sp}}\rmd\by \rmd\bx \\
		\leq& \overline{C}_{d,p} \liminf_{n\to\infty} \int_{\cD}\int_{B(\bx,\delta_n\sigma(\bx))} \alpha(\bx) \frac{|u_n(\bx)-u_n(\by)|^p}{\delta_n^{d+p}\sigma(\bx)^{d+sp}}\rmd\by \rmd\bx.
	\end{align*}
	We then take the limit as $k\to \infty$ and apply Fatou's lemma to complete the proof of the proposition.
\end{proof}

\begin{proof}[Proof of \Cref{item:i} of \Cref{Gamma-converge-unconstrained}]
We fix $s\in (0, 1)$.

{\it Liminf inequality:}
Suppose we have $(u_1, u_2)\in L^{p}(\Omega_1)\times L^{p}(\Omega_2)$ and $(u^{\delta}_1, u^{\delta}_2) \in L^{p}(\Omega_1)\times L^{p}(\Omega_2)$ such that $(u^{\delta}_1, u^{\delta}_2)\to (u_1, u_2)$ in $L^{p}(\Omega_1)\times L^{p}(\Omega_2).$
Without loss of generality, we assume that $\sup_{\delta\in (0,\underline{\delta}_0)}{\cE}_{s, \delta}(u^{\delta}_1, u^{\delta}_2) < \infty$. This implies that 
\[\sup_{\delta\in (0,\underline{\delta}_0)}[u_1^\delta]_{\mathfrak{W}^{s,p}[\delta](\Omega_1)} <\infty \,\text{ and }\, \sup_{\delta\in (0,\underline{\delta}_0)}[u_2^{\delta}]_{W^{s,p}(\Omega_2)} < \infty.\]
Choose a subsequence $u^{\delta_k}_1$ such that $\displaystyle \lim_{ k\to \infty}[u_1^{\delta_k}]^p_{\mathfrak{W}^{s,p}[\delta_k](\Omega_1)} = \liminf_{\delta\to 0} [u_1^\delta]^p_{\mathfrak{W}^{s,p}[\delta](\Omega_1)}. $
By assumption, we have that 
$u_1^{\delta_k}\to u_1$ in $L^{p}(\Omega_1)$, as $k\to \infty$, with the limiting function $u_{1}\in W^{1, p}(\Omega_1, p-sp)$. Moreover, applying \Cref{Prop:LIMINF-INEQUALITY} 
\[
\begin{aligned} &
\int_{\Omega_1} \alpha(\bx)|\nabla u_1(\bx)|^p \sigma(\bx)^{p-sp} \, \rmd\bx\\
& \qquad\le 
\liminf_{k\to \infty} \overline{C}_{d,p} \int_{\Omega_1} \int_{B(\bx,\delta_k\sigma(\bx))} \alpha(\bx)\frac{|u_1^{\delta_k}(\bx)-u_1^{\delta_k}(\by)|^p}{\delta_k^{d+p}\sigma(\bx)^{d+sp}}\, \rmd\by \, \rmd\bx. 
\end{aligned}
\]
Similarly, by the standard compactness embedding of the fractional Sobolev spaces into $L^{p}$ spaces, 
$u_{2}^{\delta_k} \to u_{2}$ in $L^{p}(\Omega_2)$ with $u_{2}\in W^{s,p}(\Omega)$ as $k\to \infty$, and by Fatou's lemma 
\[
\int_{\Omega_2} \int_{\Omega_2}\beta(\bx) \frac{|u_2(\bx)-u_2(\by)|^p}{|\bx-\by|^{d+sp}} \, \rmd \by \, \rmd \bx\leq \liminf_{k\to \infty}\int_{\Omega_2} \int_{\Omega_2}\beta(\bx) \frac{|u^{\delta_k}_2(\bx)-u^{\delta_k}_2(\by)|^p}{|\bx-\by|^{d+sp}} \, \rmd \by \, \rmd \bx.
\]
Combining the above inequalities we get the desired liminf inequality 
\begin{equation*}
\begin{split}
	\liminf_{\delta\to 0} \cE_{s, \delta}(u_{1}^\delta, u_2^\delta) &\geq 
	\cE_{s, 0}(u_1, u_2).
\end{split}
\end{equation*}

{\it Limsup inequality:}
Let $(u_1, u_2)\in L^{p}(\Omega_1) \times L^{p}(\Omega_2)$. If $(u_1, u_{2})\in L^{p}(\Omega_1) \times L^{p}(\Omega_2) $ but not in $ W^{1,p}(\Omega_1, p-sp) \times {W}^{s,p}(\Omega_2) $, then since $\cE_{s, 0}(u_1, u_2) = \infty$ there is nothing to prove.
If $(u_1, u_2)\in W^{1,p}(\Omega_1, p-sp) \times {W}^{s,p}(\Omega_2),$ we can take the constant sequence $(u_1^\delta, u_2^{\delta}) = (u_1, u_2)$ and use 
\Cref{thm:LocalizationOfSeminorm} 
\[
\begin{gathered}
\lim_{\delta\to 0} [u_1]^p_{\mathfrak{W}^{s,p}[\delta](\Omega_1)} = [u_{1}]^p_{W^{1, p}(\Omega_1, p-sp)} 
\end{gathered}
\]
to obtain the limsup inequality.  

\end{proof}
\begin{proof}[Proof of \Cref{item:ii} of \Cref{Gamma-converge-unconstrained}]
We fix $\delta >0.$

{\it Liminf inequality:}
Let $(u_{1}^{s},u_{2}^{s}) \to (u_{1}, u_{2})$ in $L^{p}(\Omega_{1})\times L^{p}(\Omega_2)$. We may assume that 
$\displaystyle \liminf_{s\to 1} {\cE}_{s, \delta}(u_{1}^{s}, u_{2}^{s}) < \infty.$
Up to a subsequence we have that  $\displaystyle\sup_{s\to 1}[u_{1}^{s}]_{\mathfrak{W}^{s,p}[\delta](\Omega_1)} < \infty$ and $\displaystyle \sup_{s\to  1}[u_{2}^{s}]_{W^{s,p}(\Omega_2)} < \infty. $

Now since $u_1^{s}\to u_1$ a.e. in $\Omega_1$ as $s\to 1,$ we have that for almost all $(x, y)\in \Omega_1\times \Omega_1$, we  have 
\[
\frac{|u^s_1(\bx)-u^s_1(\by)|^p}{ \delta^{p} \sigma(\bx)^{sp} } \to \frac{|u_1(\bx)-u_1(\by)|^p}{ \delta^{p} \sigma(\bx)^{p} } \,\,\text{as $s\to 1$. }
\]
It then follows from Fatou's lemma that 
\begin{align*}
\int_{\Omega_1} \int_{B(\bx,\delta\sigma(\bx))}&\alpha(\bx)\frac{|u_1(\bx)-u_1(\by)|^p}{ \delta^{d+p} \sigma(\bx)^{d+p} } \rmd \by \rmd\bx\\
&\le \liminf_{s\to 1} \int_{\Omega_1} \int_{B(\bx,\delta_k\sigma(\bx))} \alpha(\bx)\frac{|u_1^{s}(\bx)-u_1^{s}(\by)|^p}{\delta_k^{d+p} \sigma(\bx)^{d+sp}} \rmd\by \rmd\bx.
\end{align*}
For the sequence $ \{u_2^{s}\}$ we may apply the result of \cite{Munoz2021} -- which generalizes that of \cite{bourgain2001,ponce2004} -- to conclude that 
\[
{1\over p} \int_{\Omega_2} \beta(\bx) |\nabla u_2(\bx)|^{p} \rmd\bx \leq \liminf_{s\to 1}\frac{\kappa_{d,s,p}}{p} \int_{\Omega_2} \int_{\Omega_2}\beta(\bx) \frac{|u^{s}_2(\bx)-u^{s}_2(\by)|^p}{|\bx-\by|^{d+sp}} \, \rmd \by \, \rmd \bx.
\]
Combining the above inequalities we get the the liminf inequality.

{\it Limsup inequality:} For any $(u_1, u_2)\in  \mathfrak{W}^{1, p}[\delta](\Omega_1)\times W^{1, p}(\Omega_2)$, we may take the constant sequence $(u_1^{s},u_{2}^{s}) = (u|_{\Omega_1}, u|_{\Omega_2}) = (u_1, u_2).$ 
For $(u_1, u_2)\notin \mathfrak{W}^{1, p}[\delta](\Omega_1)\times W^{1, p}(\Omega_2)$, there is nothing to show since ${\cE}_{1, \delta}(u_{1}^{s}, u_{1}^{s}) = \infty$.

\end{proof}
\begin{proof}[Proof of \Cref{item:iii} of \Cref{Gamma-converge-unconstrained}]
As before we will establish the liminf and the limsup inequalities.

{\it Liminf inequality:}
Let $(u_{1}^{s},u_{2}^{s}) \to (u_{1}, u_{2})$ in $L^{p}(\Omega_{1})\times L^{p}(\Omega_2)$ as $s\to 1$. We may assume that 
$\liminf_{s\to 1} {\cE}_{s, 0}(u_{1}^{s}, u_{2}^{s}) < \infty.$
Up to a subsequence we have that  $\sup_{s\in(0, 1)}[u_{1}^{s}]_{W^{1,p}(\Omega_1; p-sp)} < \infty$ and $\sup_{s\in (0,1)}[u_{2}^{s}]_{W^{s,p}(\Omega_2)} < \infty. $
As in the liminf inequality for Case \ref{item:ii}, for the sequence $ \{u_2^{s}\}$ we may apply the result of  \cite{Munoz2021} (see also \cite{bourgain2001,ponce2004}) to conclude that 
\[
{1\over p} \int_{\Omega_2} \beta(\bx) |\nabla u_2(\bx)|^{p} \rmd\bx \leq \liminf_{s\to 1}\frac{\kappa_{d,s,p}}{p} \int_{\Omega_2} \int_{\Omega_2}\beta(\bx) \frac{|u^{s}_2(\bx)-u^{s}_2(\by)|^p}{|\bx-\by|^{d+sp}} \, \rmd \by \, \rmd \bx. 
\]
As for the other functional sequence, we get $\sup_{s < 1} [u_{1}^{s}]_{W^{s, p}(\Omega_1)} <\infty$ by applying \Cref{thm:WeightedEst}. By the compactness results of \cite{bourgain2001,ponce2004}, we have $u_{1} \in W^{1,p}(\Omega_1)$. Also, $u_{1}^{s}$ weakly converges to $u_1$ in $W^{1, p}_{loc}(\Omega_1).$
To complete the proof, it suffices to prove that 
\[
{1\over p} \int_{\Omega_1} \alpha(\bx) |\nabla u_1(\bx)|^{p} \rmd\bx \leq \liminf_{s\to 1}\frac{1}{p} \int_{\Omega_1} \alpha(\bx) |\nabla u^{s}_1|^{p} \sigma(\bx)^{p-sp} \rmd \bx. 
\] 

To that end, we apply the same incremental argument as in the proof of \Cref{Prop:LIMINF-INEQUALITY} to treat the case of general $\alpha(\bx)$, and we only need to show that for any $V \subset \Omega_1$ we have 
\[
{1\over p} \int_{V}  |\nabla u_1(\bx)|^{p} \rmd\bx \leq \liminf_{s\to 1}\frac{1}{p} \int_{V}  |\nabla u^{s}_1|^{p} \sigma(\bx)^{p-sp} \rmd \bx. \]
Since $u_1 \in W^{1, p}(\Omega_1)$, this inequality can in turn be proved if we establish 
\begin{equation}\label{eq:partc:liminfPf1}
\int_{A}  |\nabla u_1(\bx)|^{p} \rmd\bx \leq \liminf_{s\to 1} \int_{V} |\nabla u^{s}_1|^{p} \sigma(\bx)^{p-sp} \rmd \bx,
\end{equation}
for any $A\Subset V$.
Since $\grad u_1^s \rightharpoonup \grad u_1$ weakly in $L^p(A)$, and since $\sigma(\bx)^{1-s} \to 1$ strongly in $L^\infty(V)$, it follows that $\grad u_1^s \sigma^{1-s} \rightharpoonup \grad u_1$ weakly in $L^p(A)$. Therefore
\begin{equation*}
\int_{A}  |\nabla u_1(\bx)|^{p} \rmd\bx 
\leq \liminf_{s\to 1} \int_{A} |\nabla u^{s}_1|^{p} \sigma(\bx)^{p-sp} \rmd \bx
\leq \liminf_{s\to 1} \int_{V} |\nabla u^{s}_1|^{p} \sigma(\bx)^{p-sp} \rmd \bx,
\end{equation*}
and \eqref{eq:partc:liminfPf1} is established.

{\it Limsup inequality:} For any $(u_1, u_2)\in  W^{1, p}(\Omega_1)\times W^{1, p}(\Omega_2)$, we may take the constant sequence $(u_1^{s},u_{2}^{s}) = (u|_{\Omega_1}, u|_{\Omega_2}) = (u_1, u_2)$.
For $(u_1, u_2)\notin  W^{1, p}(\Omega_1)\times W^{1, p}(\Omega_2)$, there is nothing to prove since  ${\cE}_{1, 0}(u_{1}, u_{1}) = \infty$. 

\end{proof}

\begin{proof}[Proof of \Cref{item:iv} of \Cref{Gamma-converge-unconstrained}]
	The $\Gamma$-convergence in \Cref{item:iv} follows by applying the $\Gamma$-convergence argument of $u\mapsto [u]_{\mathfrak{W}^{1, p}[\delta](\Omega_1)}$ proved in \cite{scott2024nonlocal}. 
\end{proof}

\begin{proof}[Proof of \Cref{item:v} of \Cref{Gamma-converge-unconstrained}]
As before we will establish the liminf and the
limsup inequalities.

{\it Liminf inequality:} We apply \Cref{thm:Compactness} and the remark after it to finish the proof. Assume we have $(u_1, u_2)\in L^p(\Omega_1)\times L^p(\Omega_2)$ and $(u_1^{s_k,\delta_k}, u_2^{s_k,\delta_k})\to (u_1, u_2)$ in $L^p(\Omega_1)\times L^p(\Omega_2)$ as $k\to\infty$, we show that 
\[{\cE}_{1,0}(u_1,u_2)\le \liminf_{k\to\infty} {\cE}_{s_k, \delta_k} (u_1^{s_k,\delta_k}, u_2^{s_k,\delta_k}).\]
Without loss of generality, suppose the right-hand side is finite. Then up to a subsequence (not relabeled) we have $\sup_{k}{\cE}_{s_k, \delta_k} (u_1^{s_k,\delta_k}, u_2^{s_k,\delta_k}) <\infty$. Then 
\[\sup_k [u_1^{s_k,\delta_k}]_{\mathfrak{W}^{s_k,p}[\delta_k](\Omega_1)}<\infty \text{ and }  \sup_k [u_2^{s_k,\delta_k}]_{W^{s_k,p}(\Omega_2)}<\infty. \] 
For the sequence $ \{u_2^{s_k,\delta_k}\}_k$, as in the proof of \Cref{item:ii}, we may apply the result of  \cite{Munoz2021} (see also \cite{bourgain2001,ponce2004}) to conclude that 
\[
{1\over p} \int_{\Omega_2} \beta(\bx) |\nabla u_2(\bx)|^{p} dx \leq \liminf_{k\to\infty}\frac{\kappa_{d,s_k,p}}{p} \int_{\Omega_2} \int_{\Omega_2}\beta(\bx) \frac{|u^{s_k,\delta_k}_2(\bx)-u^{s_k,\delta_k}_2(\by)|^p}{|\bx-\by|^{d+s_kp}} \, \rmd \by \, \rmd \bx.
\]
For the sequence $ \{u_1^{s_k,\delta_k}\}_k$, we first note that
for any fixed $0 < s < 1$ there exists $K \in \bbN$ such that $s < s_k$ for all $k \geq K$. We have the estimate 

\begin{equation*}
	\begin{split}
		&\int_{\Omega_1}\int_{B(\bx, \delta_k\sigma(\bx))}\alpha(\bx) \frac{|u^{s_k,\delta_k}_1(\bx)-u^{s_k,\delta_k}_1(\by)|^p}{\delta_k^{d+p}\sigma(\bx)^{d+sp}}\rmd\by \rmd\bx \\
		&\le C(s)\int_{\Omega_1}\int_{B(\bx, \delta_k\sigma(\bx))}\alpha(\bx) \frac{|u^{s_k,\delta_k}_1(\bx)-u^{s_k,\delta_k}_1(\by)|^p}{\delta_k^{d+p}\sigma(\bx)^{d+s_k p}}\rmd\by \rmd\bx,
	\end{split}
\end{equation*}
where $C(s) = \max\{1, \text{diam}(\Omega_1)^{(1-s)p}\}$. Then we apply \Cref{Prop:LIMINF-INEQUALITY} to the left-hand side of the above inequality to obtain
\begin{equation}
	\begin{split}
		&\int_{\Omega_1}\alpha(\bx)|\nabla u_1(\bx)|^p\sigma(\bx)^{p-sp}\rmd\bx \\
		\leq& C(s) \overline{C}_{d,p} \liminf_{k \to \infty}         \int_{\Omega_1}\int_{B(\bx, \delta_k\sigma(\bx))}\alpha(\bx) \frac{|u^{s_k,\delta_k}_1(\bx)-u^{s_k,\delta_k}_1(\by)|^p}{\delta_k^{d+p}\sigma(\bx)^{d+s_k p}}\rmd\by \rmd\bx.
	\end{split}
\end{equation}
Since $C(s)\to 1$ as $s\to 1$, by Fatou's lemma, letting $s\to 1$ yields
\begingroup\makeatletter\def\f@size{9}\check@mathfonts
\def\maketag@@@#1{\hbox{\m@th\large\normalfont#1}}%
    \begin{equation*}
        \int_{\Omega_1}\alpha(\bx)|\nabla u_1(\bx)|^p \rmd\bx\le \liminf_{k\to\infty} \overline{C}_{d,p} \int_{\Omega_1}\int_{B(\bx, \delta_k\sigma(\bx))}\alpha(\bx) \frac{|u^{s_k,\delta_k}_1(\bx)-u^{s_k,\delta_k}_1(\by)|^p}{\delta_k^{d+p}\sigma(\bx)^{d+s_k p}}\rmd\by \rmd\bx.
    \end{equation*}
\endgroup

{\it Limsup inequality:} Let $(u_1, u_2)\in L^{p}(\Omega_1)\times L^{p}(\Omega_2)$. If $(u_1, u_2)\notin W^{1,p}(\Omega_1)\times W^{1,p}(\Omega_2)$, then there is nothing to prove. If $(u_1, u_2)\in W^{1,p}(\Omega_1)\times W^{1,p}(\Omega_2),$ then we take the constant sequence and apply  \Cref{thm:LocalizationOfSeminorm} with both $s\to 1$ and $\delta \to 0$. 
\end{proof}

\subsection{Convergence of minimizers: Proof of \texorpdfstring{\Cref{Main-gammaconvergence}}{Theorem 1.2}}

For any $\delta\in (0, \underline{\delta}_0)$, $s\in (0, 1)$, $p\in (1, \infty)$ such that $sp>1$, we recall that the spaces $\mathfrak{X}_{s,\delta}$ defined in \eqref{SPACES} and in \Cref{Main-gammaconvergence}. Given a linear functional $\mathfrak{f}\in \mathfrak{X}'_{s, \delta}$ we also recall the parametrized functionals    \[\cF_{s, \delta}(v_1,v_2) = \cE_{s,\delta}(v_1,v_2) - \langle\mathfrak{f}, (v_1, v_2)\rangle_{\mathfrak{X}'_{s, \delta}, \mathfrak{X}_{s, \delta}}\]
for $(u_1, u_2)\in \mathfrak{X}_{s,\delta}$ and $\infty $ otherwise. The proof of the $\Gamma$-convergence in various convergence regimes follows essentially the same tracks as the proof of \Cref{Gamma-converge-unconstrained}. We only need to ensure that the limiting functions belong to the function spaces that give the proper transmission conditions. The proof of convergence we give below emphasizes this without repeating many of the arguments used in the proof of \Cref{Gamma-converge-unconstrained}. For \Cref{Main-gammaconvergence}, we take  $\mathfrak{f} =  (f_1, f_2)\in L^{p'}(\Omega_1)\times L^{p'}(\Omega_2)\subset \mathfrak{X}'_{s, \delta}$, where $p'$ is the H\"older conjugate of $p.$
\begin{proof}[Proof of \Cref{Main-gammaconvergence} Case \ref{item:a}]
	We begin by fixing $s\in (0, 1)$, and consider the energy functional $\cF_{s,\delta}$ for $\delta\in (0, \underline{\delta}_0)$, and $\underline{\delta}_0$ as given in \eqref{bound-for-delta}. 
	To show the $\Gamma$- convergence of  $\cF_{s,\delta}$ to with respect to $L^{p}(\Omega_1)\times L^{p}(\Omega_2)$ to 
	$\cF_{s, 0}$, where we recall as defined in \eqref{defn-E(s,0)} that 
	\[
	\cF_{s, 0}(u_1, u_2) =  
	[u_1]^{p}_{W^{1,p}(\Omega_1, p-sp)} + [u_2]^{p}_{W^{s,p}(\Omega_2)} - \int_{\Omega_1} f_{1}(\bx) u_1(\bx) \rmd\bx - \int_{\Omega_2} f_{2}(\bx) u_{2}(\bx) \rmd\bx
	\]
	for $(u_1, u_2)\in W^{1,p}(\Omega_1, p-sp) \times {W}^{s,p}(\Omega_2) $ and $\infty$ otherwise. 
	we prove both the liminf and limsup inequalities.
	
	\textit{Liminf inequality:}
	Suppose we have $(u_1, u_2)\in L^{p}(\Omega_1)\times L^{p}(\Omega_2)$ and $(u^{\delta}_1, u^{\delta}_2) \in L^{p}(\Omega_1)\times L^{p}(\Omega_2)$ such that $(u^{\delta}_1, u^{\delta}_2)\to (u_1, u_2)$ in $L^{p}(\Omega_1)\times L^{p}(\Omega_2).$
	Without loss of generality, we assume that $\sup_{\delta\in (0,\underline{\delta}_0)}\cF_{s, \delta}(u^{\delta}_1, u^{\delta}_2) < \infty$. Then $(u^\delta_1, u^{\delta}_2)\in \mathfrak{X}_{s,\delta}$ and that 
	\[\sup_{\delta\in (0,\underline{\delta}_0)}[u_1^\delta]_{\mathfrak{W}^{s,p}[\delta](\Omega_1)} <\infty \text{ and } \sup_{\delta\in (0,\underline{\delta}_0)}[u_2^{\delta}]_{W^{s,p}(\Omega_2)} < \infty.\]
	Now if we show that $(u_1, u_2)\in \mathfrak{X}_{s,0}$, then applying a similar argument as in the proof of \Cref{item:i} of \Cref{Gamma-converge-unconstrained}, we can show that
	\[
	\cF_{s,0}(u_1, u_2) \leq \liminf_{\delta\to 0} \cF_{s, \delta}(u^\delta_{1}, u^{\delta}_{2}). 
	\]

To show that $(u_1, u_2)\in \mathfrak{X}_{s,0}$, we need to establish that $u_1$ and $u_2$ have a matching trace on $\Sigma$, and a vanishing trace on $\partial \Omega_1\setminus \Sigma$ and $\partial \Omega_2\setminus \Sigma$,  respectively.  The sequence ${u^{\delta}_2}$ converges weakly in $W^{s,p}(\Omega_2)$ to $u_{2}$ and since $T_{1}u^{\delta}_2 = 0$,  we may apply the continuity of the trace operator with respect to the weak topology of the spaces to conclude $u_2$ has vanishing trace on $\partial \Omega_2\setminus \Sigma$. 
To show the matching trace of $u_1$ on $\Sigma$ and that $T_1 u_1 = 0$ on $\partial \Omega_1\setminus \Sigma$,  we first note by \Cref{thm:Convolution:DerivativeEstimate}  that since $u^{\delta}_1\in \mathfrak{W}^{s,p}[\delta](\Omega_1),$ $K_{\delta} u^{\delta}_{1}\in W^{1, p}(\Omega_1, p-sp)$ with the estimate 
\[
[K_{\delta} u^{\delta}_{1}]_{W^{1, p}(\Omega_1, p-sp)} \leq C [u_1^\delta]_{\mathfrak{W}^{s,p}[\delta](\Omega_1)}. 
\]
where $C$ is independent of $\delta$. 
Therefore, $K_{\delta} u^{\delta}_{1}$ is a bounded sequence in $W^{1, p}(\Omega_1, p-sp)$ and weakly converges in $W^{1, p}(\Omega_1, p-sp)$ and strongly in $L^{p}(\Omega_1)$. Since $K_{\delta} u^{\delta}_{1}$ and $ u^{\delta}_{1}$ have the same $L^{p}$ limit, which we can show using \Cref{thm:diffuKu:Weighted}, we must have that the limit is $u_1$ and it is in $ W^{1, p}(\Omega_1, p-sp).$ By \Cref{cor:Trace:NonlocalSpace}, we also have 
that 
$
T_{1} K_{\delta} u^\delta_1 = T_{1} u^\delta_1
$ almost everywhere on $\partial \Omega_1$. Thus $T_{1} K_{\delta} u^\delta_1 = 0$ on $\partial \Omega_1\setminus \Sigma$, and therefore the weak limit $u_1$ has vanishing trace on $\partial \Omega_1\setminus \Sigma$ as well. Moreover, it follows using \Cref{rmk:StrongConvTrace} that 
\[
(T_{1}u_1-T_2u_2)|_{\Sigma} = \lim_{\delta\to 0} (T_{1} K_{\delta} u_{1}^{\delta} - T_{2} u^{\delta}_{2})|_{\Sigma} =  \lim_{\delta\to 0} (T_{1} u_{1}^{\delta} - T_{2} u^{\delta}_{2})|_{\Sigma} = 0,
\]
verifying the matching of traces on $\Sigma.$

\textit{Limsup inequality:} The proof is similar to the proof of \Cref{item:i} of \Cref{Gamma-converge-unconstrained}, using a constant sequence. 
\end{proof}

\begin{proof}[Proof of \Cref{Main-gammaconvergence} Case \ref{item:b}]

Fix $\delta\in (0, \underline{\delta}_0)$. We demonstrate the $\Gamma$- convergence of $\cF_{s,\delta}$ with respect to $L^{p}(\Omega_1)\times L^{p}(\Omega_2)$ to $\cF_{1,\delta}$ as $s\to 1$ where we recall that 
\[
\cF_{1,\delta} = \cE_{1, \delta}(u_1, u_2)
- \int_{\Omega_1} f_{1}(\bx) u_1(\bx) \rmd\bx - \int_{\Omega_2} f_{2}(\bx) u_{2}(\bx) \rmd\bx
\]
for $(u_1, u_2)\in \mathfrak{X}_{1, \delta}$ and $\infty$ otherwise. 

\textit{Liminf inequality:}
Suppose that we have $(u_1, u_2)\in L^{p}(\Omega_1)\times L^{p}(\Omega_2)$ and $(u^{s}_1, u^{s}_2) \in L^{p}(\Omega_1)\times L^{p}(\Omega_2)$ such that $(u^{s}_1, u^{s}_2)\to (u_1, u_2)$ in $L^{p}(\Omega_1)\times L^{p}(\Omega_2),$ as $s\to 1$. 
We may assume that $\liminf_{s\to 1} \cF_{s,\delta}(u^s_1, u^s_2) < \infty$. Then we may assume that we have
\[
\sup_{s\in (0,1)}[u_1^s]_{\mathfrak{W}^{s,p}[\delta](\Omega_1)} <\infty \text{ and } \sup_{s\in (0,1)}[u_2^{s}]_{W^{s,p}(\Omega_2)} < \infty.
\]
Again, if we show that $(u_1, u_{2})\in \mathfrak{X}_{1, \delta}$, then applying a similar argument as in the proof
of \Cref{item:ii} of \Cref{Gamma-converge-unconstrained}, we can show that
\[
\cF_{1, \delta}(u_1, u_{2})\leq \liminf_{s\to 1} \cF_{s,\delta}(u_1, u_2). 
\]
To show that $(u_1, u_2)\in \mathfrak{X}_{1, \delta}$, we need to establish that $u_1$ and $u_2$ have a matching trace on $\Sigma$ and a vanishing trace on $\partial \Omega_1\setminus \Sigma$ and $\partial \Omega_2\setminus \Sigma$,  respectively.  Fix $s_0 > {1\over p}$. Then for any $s_0 < s< 1$,
\[
\sup_{s_0< s< 1} [u_1^{s}]_{\mathfrak{W}^{s_0,p}[\delta](\Omega_1)} \leq C(s_0) \sup_{s_0< s< 1}[u^{s}_1]_{\mathfrak{W}^{s,p}[\delta](\Omega_1)} \leq C. 
\]
Therefore, $u_1^{s}$ converges weakly in $\mathfrak{W}^{s_0,p}[\delta](\Omega_1)$ to $u_1$ as $s\to 1$. Using the weak continuity of the trace operator (see \Cref{rmk:StrongConvTrace}), we see that $T_1u_1^{s} \to T_1u_1$ strongly in $L^{p}(\partial \Omega_1)$. Thus, 
$T_{1}u_1 = 0$ on $\partial \Omega_1 \setminus \Sigma.$ We can make similar argument to conclude that $u^s_2$ converges weakly in $W^{s_0, p}(\Omega_2)$ to $u_2$, and so 
$T_2u_2^{s} \to T_2u_2$ strongly in $L^{p}(\partial \Omega_2)$. Thus, 
$T_{2}u_2 = 0$ on $\partial \Omega_2 \setminus \Sigma.$ Moreover, 
\[
(T_{1}u_1-T_2u_2)|_{\Sigma}  =  \lim_{\delta\to 0} (T_{1} u_{1}^{\delta} - T_{2} u^{\delta}_{2})|_{\Sigma} = 0
\]

\textit{Limsup inequality:} The proof is similar to the proof of \Cref{item:ii} of \Cref{Gamma-converge-unconstrained},
using a constant sequence.

\end{proof}


Next, we establish the $\Gamma$-convergence of $\cE_{s,0}$ to $\cE_{1,0}$ where 
\[
\cE_{1,0}(u) = {1\over p}\int_{\Omega} \Lambda(\bx)|\nabla u(\bx)|^{p} \rmd\bx \quad \text{for $u\in W^{1,p}(\Omega)$}
\]
and $\infty $ otherwise,  where $\Lambda(\bx) = \alpha(\bx) \mathds{1}_{\Omega_1}(\bx) + \beta(\bx)\mathds{1}_{\Omega_2}(\bx)$.  
We begin with a compactness result for sequences of functions in $\mathfrak{X}_{s,0}$, defined in \eqref{Admissible-delta0}. 

\begin{lemma}\label{thm:Case2:Compactness}
Let a sequence $\{(u_1^s,u_2^s)\}_s \in \mathfrak{X}_{s,0}$ satisfy $\sup_{ s \in (0,1) } \cE_{s,0}(u_1^s,u_2^s)  < \infty$. Then $\{(u_1^s,u_2^s)\}_s$ is precompact in $L^p(\Omega_1) \times L^p(\Omega_2)$. Moreover, any limit point $(u_1,u_2)$ belongs to $\mathfrak{X}_{1,0}$. 
\end{lemma}

\begin{proof}
By assumption and applying \Cref{PI-onX}, we have that 
\begin{equation*}
	\vnorm{u_1^s}_{W^{1,p}(\Omega_1;p-sp)}^p + \vnorm{u_2^s}_{W^{s,p}(\Omega_2)}^p \leq C \cE_{s,0}(u_1^s,u_2^s),
\end{equation*}
where the constant $C$ is uniform in $s$.
Then by \Cref{thm:WeightedEst} we have 
\begin{equation*}
	\vnorm{u_1^s}_{W^{s,p}(\Omega_1)}^p + \vnorm{u_2^s}_{W^{s,p}(\Omega_2)}^p \leq C \cE_{s,0}(u_1^s,u_2^s).
\end{equation*}
Therefore, we obtain from \cite{ponce2004,bourgain2001} that $\{u_i^s\}_s$ is precompact in $L^p(\Omega_i)$, $i = 1,2$, and any limit point $u_i$ belongs to $W^{1,p}(\Omega_i)$.  
Notice that, for $\bar{s}<1$ with $\bar{s}p>1$
\[
u_{1}^{s}\rightharpoonup u_{1} \quad \text{weakly in $W^{1, p}(\Omega; p-\bar{s} p)$ \,\, and\,\, }u_{2}^{s}\rightharpoonup u_{2} \quad \text{weakly in $W^{\bar{s}, p}(\Omega)$.}
\]  We use \Cref{rmk:StrongConvTrace} to conclude that $T_1 u_1 = 0$ on $\p \Omega_1 \setminus \Sigma$, $T_2 u_2 = 0$ on $\p \Omega_2 \setminus \Sigma$, and 
\[
\vnorm{ (T_{1} u_{1} - T_2 u_{2})|_{\Sigma} }_{L^p(\p \Omega)} = \lim_{s\to 1} \vnorm{ (T_{1} u^{s}_{1} - T_2 u^{s}_{2})|_{\Sigma} }_{L^p(\p \Omega)} = 0.
\]
From this, we conclude that $u(\bx):= u_{1}|_{\Omega_1}(\bx) + u_{2}|_{\Omega_2}(\bx) \in W^{1,p}_0(\Omega)$, and therefore $(u_{1}, u_2)\in \mathfrak{X}_{1,0}.$
\end{proof}

\begin{proof}[Proof of \Cref{Main-gammaconvergence} Case \ref{item:c}]
As before we will prove the liminf and the limsup inequalities. 

\textit{Liminf inequality:}
Let $(u_{1}^{s},u_{2}^{s}) \to (u_{1}, u_{2})$ in $L^{p}(\Omega_{1})\times L^{p}(\Omega_2)$ as $s\to 1$. We may assume that 
$\liminf_{s\to 1} \cE_{s, 0}(u_{1}^{s}, u_{2}^{s}) < \infty.$
Applying \Cref{thm:Case2:Compactness}, we have that the limit $(u_{1}, u_{2})$ can be identified with $u = u_{1}|_{\Omega_1} +  u_{2}|_{\Omega_2} \in \mathfrak{X}_{1,0}$. An argument similar to that in the proof of  part \ref{item:iii} of \Cref{Gamma-converge-unconstrained} shows that 
\begin{equation*}
	\begin{split}
		{1\over p}\int_{\Omega}\Lambda(\bx)|\nabla u(\bx)|^{p} \rmd\bx
		&\leq \liminf_{s\to 1}\cE_{s, 0}(u_{1}^{s}, u_{1}^{s}).
	\end{split}
\end{equation*}

\textit{Limsup inequality:} For any $u\in W_0^{1, p}(\Omega)$, we may take the constant sequence $(u_1^{s},u_{2}^{s}) = (u|_{\Omega_1}, u|_{\Omega_2}).$ Note that $u_1^{s}$ and $u_2^{s}$ will have a matching trace on the interface. The necessary inequality follows from the nonlocal characterization of Sobolev spaces, \cite{bourgain2001, ponce2004}.   
\end{proof}

\begin{proof}[Proof of \Cref{Main-gammaconvergence} Case \ref{item:d}]

Let $(u_{1}^{\delta},u_{2}^{\delta}) \to (u_{1}, u_{2})$ in $L^{p}(\Omega_{1})\times L^{p}(\Omega_2)$ as $\delta\to 0$. We may assume that $\liminf_{\delta\to0}\cF_{1,\delta}(u^\delta_1, u^\delta_2) < \infty$. Thus, $(u_{1}^{\delta},u_{2}^{\delta}) \in \cX_{1, \delta}$ and $\sup_{0<\delta<1}[u^{\delta}_1]_{\mathfrak{W}^{1, p}[\delta](\Omega_1)}<\infty$ and $u_{2}^{\delta}$ is bounded in $W^{1, p}(\Omega)$. 

The  $\Gamma$-convergence will follow, using arguments similar to the proof of \Cref{item:iv} of \Cref{Gamma-converge-unconstrained}, if we show that $(u_1, u_2)\in \mathfrak{X}_{1,0}$. To that end, applying \Cref{thm:Convolution:DerivativeEstimate}, we see that $K_{\delta} u_1^{\delta}$ is a bounded sequence in $ W^{1,p}(\Omega_1)$ that has the same trace as $u_1^{\delta}$, by \Cref{cor:Trace:NonlocalSpace}. Now $K_\delta u^{\delta}$ converges weakly in $W^{1, p}(\Omega_1)$ and strongly in $L^{p}(\Omega_1)$ to $u_1$. By the weak continuity of the trace operator (see \Cref{rmk:StrongConvTrace}), we have that $T_1u_1=0$ on $\partial \Omega_1\setminus \Sigma.$ A similar argument applies to the sequence $u^{s}_2$ to obtain that  $T_2 u_2=0$ on $\partial \Omega_2\setminus \Sigma$.  Moreover, 
\[
(T_{1}u_1-T_2u_2)|_{\Sigma} = \lim_{\delta\to 0} (T_{1} K_{\delta} u_{1}^{\delta} - T_{2} u^{\delta}_{2})|_{\Sigma} =  \lim_{\delta\to 0} (T_{1} u_{1}^{\delta} - T_{2} u^{\delta}_{2})|_{\Sigma} = 0.
\]
\end{proof}

\begin{proof}[Proof of \Cref{Main-gammaconvergence} Case \ref{item:e}]

Assume we have $(u_1, u_2)\in L^p(\Omega_1)\times L^p(\Omega_2)$ and $(u_1^{s_k,\delta_k}, u_2^{s_k,\delta_k})\to (u_1, u_2)$ in $L^p(\Omega_1)\times L^p(\Omega_2)$ as $k\to\infty$ where $s_k\to 1$ and $\delta_k\to 0$. It suffices to work on the case when  
$ \liminf_{k\to\infty} {\cE}_{s_k, \delta_k} (u_1^{s_k,\delta_k}, u_2^{s_k,\delta_k})<\infty$. Without relabeling a subsequence, it follows that $(u_1^{s_k,\delta_k}, u_2^{s_k,\delta_k})\in \mathfrak{X}_{s_k, \delta_k}$. If we show that $(u_1, u_2)\in \mathfrak{X}_{1, 0}$, then the remaining portion of the argument follows as in the proof of Case \ref{item:v} of \Cref{Gamma-converge-unconstrained}. To this end, fix $s_0\in (1/p, 1)$. Then there exists $k_0\geq 1$ such that $s_k\in (s_0, 1)$ and 
\[
\sup_{k\geq k_0} [u_{1}^{s_k, \delta_k}]_{\mathfrak{W}^{s_0, p}[\delta_k](\Omega_1)} \leq C(s_0)  \sup_{k} [u_{1}^{s_k, \delta_k}]_{\mathfrak{W}^{s_k, p}[\delta_k](\Omega_1)}\leq C<\infty.
\]
From \Cref{thm:Compactness} and \Cref{Prop:LIMINF-INEQUALITY}, it follows that $\{u_{1}^{s_k, \delta_k}\}_k$ converges weakly in $W^{1, p}(\Omega_1, p-s_0p)$ and strongly in $L^p(\Omega_1)$ to $u_1$. Thus $u_1$ is in $W^{1, p}(\Omega_1, p-s_0p)$ for any $s_0>1/p$, and therefore  $u_1\in W^{1, p}(\Omega_1)$. Moreover, by weak continuity of the trace operator (see \Cref{rmk:StrongConvTrace}), we have that $T_1 u^{s_k, \delta_k}_1 \to T_1u_1$ strongly in $L^{p}(\partial \Omega_1)$, and so $T_1 u_1 = 0$ on $\partial \Omega_1\setminus \Sigma$. Similar argument can also be  applied for the sequence $\{u^{s_k, \delta_k}_2\}_k$ to conclude that $u_2\in W^{1, p}(\Omega_2)$ with $T_2 u^{s_k, \delta_k}_2 \to T_2u_2$ strongly in $L^{p}(\partial \Omega_2)$. Thus, $T_2u_2=0$ on $\partial \Omega_2\setminus \Sigma$ and that 
\[
\vnorm{ (T_{1} u_{1} - T_2 u_{2})|_{\Sigma} }_{L^p(\p \Omega)} = \lim_{s\to 1} \vnorm{ (T_{1} u^{s}_{1} - T_2 u^{s}_{2})|_{\Sigma} }_{L^p(\p \Omega)} = 0.
\]
From this, we conclude that $u(\bx):= u_{1}|_{\Omega_1}(\bx) + u_{2}|_{\Omega_2}(\bx) \in W^{1,p}_0(\Omega)$, and therefore $(u_{1}, u_2)\in \mathfrak{X}_{1,0}.$
\end{proof}






In fact, the limit cases of variational problems in \Cref{Main-gammaconvergence} are well-posed. An application of the Poincar\'e inequality \Cref{PI-onX} in the cases $\delta = 0$ and/or $s = 1$ give the following:

\begin{proposition}\label{minimizer-existence}
Under the assumptions of \Cref{Main-gammaconvergence}, the limiting energy functionals $\cF_{s,0}$, $\cF_{1,\delta}$, and $\cF_{1,0}$
have unique minimizers. 
\end{proposition}



We conclude this section with the proof of the convergence of minimizers of $\cF_{s,\delta}$.

		


\begin{proof}[Proof of \Cref{Convergence-minimizer}]
Since the existence and uniqueness of minimizers of functionals have been shown in \Cref{thm:wellposedness} and \Cref{minimizer-existence}, by \cite[Corollary~7.24]{dalmaso1993introduction} it remains to show that $\cF_{s,\delta}$ is equicoercive for Case \ref{item:a}, Case \ref{item:c}, Case \ref{item:d}, and Case \ref{item:e}, which are consequences of Poincar\'e inequalities for $\mathfrak{X}_{s,\delta}$ with constants independent of $s$ and $\delta$ and compactness results for $\mathfrak{X}_{s,\delta}$ in these parameter convergence regimes. First, \Cref{PI-onX} ensures the Poincar\'e constants are independent of $s$ and $\delta$ in all cases. Second, the compactness result for $\mathfrak{W}^{s,p}[\delta](\Omega_1)$ as $\delta \to 0$ and $s\to 1$ or $s$ fixed is established in \Cref{thm:Compactness} which is required for Case \ref{item:a}, Case \ref{item:d} and Case \ref{item:e}. In Case \ref{item:c}, the analogous result is the compact embedding of $W^{1,p}(\Omega_1, p-sp)$ into $L^p(\Omega_1)$ given by \Cref{thm:FxnSpProp:Weighted} item 4). To complete the compactness results for $\mathfrak{X}_{s,\delta}$, we need the well-known compact embedding of $W^{s,p}(\Omega_2)$ into $L^p(\Omega_2)$ when $s\in (0, 1]$ is fixed for Case \ref{item:a} and Case \ref{item:d}, and the compactness result for $W^{s,p}(\Omega_2)$ when $s\to 1$ for Case \ref{item:c} and Case \ref{item:e}, which is a corollary of the former compact embedding with fixed $s$ and the continuous embedding $W^{s,p}(\Omega_2)\hookrightarrow W^{s',p}(\Omega_2)$ for $s>s'$. This completes the proof.
\end{proof}

	\begin{proof}[Proof of Proposition \ref{Convergence-minimizer-weakLp}] Case \ref{item:b}'.    
    We prove the the result for the special case when $\alpha(\bx) = \beta(\bx)=1$  by establishing   convergence of minimizers of $\cF_{s, \delta}$ to a minimizer of $\cF_{1,\delta}$ in the weak $L^{p}(\Omega_1)\times L^{p}(\Omega_2)$ topology, which will be shown as a consequence of the $\Gamma$-convergence of  $\cF_{s,\delta}$ to $\cF_{1,\delta}$ in the same weak topology. The general case for positive coefficients $\alpha(\bx)$ and $\beta(\bx)$ can be shown using similar techniques as in the proof of \Cref{Prop:LIMINF-INEQUALITY}.
		As before we establish the liminf and the limsup inequalities with respect to the weak topology. Since we are considering the energy for a fixed $\delta,$ below we suppress the dependence of the functions on $\delta$.
		
		{\it Liminf inequality:} Suppose that $(u_1^{s}, u_{2}^{s})$ converges to $(u_1, u_2)$ weakly in $L^{p}(\Omega_1) \times L^{p}(\Omega_2)$.
        Without loss of generality we may assume that 
		\[\sup_{0<s<1} \mathcal{F}_{s,\delta}(u_1^{s},u_{2}^{s}) < \infty.\]
        We then claim that $(u_1, u_2)\in \mathfrak{X}_{1,\delta}$ and that 
        \[
        \mathcal{F}_{1,\delta}(u_1,u_{2})\leq \lim_{s\to 1}\mathcal{F}_{s,\delta}(u_1^{s},u_{2}^{s}). 
        \]
		Now, it  follows from the boundedness of $\mathcal{F}_{s,\delta}(u_1^{s},u_{2}^{s})$ in $s$ that  
		\[
		\sup_{s}[u_{1}^{s}]_{\mathfrak{W}^{s,p}[\delta](\Omega_1)}<\infty,\quad \sup_{s}[u_{2}^{s}]_{W^{s,p}(\Omega_2)} < \infty.
		\]
		Using the super-additive property of the liminf, it suffices to show the liminf inequality for each of the parts of the energy.  
		Now, since $u_1^s$ converges weakly in $L^p$ to $u_1$  as $s\to 1$, then for any $\epsilon>0$ sufficiently small, $K_\varepsilon u_1^s (\bx)$ converges pointwise to $K_\varepsilon u_1 (\bx)$ as $s\to 1$ (notice $\psi_\varepsilon(\bx,\cdot) \in L^{p'}(\Omega_1)$ for any $\bx \in \Omega_1$). Then by Fatou's lemma 
		\[
		[K_\varepsilon u_1]^p_{\mathfrak{W}^{1,p}[\theta(\varepsilon)\delta](\Omega_1)} \leq \liminf_{s\to 1} [K_\varepsilon u_1^s]^p_{\mathfrak{W}^{s,p}[\theta(\varepsilon)\delta](\Omega_1)}, 
		\]
		where $\theta(\varepsilon)$ is as in \Cref{lem:with-phiepsilon}. Applying  \eqref{eq:Comp:ConvolutionEstimate2} in \Cref{lem:with-phiepsilon} with $\varrho =0$, the right hand side of the above inequality can be estimated as 
		\[
		\liminf_{s\to 1} [K_\varepsilon u_1^s]^p_{\mathfrak{W}^{s,p}[\theta(\varepsilon)\delta](\Omega_1)} \leq \phi(\varepsilon)\liminf_{s\to 1}[u_{1}^{s}]_{\mathfrak{W}^{s,p}[\delta](\Omega_1)}.
		\]
		Now
		$K_\varepsilon u_1$ converges to $u_1$ in $L^p(\Omega)$ as $\varepsilon\to 0$, and therefore a.e. (up to a subsequence). Applying Fatou's lemma again we obtain that \[[u_1]^p_{\mathfrak{W}^{1,p}[\delta](\Omega_1)}\leq \liminf_{\varepsilon\to0}
		[K_\varepsilon u_1]^p_{\mathfrak{W}^{1,p}[\theta(\varepsilon)\delta](\Omega_1)} \leq
		\liminf_{\varepsilon\to0} \phi(\varepsilon)\liminf_{s\to 1} [u_1^s]^p_{\mathfrak{W}^{s,p}[\delta](\Omega_1)}. \] Noting that $\phi(\varepsilon)\to 1$ as $\varepsilon \to 0$, we get the liminf inequality of the nonlocal energy after multiplying both sides by the constant ${1\over p}$. For the liminf inequality for the fractional part of the parametrized energy, choose a subsequence $u_2^{s}$ (not labeled) so that $\liminf_{s\to 1}[u_{2}^{s}]_{W^{s,p}(\Omega_2)}=\lim_{s\to 1}[u_{2}^{s}]_{W^{s,p}(\Omega_2)}$. Then from the boundedness of $[u_{2}^{s}]_{W^{s,p}(\Omega_2)}$ in $s$, we can select a $L^{p}$-strongly converging subsequence by \cite{ponce2004,bourgain2001}. Moreover, the limit belongs to $W^{1, p}(\Omega_2)$ and, by uniqueness of the limit, must be $u_{2}$. Moreover, applying  \cite{ponce2004,bourgain2001} again, we have 
		\[
		[u_2]_{W^{1,p}(\Omega_2)} \leq \liminf_{s\to 1}[u_2^{s}]_{W^{s,p}(\Omega_2)}.
		\]
		To show that $(u_1, u_2)\in \mathfrak{X}_{1, \delta}$, we essentially repeat the argument used in the proof of Case \ref{item:b}.  We need to establish that $u_1$ and $u_2$ have a matching trace on $\Sigma$ and a vanishing trace on $\partial \Omega_1\setminus \Sigma$ and $\partial \Omega_2\setminus \Sigma$,  respectively.  Fix $s_0 > {1\over p}$. Then for any $s_0 < s< 1$,
\[
\sup_{s_0< s< 1} [u_1^{s}]_{\mathfrak{W}^{s_0,p}[\delta](\Omega_1)} \leq C(s_0) \sup_{s_0< s< 1}[u^{s}_1]_{\mathfrak{W}^{s,p}[\delta](\Omega_1)} \leq C. 
\]
Therefore, up to a subsequence $u_1^{s}$ converges weakly in $\mathfrak{W}^{s_0,p}[\delta](\Omega_1)$ to the same limit, $u_1$, as $s\to 1$. Using the weak continuity of the trace operator (see \Cref{rmk:StrongConvTrace}), we see that $T_1u_1^{s} \to T_1u_1$ strongly in $L^{p}(\partial \Omega_1)$. Thus, 
$T_{1}u_1 = 0$ on $\partial \Omega_1 \setminus \Sigma.$ We can make similar argument to conclude that up to a subsequence $u^s_2$ converges weakly in $W^{s_0, p}(\Omega_2)$ to $u_2$, and so 
$T_2u_2^{s} \to T_2u_2$ strongly in $L^{p}(\partial \Omega_2)$. Thus, 
$T_{2}u_2 = 0$ on $\partial \Omega_2 \setminus \Sigma.$ Moreover, 
\[
(T_{1}u_1-T_2u_2)|_{\Sigma}  =  \lim_{\delta\to 0} (T_{1} u_{1}^{\delta} - T_{2} u^{\delta}_{2})|_{\Sigma} = 0.
\]

		{\it Limsup inequality}.  
        Taking the constant sequence $(u_1^{s}, u_{2}^{s}) = (u_1, u_2)\in \mathfrak{X}_{1,\delta}$,
         the result follows.
         
		Finally, the convergence of the sequence of minimizers with respect to the weak $L^p$-topology follows from the $\Gamma$-convergence result. 
	\end{proof}

\section{Extensions}\label{sec:extensions}
We note that the framework and techniques developed in earlier sections can be extended in various directions.
For example, we may consider coupled problems for the same energy \cref{energy} but with a different set of feasible spaces. Meanwhile, we may also consider different variants of the energy.

On the latter, it is trivial to see that the earlier results hold if we replace the energy $\cE_{s,\delta}(u_1,u_2)$  in \cref{energy} by a more general $\hat{\cE}_{s,\delta}(u_1,u_2)$ defined as
\begin{equation*}
\begin{split}
\hat{\cE}_{s,\delta}(u_1,u_2) &:=     
\frac{\overline{C}_{d,p}}{p}
\int_{\Omega_1} \int_{ B(\bx,\delta \sigma(\bx)) } \frac{\alpha(\bx) \sigma(\bx)^{p-sp}}{ (\delta \sigma(\bx))^d }
\rho_1\left(
\frac{|u_1(\bx)-u_1(\by)|}{ \delta \sigma(\bx) }\right)
\, \rmd \by \, \rmd \bx \\
&\qquad + \frac{ \kappa_{d,s,p} }{p} \int_{\Omega_2} \int_{\Omega_2}
\frac{\beta(\bx)}{|\bx-\by|^{d+sp-p}} \rho_2\left(
\frac{ |u_2(\bx)-u_2(\by)|}{|\bx-\by|} \right) \, \rmd \by \, \rmd \bx
\end{split}
\end{equation*}
where $\rho_1$ and $\rho_2$ represent two energy potentials satisfied, for example, 
$$c_1r^p\leq \rho_i(r)\leq c_2 r^p,  \quad \forall\, r\in \mathbb{R}_{+},\; i=1,2 $$
for some positive constants $c_1$ and $c_2$, independent of $s$ and $\delta$.

The admissible function space may also take a more general form. For instance, we may consider a more general transmission condition than 
\eqref{eq:TransmissionContinuity} that imposes certain continuity or jump conditions on
a portion of the interfaces, together with possibly inhomogeneous Dirichlet data on the boundary. To give a specific example, we introduce three closed hypersurfaces $\{\Sigma_i\}_{i=0,1,2}$ with nonzero measure on the interface and the boundary, 
$$\Sigma_0 \subset \Sigma, \quad
\Sigma_i \subset \partial\Omega_i\setminus \Sigma_0, \; i=1,2.$$
Then we define, for $g_i\in  W^{s-{1/p}, p}(\Sigma_i)$, $i=0,1,2$,
\begin{equation}\label{eq:Transmissionjump}
\left\{\begin{aligned}
& (T_1 u_1 - T_2 u_2) \mathds{1}_{\Sigma_0} = g_0 \quad \scH^{d-1}\text{-a.e. on } \Sigma_0, \\
& T_i u_i 
\mathds{1}_{\Sigma_i} = g_i \quad \scH^{d-1}\text{-a.e. on } \Sigma_i,\; i=1,2. 
\end{aligned}
\right.
\end{equation}
A more general admission set may then be defined as
\begin{equation*}
\hat{\mathfrak{X}}_{s,\delta} := \left\{ (u_1, u_2) \in \mathfrak{W}^{s,p}[\delta](\Omega_1) \times W^{s,p}(\Omega_2) \text{ with } \eqref{eq:Transmissionjump} \text{ satisfied}.
\right\}
\end{equation*}
with the generalized variational problem  
\begin{equation*}
\begin{array}{rl}
(u_1,u_2) & = \arg\min_{
	(v_1,v_2) \in \hat{\mathfrak{X}}_{s,\delta}} \tilde{\cF}_{s,\delta}(v_1,v_2), \\
& \text{ where } \hat{\cF}_{s, \delta}(v_1,v_2) := \hat{\cE}_{s,\delta}(v_1,v_2) - \langle \hat{\mathfrak{f}}_\delta, (v_1, v_2)\rangle_{\hat{\mathfrak{X}}'_{s, \delta}, \hat{\mathfrak{X}}_{s, \delta}}
\end{array}
\end{equation*}
$\hat{\mathfrak{X}}_{s, \delta,0}$ is the same as
$\mathfrak{X}_{s, \delta}$
but with $g_i=0$, for $i=0,1,2$. Also,  the duality pairing 
$\langle\cdot, \cdot\rangle_{\hat{\mathfrak{X}}'_{s, \delta,0}, \hat{\mathfrak{X}}_{s, \delta,0}}  $  is
on  $\hat{\mathfrak{X}}_{s, \delta,0}$ and $\hat{\mathfrak{X}}'_{s, \delta,0}$, and $ \hat{\mathfrak{f}}_\delta\in \hat{\mathfrak{X}}'_{s, \delta,0}$ forms a sequence that is uniformly bounded in 
$\hat{\mathfrak{X}}'_{s, \delta,0}$ and converges to a distribution $f$ in the dual space of the local limit of $\hat{\mathfrak{X}}_{s, \delta,0}$.

Note that if $\Sigma_0=\emptyset$, then the two problems on $\Omega_1$ and $\Omega_2$ respectively become completely decoupled. Our interest is thus in the case with a nontrivial $\Sigma_0$. 
Furthermore, we could replace the {\em hard constraint} $(T_1 u_1 - T_2 u_2) \mathds{1}_{\Sigma_0} = g_0$ 
on $\Sigma_0$ introduced in 
\eqref{eq:Transmissionjump}
by a {\em soft} one, enforced via a penalty term in the energy, that is,
\begin{equation*}
\tilde{\cE}_{s,\delta}(u_1,u_2) := 
\hat{\cE}_{s,\delta}(u_1,u_2)+ 
\frac{1}{\varepsilon^2}
\|(T_1 u_1 - T_2 u_2) - g_0\|_{ W^{s-{1/p}, p}(\Sigma_0)}^p.
\end{equation*}
The problem associated with the hard constraint (that is, the jump condition on $\Sigma_0$) can be recovered in the limit as $\veps\to 0$. 

The examples stated above can all be analyzed in the same fashion as the case studied in earlier sections. 

\section*{Acknowledgment} {The project was initiated at a SQuaRE at the American Institute for Mathematics (AIM). The authors thank AIM for providing a supportive and mathematically
enriching environment.}

\bibliographystyle{siamplain}
\bibliography{refs}

\begin{thebibliography}{10}

\bibitem{acosta2022local}
{\sc G.~Acosta, F.~Bersetche, and J.~D. Rossi}, {\em Local and nonlocal
  energy-based coupling models}, SIAM Journal on Mathematical Analysis, 54
  (2022), pp.~6288--6322.

\bibitem{acosta2023domain}
{\sc G.~Acosta, F.~M. Bersetche, and J.~D. Rossi}, {\em A domain decomposition
  scheme for couplings between local and nonlocal equations}, Computational
  Methods in Applied Mathematics, 23 (2023), pp.~817--830.

\bibitem{Bogdon-Dyda2011}
{\sc K.~Bogdan and B.~Dyda}, {\em The best constant in a fractional {H}ardy
  inequality}, Mathematische Nachrichten, 284 (2011), pp.~629--638,
  \url{https://doi.org/https://doi.org/10.1002/mana.200810109}.

\bibitem{JuanPablo-Ciarlet2025}
{\sc J.~P. Borthagary and P.~C. Jr}, {\em On some coupled local and nonlocal
  diffusion models}, https://arxiv.org/pdf/2505.19765,  (2025).

\bibitem{bourgain2001}
{\sc J.~Bourgain, H.~Brezis, and P.~Mironescu}, {\em Another look at {S}obolev
  spaces}, IOS Press, Amsterdam, 2001, pp.~439--455.

\bibitem{bourgain2002limiting}
{\sc J.~Bourgain, H.~Brezis, and P.~Mironescu}, {\em Limiting embedding
  theorems for {$W^{s,p}$} when $s \uparrow 1$ and applications}, Journal
  d'Analyse Math{\'e}matique, 87 (2002), p.~77.

\bibitem{Brez11}
{\sc H.~Brezis}, {\em Functional Analysis, Sobolev Spaces and Partial
  Differential Equations}, Universitext, Springer, New York, NY, 2011.

\bibitem{capodaglio2020energy}
{\sc G.~Capodaglio, M.~D’Elia, P.~Bochev, and M.~Gunzburger}, {\em An
  energy-based coupling approach to nonlocal interface problems}, Computers \&
  Fluids, 207 (2020), p.~104593.

\bibitem{dalmaso1993introduction}
{\sc G.~Dal~Maso}, {\em An {{Introduction}} to {{$\Gamma$-Convergence}}},
  Birkh{\"a}user, Boston, MA, 1993,
  \url{https://doi.org/10.1007/978-1-4612-0327-8}.

\bibitem{DLST21}
{\sc M.~D'Elia, X.~Li, P.~Seleson, X.~Tian, and Y.~Yu}, {\em A review of
  local-to-nonlocal coupling methods in nonlocal diffusion and nonlocal
  mechanics}, Journal of Peridynamics and Nonlocal Modeling, 4 (2022),
  pp.~1--50.

\bibitem{Delia2015a}
{\sc M.~D'Elia, M.~Perego, P.~Bochev, and D.~Littlewood}, {\em A coupling
  strategy for nonlocal and local diffusion models with mixed volume
  constraints and boundary conditions}, Computers and Mathematics with
  applications, 71(11) (2015), pp.~2218--2230.

\bibitem{DLLT18}
{\sc Q.~Du, X.~H. Li, J.~Lu, and X.~Tian}, {\em A quasi-nonlocal coupling
  method for nonlocal and local diffusion models}, SIAM Journal on Numerical
  Analysis, 56 (2018), pp.~1386--1404.

\bibitem{Du2022Fractional}
{\sc Q.~Du, T.~Mengesha, and X.~Tian}, {\em Fractional {H}ardy-type and trace
  theorems for nonlocal function spaces with heterogeneous localization},
  Analysis and Applications, 20 (2022), pp.~579--614.

\bibitem{du2024weighted}
{\sc Q.~Du and J.~M. Scott}, {\em Weighted nonlocal operators and their
  applications in semi-supervised learning}, arXiv preprint arXiv:2412.16109,
  (2024).

\bibitem{dyda2006comparability}
{\sc B.~Dyda}, {\em On comparability of integral forms}, Journal of
  Mathematical Analysis and Applications, 318 (2006), pp.~564--577.

\bibitem{dyda2023fractional}
{\sc B.~Dyda, J.~Lehrb{\"a}ck, and A.~V. V{\"a}h{\"a}kangas}, {\em Fractional
  {P}oincar{\'e} and localized {H}ardy inequalities on metric spaces}, Advances
  in Calculus of Variations, 16 (2023), pp.~867--884.

\bibitem{Fabes-Kenig-Serapioni}
{\sc C.~E.~K. Eugene B.~Fabes and R.~P. Serapioni}, {\em The local regularity
  of solutions of degenerate elliptic equations}, Communications in Partial
  Differential Equations, 7 (1982), pp.~77--116,
  \url{https://doi.org/10.1080/03605308208820218}.

\bibitem{Foss2021}
{\sc M.~Foss}, {\em Traces on general sets in {$R^{n}$} for functions with no
  differentiability requirements}, SIAM Journal on Mathematical Analysis, 53
  (2021), pp.~4212--4251.

\bibitem{Frank2010}
{\sc R.~L. Frank and R.~Seiringer}, {\em Sharp fractional {H}ardy inequalities
  in half-spaces}, in Around the Research of Vladimir Maz'ya I: Function
  Spaces, A.~Laptev, ed., Springer New York, New York, NY, 2010, pp.~161--167.

\bibitem{Gal2017}
{\sc C.~G. Gal and M.~Warma}, {\em Nonlocal transmission problems with
  fractional diffusion and boundary conditions on non-smooth interfaces},
  Communications in Partial Differential Equations, 42 (2017), pp.~579--625.

\bibitem{Garriz-Ignat2021}
{\sc A.~G\'{a}rriz and L.~I. Ignat}, {\em A non-local coupling model involving
  three fractional {L}aplacians}, Bulletin of Mathematical Sciences, 11 (2021),
  p.~2150007, \url{https://doi.org/10.1142/S1664360721500077}.

\bibitem{Garriz2020}
{\sc A.~G\'arriz, F.~Quir\'os, and J.~D. Rossi}, {\em Coupling local and
  nonlocal evolution equations}, Calculus of Variations and Partial
  Differential Equations, 59 (2020), p.~24.

\bibitem{gol2009weighted}
{\sc V.~Gol’dshtein and A.~Ukhlov}, {\em Weighted {S}obolev spaces and
  embedding theorems}, Transactions of the American Mathematical Society, 361
  (2009), pp.~3829--3850.

\bibitem{gurka1988continuous}
{\sc P.~Gurka and B.~Opic}, {\em Continuous and compact imbeddings of weighted
  {S}obolev spaces. {I}}, Czechoslovak Mathematical Journal, 38 (1988),
  pp.~730--744.

\bibitem{hurri2013fractional}
{\sc R.~Hurri-Syrj{\"a}nen and A.~V. V{\"a}h{\"a}kangas}, {\em On fractional
  {P}oincar{\'e} inequalities}, Journal d'Analyse Math{\'e}matique, 120 (2013),
  pp.~85--104.

\bibitem{KiMa10}
{\sc B.~Kilic and E.~Madenci}, {\em Coupling of peridynamic theory and the
  finite element method}, J. Mechanics of Materials and Structures, 5 (2010),
  pp.~707--733.

\bibitem{kim2007trace}
{\sc D.~Kim et~al.}, {\em Trace theorems for {S}obolev-{S}lobodeckij spaces
  with or without weights}, J. Funct. Spaces Appl, 5 (2007), pp.~243--268.

\bibitem{kriventsov2015regularity}
{\sc D.~Kriventsov}, {\em Regularity for a local--nonlocal transmission
  problem}, Archive for Rational Mechanics and Analysis, 217 (2015),
  pp.~1103--1195.

\bibitem{kufner1980weighted}
{\sc A.~Kufner}, {\em Weighted {S}obolev spaces}, Teubner Texte Zur Mathematik,
  Springer, 1980.

\bibitem{Leoni2023}
{\sc G.~Leoni}, {\em A First Course in Fractional Sobolev Spaces}, Graduate
  Studies in Mathematics, American Mathematical Society, Providence, RI,
  1st~ed., 2023.

\bibitem{LOSS20101369}
{\sc M.~Loss and C.~Sloane}, {\em Hardy inequalities for fractional integrals
  on general domains}, Journal of Functional Analysis, 259 (2010),
  pp.~1369--1379,
  \url{https://doi.org/https://doi.org/10.1016/j.jfa.2010.05.001}.

\bibitem{mingione2003singular}
{\sc G.~Mingione}, {\em The singular set of solutions to non-differentiable
  elliptic systems.}, Archive for Rational Mechanics \& Analysis, 166 (2003).

\bibitem{Munoz2021}
{\sc J.~Muñoz}, {\em Generalized {P}once’s inequality}, Journal of
  Inequalities and Applications, 2021 (2021), p.~11,
  \url{https://doi.org/10.1186/s13660-020-02543-1}.

\bibitem{nekvinda1993characterization}
{\sc A.~Nekvinda}, {\em Characterization of traces of the weighted {S}obolev
  space {$W^{1,p}(\Omega, d_M^\epsilon)$} on {$M$}}, Czechoslovak Mathematical
  Journal, 43 (1993), pp.~695--711.

\bibitem{Nezza2012Hitchhikers}
{\sc E.~D. Nezza, G.~Palatucci, and E.~Valdinoci}, {\em Hitchhiker's guide to
  the fractional {S}obolev spaces}, Bull. Sci. Math., 136 (2012), pp.~521--573.

\bibitem{ponce2004}
{\sc A.~C. Ponce}, {\em An estimate in the spirit of {P}oincar{\'e}'s
  inequality}, J. Eur. Math. Soc, 6 (2004), pp.~1--15.

\bibitem{ponce2004new}
{\sc A.~C. Ponce}, {\em A new approach to {{Sobolev}} spaces and connections to
  {$\Gamma$}-convergence}, Calculus of Variations and Partial Differential
  Equations, 19 (2004), pp.~229--255,
  \url{https://doi.org/10.1007/s00526-003-0195-z}.

\bibitem{DosSantos2021}
{\sc B.~C.~D. Santos, S.~M. Oliva, and J.~D. Rossi}, {\em A local/nonlocal
  diffusion model}, Applicable Analysis,  (2021), pp.~1--34.

\bibitem{scott2024nonlocal}
{\sc J.~M. Scott and Q.~Du}, {\em Nonlocal problems with local boundary
  conditions {I}: function spaces and variational principles}, SIAM Journal on
  Mathematical Analysis, 56 (2024), pp.~4185--4222.

\bibitem{Seleson2013a}
{\sc P.~Seleson, S.~Beneddine, and S.~Prudhomme}, {\em A force based coupling
  scheme for peridynamics and classical elasticity}, Computational Materials
  Science, 66 (2013), pp.~34--49.

\bibitem{Seleson2015a}
{\sc P.~Seleson, Y.~D. Ha, and S.~Beneddine}, {\em Concurrent coupling of bond
  based peridynamics and the {N}avier equation of classical elasticity by
  blending}, Journal for Multiscale Computational Engineering, 13 (2015),
  pp.~91--113.

\bibitem{Stein1970}
{\sc E.~M. Stein}, {\em Singular Integrals and Differentiability Properties of
  Functions (PMS-30), Volume 30}, Princeton University Press, 1970.

\bibitem{TaTiDu19}
{\sc Y.~Tao, X.~Tian, and Q.~Du}, {\em Nonlocal models with heterogeneous
  localization and their application to seamless local-nonlocal coupling},
  Multiscale Modeling \& Simulation, 17 (2019), pp.~1052--1075.

\bibitem{TiDu17}
{\sc X.~Tian and Q.~Du}, {\em Trace theorems for some nonlocal function spaces
  with heterogeneous localization}, SIAM Journal on Mathematical Analysis, 49
  (2017), pp.~1621--1644.

\bibitem{Triebel1995Interpolation}
{\sc H.~Triebel}, {\em Interpolation Theory, Function Spaces, Differential
  Operators}, Barth, Heidelberg, 2nd, revised and enlarged~ed., 1995.

\end{thebibliography}
\appendix

\section{Fractional Sobolev spaces}
Although the results in this section are already known quantitatively, we collect the statements here in a way to highlight certain aspects, e.g. the explicit dependence of constants as well as the specification of parameters to best suit our setting.

\subsection{Trace inequality}

For $1<p<\infty$ and $s\in (0, 1)$ with $sp>1,$ and a bounded Lipschitz domain $\cD$, the identification of the trace space of $W^{s,p}(\cD)$ by $W^{s-1/p,p}(\p \cD)$ is well known,  see \cite[Theorem 9.39]{Leoni2023} for the existence of the trace as a bounded linear operator $T : W^{s,p}(\cD) \to W^{s-1/p,p}(\p \cD)$ as well as the existence of an extension operator from $W^{s-1/p,p}(\p \cD)$  to $W^{s,p}(\cD)$. For our purpose, we would like to identify the dependence of the boundedness constant as a function of $d$, $s$, $p$ and the domain $\cD$. 
The following theorem restates a part of \cite[Theorem 9.39]{Leoni2023} while clarifying the dependence of the constant on $s$. 

\begin{theorem}\label{thm:Trace:Fractional}
	For $1<p<\infty$ and $s\in (0, 1)$ with $sp>1,$ and a bounded Lipschitz domain $\cD$, there exists a constant $C>0$ depending only on $d$, $p$, and $\cD$ such that
	\begin{equation*}
		\vnorm{T u}_{ W^{s-1/p,p}(\p \cD) } \leq \frac{C}{ sp-1 } \vnorm{u}_{ W^{s,p}(\cD) }, \qquad \forall u \in W^{s,p}(\cD).
	\end{equation*}
\end{theorem}

\begin{proof}
    The proof of this theorem is exactly the same as the proof of \cite[Theorem 9.39]{Leoni2023} after accounting for the constant $\kappa_{d,s,p}$ used in the definition of the seminorm for $W^{s,p}(\cD)$. The strategy of proof is to first establish the result for the half-space $\mathbb{R}^{d}_{+}$, and then use a partition of unity argument to complete the proof for bounded domains with Lipschitz boundary. The boundedness of the trace operator on $W^{s,p}(\mathbb{R}^{d}_{+})$ critically depends on a one-dimensional fractional Hardy inequality. The Hardy inequality used in \cite[Corollary 1.91]{Leoni2023} does not explicitly specify the dependence of the constant on $s$, $p$, and $d$ and this lack of specificity propagates throughout the proof of \cite[Theorem 9.39]{Leoni2023}. Since our interest is to quantify the dependence of the constant on the parameters we instead use a sharp fractional Hardy-type inequality proved in \cite{Bogdon-Dyda2011, Frank2010,LOSS20101369}. We select the version in \cite[Theorem 2.6]{LOSS20101369}; for any $0\leq a<b\leq \infty$ and $f\in C^{\infty}_c((a, b))$, we have 
     \[
    \int_{a}^{b}\int_{a}^{b} \frac{|f(x)-f(x)|^{p}}{|x-y|^{1 + sp}}d x \geq D_{s,p} \int_{a}^{b} \frac{|f(x)|^{p}}{\left(\min\{(x-a), (b-x)\}\right)^{sp}} d x
    \]
    where the constant $D_{s,p}$ is sharp and is given by the formula
    \[
    D_{s,p} = 2 \int_{0}^{1}\frac{|1-r^{{sp-1\over p}}|^{p}}{(1-r)^{1 + sp}} d r.  
    \]
    Using the convexity of the function $r\mapsto 1-r^{{sp-1\over p}}$ on $(0,1)$, we have 
    \[
    1-r^{{sp-1\over p}} \geq {sp-1\over p}(1-r), 
    \]
    and therefore
    \[
     D_{s,p} \geq 2\left({sp-1\over p}\right)^{p} \int_{0}^{1}\frac{(1-r)^p}{(1-r)^{1+sp}} dr = 2\left({sp-1\over p}\right)^{p} {1\over p(1-s)} = \left({sp-1\over p}\right)^{p} {1\over \kappa_{1, s, p}},
    \]
    where $\kappa_{1, s, p}$ is as in \eqref{constant-A}. Rewriting the fractional Hardy inequality in terms of the seminorm we are using in this paper, we obtain that 
    \[
    [f]_{W^{s, p}((a,b))}^{p} \geq \left({sp-1\over p}\right)^{p} \int_{a}^b\frac{|f(x)|^{p}}{\left(\min\{(x-a), (b-x)\}\right)^{sp}} d x.
    \]
    Notice in particular, as $s\to 1$, using \cite{bourgain2001}, the inequality is consistent with the well-known Hardy's inequality for classical Sobolev functions. 

    Now that we have established explicit constants in this inequality, the remaining part of the proof follows the argument used in the proof of \cite[Theorem 9.39]{Leoni2023}. We omit the details.  
\end{proof}

\subsection{Poincar\'e inequality}

The goal of this subsection is to prove the following Poincar\'e inequality with constant independent of $s$.

\begin{theorem}\label{thm:Poincare:Fractional}
	Let $1 \leq p < \infty$ and $s \in (0,1)$ with $sp > 1$, and let $\cD \subset \bbR^d$ be a bounded Lipschitz domain. Let $\p \cD_D \subset \p \cD$ be a closed set such that $\scH^{d-1}(\p \cD_D) > 0$. Then there exists a constant $C$ depending only on $d$, $p$, $\cD$ and $\p \cD_D$ such that
	\begin{equation*}
		\vnorm{u}_{L^p(\cD)} \leq C [u]_{W^{s,p}(\cD)}, \qquad \forall u \in W^{s,p}_{0,\p \cD_D}(\cD).
	\end{equation*}
\end{theorem}

The proof relies on several results collected from various sources.
From \cite[Fact on page 4]{bourgain2002limiting}, along with the technique used to prove \cite[(4.3)]{mingione2003singular} but with the unit cube in place of the unit ball for $s\in (0, 1/2)$, we have the following Poincar\'e inequality (without boundary conditions) on a cube with the desired scaling on the constant:

\begin{theorem}\label{thm:BBM:PoincareOnCube}
	Let $1 \leq p < \infty$, let $s \in (0,1)$, and let $Q = [0,1]^d \subset \bbR^d$ be the unit cube. Then there exists a constant $C$ depending only on $d$ and $p$ such that
	\begin{equation*}
		\int_Q |u(\bx) - u_Q |^p \, \rmd \bx \leq C [u]_{W^{s,p}(Q)}^p, \qquad \forall u \in W^{s,p}(Q),
	\end{equation*}
	where $u_Q := \fint_Q u(\bx) \, \rmd \bx$.
\end{theorem}

From this inequality, we can obtain the following improved Poincar\'e inequality:
\begin{lemma}\label{lma:FractionalPoincare:Improved}
	Let $1 \leq p < \infty$, let $s \in (0,1)$, and let $Q \subset \bbR^d$ be a cube with side length denoted $\ell(Q)$. Then there exists a constant $C$ depending only on $d$ and $p$ such that
	\begin{equation*}
		\int_{Q} |u(\bx) - u_Q|^p \, \rmd \bx \leq C \ell(Q)^{sp} \kappa_{d,s,p} \int_Q \int_{Q \cap B(\bx,\ell(Q)/8)} \frac{|u(\bx)-u(\by)|^p}{|\bx-\by|^{d+sp}} \, \rmd \by \, \rmd \bx,
	\end{equation*}
	for all $u \in W^{s,p}(Q)$.
\end{lemma}

\begin{proof}
    The proof is very similar to that of \cite[Lemma 2.2]{hurri2013fractional}.
	After scaling and translation we may assume that $Q$ is the unit cube in $\bbR^d$. For $k \in \bbN$, subdivide $Q$ into $2^{kd}$ congruent subcubes; denote this collection by $\mathcal{W}_k$. Choose $k$ large enough so that, for all $Q_0 \in \mathcal{W}_{k-1}$, $Q_0 \subset B(\bx,1/8)$ for all $\bx \in Q_0$.
	For cubes $Q_i$, $Q_j \in \mathcal{W}_k$, denote $G_{ij} = Q_i \cup Q_j$ whenever $Q_i$ and $Q_j$ share a common face, and the empty set otherwise; note that the case $i=j$ is included here.
	
	Now fix $G_{ij} = G$, and define $\wt{Q} \subset Q$ to be a cube with side length $2 \ell(Q_i)$ that contains $G$. Then by choice of $k$, it still holds that $\wt{Q} \subset B(\bx,1/8)$ for all $\bx \in \wt{Q}$.
	Then H\"older's and Minkowski's inequalities imply
	\begin{equation*}
		\vnorm{u-u_G}_{L^p(G)} 
		\leq (1+|G|^{-1/p}) \vnorm{u - u_{\wt{Q}} }_{L^p(G)} 
		\leq C(d,p) \vnorm{u - u_{\wt{Q}} }_{L^p(\wt{Q})}.
	\end{equation*}
	Therefore, by \Cref{thm:BBM:PoincareOnCube} applied to $\wt{Q}$
	\begin{equation}\label{eq:FractionalPoincare:Improved:Pf1}
		\begin{split}
			\vnorm{u-u_G}_{L^p(G)}^p &\leq C(d,p) \ell( \wt{Q} )^{sp} [u]_{W^{s,p}(\wt{Q})}^p \\
			&\leq C(d,p) \kappa_{d,s,p} \int_Q \int_{Q \cap B(\bx,1/8) } \frac{|u(\bx)-u(\by)|^p}{|\bx-\by|^{d+sp}} \, \rmd \by \, \rmd \bx.
		\end{split}
	\end{equation}
	
	Now, write $\mathcal{W}_k = \{ Q_1, Q_2, \ldots, Q_{2^{kd}} \}$; we can adopt an enumeration such that $Q_i$ and $Q_{i+1}$ share a common face for all $i \in \{1, \ldots, 2^{kd} -1 \}$; denote $Q_i \cup Q_{i+1} = G_i$. Then by Minkowski's and H\"older's inequalities
	\begin{equation*}
		\begin{split}
			\vnorm{u- u_Q}_{L^p(Q)} 
			&\leq (1+|Q|^{-1/p}) \vnorm{u - u_{Q_1}}_{L^p(Q)} \\
			&\leq C(d,p) \sum_{j=1}^{2^{kd}} \left( \vnorm{u- u_{Q_j}}_{L^p(Q_j)} + |Q_j|^{1/p} |u_{Q_j} - u_{Q_1}| \right) \\
			&\leq C(d,p) \sum_{j=1}^{2^{kd}} \left( \vnorm{u- u_{Q_j}}_{L^p(Q_j)} + |Q_j|^{1/p} \sum_{i=1}^{j-1} |u_{Q_{i+1}} - u_{Q_i}| \right).
		\end{split}
	\end{equation*}
	By \eqref{eq:FractionalPoincare:Improved:Pf1} we need only estimate the terms in the second sum. For each $i$, we have by H\"older's inequality
	\begin{equation*}
		\begin{split}
			|u_{Q_{i+1}} - u_{Q_i}|
			&\leq |u_{Q_{i+1}} - u_{G_i}| + |u_{G_i} - u_{Q_i}| \\
			&\leq C(d) \fint_{G_i} |u(\bx) - u_{G_i}| \, \rmd \bx \leq \frac{C(d)}{|G_i|^{1/p}} \vnorm{u - u_{G_i}}_{L^p(G_i)}.
		\end{split}
	\end{equation*}
	Applying \eqref{eq:FractionalPoincare:Improved:Pf1} completes the proof.
\end{proof}

From this, one can follow the same method in the proof of \cite[Theorem 3.1]{hurri2013fractional}, using \Cref{lma:FractionalPoincare:Improved} where appropriate, to obtain the following Poincar\'e inequality on more general domains. The point of departure from the proof in that paper is the tracking of the scaling in $s$ on the Poincar\'e constant.

\begin{theorem}\label{thm:FractionalPoincare:Neumann}
	Let $1 \leq p < \infty$, let $s \in (0,1)$, and let $\cD \subset \bbR^d$ be a bounded Lipschitz domain. Then there exists a constant $C$ depending only on $d$, $p$, and the Lipschitz constant of $\cD$ such that
	\begin{equation*}
		\int_\cD |u(\bx) - u_\cD |^p \, \rmd \bx \leq C \diam(\cD)^{sp} [u]_{W^{s,p}(\cD)}^p, \qquad \forall u \in W^{s,p}(\cD),
	\end{equation*}
	where $u_\cD := \fint_\cD u(\bx) \, \rmd \bx$.
\end{theorem}

To obtain a Poincar\'e inequality for Sobolev functions that vanish on a portion of the boundary, we introduce the notion of a fractional relative capacity inspired by \cite{dyda2023fractional}. For a bounded Lipschitz domain $\cD$, we consider a closed set $\p \cD_D \subset \p \cD$ that satisfies $\scH^{d-1}(\p \cD_D) \in (0,\infty)$. Let $1 \leq p < \infty$, $0 < s < 1$, and let $A\subset \bbR^d $ be any set and  let $B = B(\bx_0,R) \subset \bbR^d$ be any Euclidean ball such that $\cD \cap B \subset A$. Then we define the capacity
\begin{equation*}
	\mathrm{cap}_{s,p}(\p \cD_D, B, A) := \inf \left\{ [\varphi]_{W^{s,p}(A)}^p \, : \, \varphi \in C^0(\overline{\cD}), \quad
	\begin{gathered}
		\varphi \geq 1 \text{ on } \p \cD_D \cap \overline{B}, \\
		\varphi = 0 \text{ on } \overline{\cD} \setminus 2B
	\end{gathered} 
	\right\}.
\end{equation*}
Although the techniques involving the capacity could give sharper results (such as the Poincar\'e inequalities for functions vanishing on sets of more general Hausdorff dimension) we choose not to pursue those topics here for the sake of brevity, and instead refer the interested reader to \cite{dyda2023fractional}.

Next we denote the outer Hausdorff measure of a set $A$ as
\begin{equation*}
    \scH^{d-1}_R(A) := \inf \left\{ \sum_{j=1}^\infty \frac{2^{1-d}  \pi^{(d-1)/2}}{ \Gamma(d/2 + 1)} \diam(V_j)^{d-1} \, : \, A \subset \bigcup_{j=1}^\infty V_j, \diam(V_j) \leq R \right\},
\end{equation*}
so that
\begin{equation*}
    \scH^{d-1}(A) = \lim\limits_{R \to 0} \scH^{d-1}_R(A) = \sup_{R > 0} \scH^{d-1}_R(A).
\end{equation*}

\begin{lemma}\label{lma:Capacity}
	Let $sp > 1$, and let $\bx_0 \in \p \cD_D$ and fix $R \in (0,\diam(\cD)/4)$. Then for $B=B(\bx_0,R)$
	\begin{equation*}
		\scH^{d-1}_{15R}(\p \cD_D) \leq C R^{sp-1} \mathrm{cap}_{s,p}(\p \cD_D,B,\cD \cap 3B),
	\end{equation*}
	where the constant $C$ depends only on $d$, $p$, and the Lipschitz constant of $\cD$.
\end{lemma}

\begin{proof}
	Let $\varphi \in C^0(\overline{\cD})$ with $\varphi = 1$ on $B \cap \p \cD_D$ and $\varphi = 0$ on $\overline{\cD} \setminus 2 B$; by considering $\max \{ 0, \min\{\varphi, 1\} \}$ we can assume without loss of generality that $0 \leq \varphi \leq 1$ on $\cD$. Define $R_0 = 3R$ and $B_0 := B(\bx_0,R_0) \cap \cD$. Fix $\by \in \p \cD_D$, and for $j \in \bbN$ define $R_j = 2^{1-j} R$ and $B_j := B(\by,R_j) \cap \cD$. Then since $\varphi$ is continuous, by Lebesgue differentiation we have
	\begin{equation*}
		\begin{split}
			|\varphi(\by) - \varphi_{B_0}| 
			&= \lim\limits_{k \to \infty} | \varphi_{B_k} - \varphi_{B_0}| \\
			&= \lim\limits_{k \to \infty} \left| \sum_{j=0}^{k-1} \varphi_{B_{j+1}} - \varphi_{B_j} \right| \\
			&\leq \sum_{j=0}^\infty |\varphi_{B_{j+1}} - \varphi_{B_j}| \\
			&\leq \sum_{j=0}^\infty \fint_{B_{j+1}} |\varphi(\bx) - \varphi_{B_j} | \, \rmd \bx \\
			&\leq C(d) \sum_{j=0}^\infty \left( \fint_{B_{j}} |\varphi(\bx) - \varphi_{B_j} |^p \, \rmd \bx \right)^{1/p}.
		\end{split}
	\end{equation*}
	Then by \Cref{thm:FractionalPoincare:Neumann}
	\begin{equation}\label{eq:CapacityLma:Pf1}
		|\varphi(\by) - \varphi_{B_0}| \leq C(d,p) \sum_{j=0}^\infty R_j^{s-d/p} [\varphi]_{W^{s,p}(B_j)}. 
	\end{equation}
	
	Now, since $\cD$ is a Lipschitz domain
	\begin{equation*}
		0 \leq \fint_{B_0} \varphi(\by) \, \rmd \by \leq \frac{1}{|B_0|} \int_{2B} \varphi(\by) \, \rmd \by \leq \frac{|2B|}{|B_0|} = \frac{|\cD \cap 2B|}{ |\cD \cap 3 B| }:= c < 1,
	\end{equation*}
	where $c$ depends only on $d$ and the Lipschitz constant of $\cD$, which we denote by $L$. Since $\by \in \p \cD_D$, we have
	\begin{equation*}
		|\varphi(\by) - \varphi_{B_0}| \geq 1 - c > 0.
	\end{equation*}
	Now, since $sp > 1$ we have by \eqref{eq:CapacityLma:Pf1}
	\begin{equation*}
		\sum_{j=0}^\infty 2^{-j(sp-1)/p} \leq C(p) \frac{ |\varphi(\by) - \varphi_{B_0}| }{1-c} \leq C(d,p,L) \sum_{j=0}^\infty r_j^{s-d/p} [\varphi]_{W^{s,p}(B_j)}.
	\end{equation*}
	Therefore, there exists $j$ depending on $\by$ such that
	\begin{equation*}
		2^{-j(sp-1)} \leq C(d,p,L)  R_j^{sp-d} [\varphi]_{W^{s,p}(B_j)}^p,
	\end{equation*}
	that is,
	\begin{equation*}
		R_j^{d-1} \leq C(d,p,L) R^{sp-1} [\varphi]_{W^{s,p}(B_j)}^p.
	\end{equation*}
	Relabeling $R_j = R_{\by}$ and $B_j = B_{\by}$, we have that
	\begin{equation*}
		R_{\by}^{d-1} \leq C(d,p,L) R^{sp-1} [\varphi]_{W^{s,p}(B_{\by})}^p.
	\end{equation*}
	The sets $\{ B_{\by} \}_{\by \in \p \cD_D}$ form an open cover of $\p \cD_D$.
	By the Vitali covering lemma, we obtain a countable collection $\{ \by_k \}_{k=1}^\infty$ with associated pairwise disjoint sets $B_{\by_k} \subset B_0$ such that $\p \cD_D \subset \cup_{k=1}^\infty 5 B_{\by_k}$. Therefore
	\begin{equation*}
		\begin{split}
			\scH^{d-1}_{15R}(\p \cD_D) \leq C(d) \sum_{k = 1}^\infty |B_{\by_k}|^{(d-1)/d} 
			&\leq C(d,p,L) R^{sp-1} \sum_{k=1}^\infty [\varphi]_{W^{s,p}(B_{\by_k})}^p \\
			&\leq C(d,p,L) R^{sp-1} [\varphi]_{W^{s,p}(B_0)}^p.
		\end{split}
	\end{equation*}
	The result follows by taking the infimum over $\varphi$ on the right-hand side of the inequality.
\end{proof}

\begin{proof}[Proof of \Cref{thm:Poincare:Fractional}]
	It suffices to prove the inequality for any nontrivial $u \in C^1_c(\overline{\cD} \setminus \p \cD_D)$.
	
	Step 1: Let us choose $R \in (0,\diam(\cD)/4)$ sufficiently small so that $\scH^{d-1}(\p \cD_D) \leq 2 \scH^{d-1}_{15R} (\p \cD_D)$. Let $\bx_0 \in \p \cD_D$. We first show that
	\begin{equation}\label{eq:FractionalPoincare:Final:Pf1}
		\vnorm{u}_{L^p(\cD)}^p   \mathrm{cap}_{s,p}(\p \cD_D,B,\cD) \leq C(d,p,\cD)[u]_{W^{s,p}(\cD)}^p,
	\end{equation}
	where $B = B(\bx_0,R)$. To this end, define $\eta(\bx) := \max \{0, 1- \dist(\bx,B)/R\}$. Then $\eta$ has Lipschitz constant $1/R$. Setting $\bar{u} = \left( \fint_\cD |u(\bx)|^p \, \rmd \bx \right)^{1/p}$, we define
	\begin{equation*}
		\varphi(\bx) := (1 - u(\bx)/\bar{u}) \eta(\bx);
	\end{equation*}
	then $\varphi \in C^0(\overline{\cD})$, $\varphi = 1$ on $\p \cD_D \cap \overline{B}$, and $\varphi = 0$ on $\overline{\cD} \setminus 2B$. Thus,
	\begin{equation*}
		\mathrm{cap}_{s,p}(\p \cD_D, B, \cD) \leq \frac{\kappa_{d,s,p}}{ \bar{u}^p } \int_\cD \int_\cD \frac{ |\eta(\bx) (\bar{u} - u(\bx)) - \eta(\by) (\bar{u} - u(\by))|^p  }{ |\bx-\by|^{d+sp} } \, \rmd \by \, \rmd \bx,
	\end{equation*}
	and therefore
	\begin{equation*}
		\begin{split}
			&\mathrm{cap}_{s,p}(\p \cD_D, B, \cD) \fint_\cD |u(\bx)|^p \, \rmd \bx \\
			\leq& C(p) \kappa_{d,s,p} \int_\cD \int_\cD \frac{ |\eta(\bx) -\eta(\by)|^p |u(\bx)-\bar{u}|^p }{ |\bx-\by|^{d+sp} } \, \rmd \by \, \rmd \bx \\
			&\quad + C(p) \kappa_{d,s,p} \int_\cD \int_\cD \frac{ |\eta(\by)|^p |u(\bx)-u(\by)|^p }{ |\bx-\by|^{d+sp} } \, \rmd \by \, \rmd \bx = I + II.
		\end{split}
	\end{equation*}
	Using the Lipschitz constant of $\eta$ and the dependence of $\kappa_{d,s,p}$ on $s$, we have
	\begin{equation*}
		\begin{split}
			I &\leq \frac{C(d,p)}{R^p} \int_{\cD} |u(\bx) - \bar{u}|^p \int_{3B} \frac{(1-s)}{|\by-\bx|^{d+sp-p}} \, \rmd \by \, \rmd \bx \\
			&\leq C(d,p) R^{-sp} \int_{\cD} |u(\bx) - \bar{u} |^p \, \rmd \bx.
		\end{split}
	\end{equation*}
	Since
	\begin{equation*}
		\begin{split}
			\vnorm{u-\bar{u}}_{L^p(\cD)} 
			&\leq \vnorm{u- u_\cD}_{L^p(\cD)} + |\cD|^{1/p} |u_{\cD} - \bar{u}| \\
			&= \vnorm{u- u_\cD}_{L^p(\cD)} + \big| \vnorm{u_{\cD}}_{L^p(\cD)} - \vnorm{u}_{L^p(\cD)} \big| \\
			&\leq 2 \vnorm{u- u_\cD}_{L^p(\cD)},
		\end{split}
	\end{equation*}
	it follows from \Cref{thm:FractionalPoincare:Neumann} that
	\begin{equation*}
		I \leq C(d,p,\cD) [u]_{W^{s,p}(\cD)}^p.
	\end{equation*}
	The same bound also holds for $II$ since $|\eta| \leq 1$, and so \eqref{eq:FractionalPoincare:Final:Pf1} is established.
	
	Step 2: we prove the theorem. By the monotonicity of the capacity, we have from \eqref{eq:FractionalPoincare:Final:Pf1}
	\begin{equation*}
		\vnorm{u}_{L^p(\cD)}^p \leq \frac{C(d,p,\cD)}{ \mathrm{cap}_{s,p}(\p \cD_D,B,\cD \cap 3B)} [u]_{W^{s,p}(\cD)}^p.
	\end{equation*}
	Then by \Cref{lma:Capacity} and our choice of $R$
	\begin{equation*}
		\vnorm{u}_{L^p(\cD)}^p \leq \frac{C(d,p,\cD) R^{sp-1}}{ \scH^{d-1}(\p \cD_D) } [u]_{W^{s,p}(\cD)}^p \leq C(d,p,\cD,\p \cD_D) [u]_{W^{s,p}(\cD)}^p.
	\end{equation*}
\end{proof}

Observe that the dependence of the constant on the size of $\cD$ can be sharpened if one considers local Poincar\'e inequalities on sets of the form $\cD \cap B$; see \cite{dyda2023fractional} for details.

\section{Weighted Sobolev spaces}

We begin first by stating some relations between weighted Sobolev spaces. For $s \in (0,1]$ and $1 < p < \infty$, define the space $\wt{W}^{1,p}(\cD;p-sp)$ as the class of all measurable functions $u : \cD \to \bbR^d$ such that its norm is finite, i.e.
	\begin{equation}
		\vnorm{u}_{\wt{W}^{1,p}(\cD;p-sp)}^p := \int_{\cD} \sigma(\bx)^{p-sp} |u(\bx)|^p \, \rmd \bx + \int_\cD \sigma(\bx)^{p-sp} |\grad u(\bx)|^p \, \rmd \bx < \infty.
	\end{equation}
Since $p - sp \geq 0$, it follows that $ W^{1,p}(\cD;p-sp) \subset \wt{W}^{1,p}(\cD;p-sp)$, with
	\begin{equation}\label{eq:appendix:weightedSobolevIneq}
		\vnorm{u}_{\wt{W}^{1,p}(\cD;p-sp)} \leq C(s,p,\cD) \vnorm{u}_{W^{1,p}(\cD;p-sp)}, \qquad \forall u \in W^{1,p}(\cD;p-sp).
	\end{equation}
As a consequence, we can obtain the following properties of functions in $W^{1,p}(\cD;p-sp)$:

\begin{theorem}\label{thm:FxnSpProp:Weighted}
	Let $1 < p < \infty$ and $s \in (0,1]$. Let $\cD$ be a bounded Lipschitz domain. The following hold:
	
	\begin{enumerate}[1)]
		\item $C^\infty(\overline{\cD})$ is dense in $W^{1,p}(\cD;p-sp)$.
		
		\item Let $sp > 1$. Let $T$ denote the trace operator on $\p \cD$. Then $T$ is a bounded linear operator from $W^{1,p}(\cD;p-sp)$ to $W^{s-1/p,p}(\p \cD)$.
		
		\item 
		Let $sp > 1$, and let $G \subset \p \cD$ be a $\scH^{d-1}$-measurable closed set that is diffeomorphic to a finite union of the closures of $W^{s-1/p,p}$-extension domains.  
		Then there exists a bounded linear extension operator $E : W^{s-1/p,p}(G) \to W^{1,p}(\cD;p-sp)$, whose operator norm does not exceed $C/s$ for a fixed constant $C = C(d,p,\cD,G) > 0$.
		
		\item Let a sequence $\{u_n\}_n \subset W^{1,p}(\cD;p-sp)$ satisfy $\sup_{n \in \bbN} \vnorm{u_n}_{W^{1,p}(\cD;p-sp)} \leq B$. Then $\{u_n\}_n$ is precompact in $L^p(\cD)$.
		
		\item Let $sp > 1$. For $G$ as in item 3), $u \in W^{1,p}_{0,G}(\cD;p-sp)$ if and only if $u \in W^{1,p}(\cD;p-sp)$ and $T u = 0$ $\scH^{d-1}$-a.e. on $G$.
		
		\item Let $sp > 1$ and let $\p \cD_D \subset \p \cD$ with $\scH^{d-1}(\p \cD_D) > 0$. Then there exists a constant $C > 0$ depending only on $d$, $p$, $\cD$, and $\p \cD_D$ such that
		\begin{equation*}
			\vnorm{u}_{L^p(\cD)} \leq C [u]_{W^{1,p}(\cD;p-sp)}, \qquad \forall u \in W^{1,p}_{0,\p \cD_D}(\cD;p-sp).
		\end{equation*}
	\end{enumerate}

\end{theorem}

\begin{proof}
	1) can be proved using a ``pull-and-convolve'' strategy, i.e. a translation and mollification approximation. The argument in \cite[Theorem 7.2]{kufner1980weighted}, used to prove density of $C^\infty(\overline{\cD})$ in $\wt{W}^{1,p}(\cD;p-sp)$, can be adapted to prove the same result for $W^{1,p}(\cD;p-sp)$.
	
	To prove item 2), first let $u \in C^\infty(\overline{\cD})$. Then \cite{nekvinda1993characterization} shows that
	\begin{equation*}
		\vnorm{T u}_{ W^{s-1/p,p}(\p \cD) }^p \leq C \int_{\cD} \eta(\bx)^{p-sp} |u(\bx)|^p \, \rmd \bx + C [u]_{ W^{1,p}(\cD;p-sp) }^p.
	\end{equation*}
	By \eqref{eq:appendix:weightedSobolevIneq} we can estimate 
	\begin{equation*}
		\vnorm{T u}_{ W^{s-1/p,p}(\p \cD) }^p \leq C \int_{\cD} |u(\bx)|^p \, \rmd \bx + C [u]_{ W^{1,p}(\cD;p-sp) }^p.
	\end{equation*}
	Item 2) then follows by using item 1).
	
	To show item 3), we note that there exists a bounded linear extension operator $E_1 : W^{s-1/p,p}(G) \to W^{s-1/p,p}(\p \cD)$, with operator norm bounded from above by $C(d,p,\cD,G)/s$; see for example \cite{Nezza2012Hitchhikers} for extending functions in fractional Sobolev spaces.
    Next, one can use the argument of \cite[Lemma 3.1]{nekvinda1993characterization} to obtain a bounded linear extension operator $E_2 : W^{s-1/p,p}(\bbR^{d-1}) \to W^{1,p}(\bbR^{d-1} \times (0,\infty) ; p-sp)$, with operator norm bounded from above by $\frac{C(d,p)}{(p-sp+1)(d+sp-2)}$. With the operator $E_2$, an argument involving the flattening of the boundary and a partition of unity (see for instance the proof of \cite[Theorem 2.10]{kim2007trace}) leads to the existence of a bounded linear extension operator $E_3 : W^{s-1/p,p}(\p \cD)\to W^{1,p}(\cD;p-sp) $ whose operator norm does not exceed $\frac{C(d,p,\cD)}{(p-sp+1)(d+sp-2)}$.
    Then $E=E_3 \circ E_1$ is the desired extension operator.

	To see item 4), we use \Cref{thm:WeightedEst} to get $\sup_n \vnorm{u_n}_{W^{s,p}(\cD)} < \infty$. 
	The result then follows from \cite[Theorem 7.1]{Nezza2012Hitchhikers}.
	
	The proof of item 5) uses a strategy similar to that of item 3). The forward implication is straightforward; the reverse implication in the case that $G$ is a half-space is contained in \cite[Section 2.9.2, Theorem 1, b)]{Triebel1995Interpolation}. The proof in the case of general domain $G$ follows from this case via a flattening of the boundary and a partition of unity argument.
    
	Item 6) follows from \Cref{thm:Poincare:Fractional} and \Cref{thm:WeightedEst}.
	
\end{proof}

\end{document}